\documentclass[11pt]{amsart}
\usepackage[dvipsnames,usenames]{color}
\usepackage{hyperref}
\usepackage{comment}
\usepackage{graphicx}
\usepackage{epsfig}
\usepackage{transparent}
\usepackage[latin1]{inputenc}
\usepackage{amsmath}
\usepackage{amsfonts}
\usepackage{amssymb}
\usepackage{mathrsfs}
\usepackage{mathtools}
\usepackage{xypic}
\usepackage{amsthm}
\usepackage{amscd}
\usepackage{verbatim}
\usepackage{caption}
\usepackage{subcaption}
\usepackage{pinlabel}
\usepackage{stmaryrd}
\usepackage{enumerate, enumitem}
\usepackage{todonotes}
\usepackage{bm}
\usepackage{thmtools}
\usepackage{thm-restate}
\usepackage{stackengine}

\usepackage{array}   
\newcolumntype{L}{>{$}l<{$}}
\newcolumntype{C}{>{$}c<{$}}

\usepackage{pgf}
\usepackage{tikz}
\usetikzlibrary{arrows,positioning,fit,calc,shapes,math}
\usetikzlibrary{decorations.pathreplacing}
\usepackage{verbatim}
\usetikzlibrary{cd}

\newcommand{\disknw}{120}
\newcommand{\diskne}{60}
\newcommand{\disksw}{240}
\newcommand{\diskse}{300}

\newcommand{\vres}[1][scale=.35]{
\begin{tikzpicture}[#1]
\draw[thick] (0,0) to[out=-70,in=70] (0,-1);
\draw[thick] (1,0) to[out=250,in=110] (1,-1);
\end{tikzpicture}
}

\newcommand{\hres}[1][scale=.35]{
\begin{tikzpicture}[#1]
\draw[thick] (0,0) to[out=-70,in=250] (1,0);
\draw[thick] (0,-1) to[out=70,in=110] (1,-1);
\end{tikzpicture}
}

\newcommand{\modbox}[5][]{ 
\node[align=center](#2) at (#3,#4) {$#5$};
\node[align=center,draw,rectangle,fit=(#2),#1](#2 box) {};
\path (#2 box.north west) -- (#2 box.north east) node[pos=.2, inner sep=0, outer sep=0](#2 box nw){} node[pos=.8, inner sep=0, outer sep=0](#2 box ne){};
\path (#2 box.south west) -- (#2 box.south east) node[pos=.2, inner sep=0, outer sep=0](#2 box sw){} node[pos=.8, inner sep=0, outer sep=0](#2 box se){};
}

\newcommand{\strandsboxtodisk}[2]{ 
\foreach \stpt/\stdir/\endpt/\enddir in
    {#1 box.\disknw/90/#2 disk.\disknw/\diskse,
    #1 box.\diskne/90/#2 disk.\diskne/\disksw,
    #1 box.\disksw/-90/#2 disk.\disksw/\diskne,
    #1 box.\diskse/-90/#2 disk.\diskse/\disknw}
{
\draw[thick] (\stpt) to[out=\stdir,in=\enddir] (\endpt);
}
}

\newcommand{\modboxindisk}[6][1]{ 
\node[align=center](#2) at (#3,#4) {$#5$};
\node[align=center,draw,rectangle,fit=(#2),xscale=#1](#2 box) {};
\node[align=center,draw,circle,dashed,fit=(#2 box),scale=1.1](#2 disk){};
    \foreach \stpt/\stdir/\endpt/\enddir in
            {#2 box.\disknw/90/#2 disk.\disknw/\diskse,
            #2 box.\diskne/90/#2 disk.\diskne/\disksw,
            #2 box.\disksw/-90/#2 disk.\disksw/\diskne,
            #2 box.\diskse/-90/#2 disk.\diskse/\disknw}
        {\draw[thick] (\stpt) to[out=\stdir,in=\enddir] (\endpt);}
\nstrandsalongpath[#6]{#2 box.north}{#2 disk.north}
\nstrandsalongpath[]{#2 box.south}{#2 disk.south}
}

\newcommand{\straightmodboxindisk}[5][]{ 
\node[align=center](#2) at (#3,#4) {$#5$};
\node[align=center,draw,rectangle,fit=(#2),inner sep=1.2,#1](#2 box) {};
\node[align=center,draw,circle,dashed,fit=(#2 box),inner sep=1](#2 disk){};
\foreach \ang/\ns in {\disknw/north,\diskne/north,\disksw/south,\diskse/south} {
    \draw[thick] (#2 disk.\ang)--(#2 disk.\ang|-#2 box.\ns);
    }
\nstrandsalongpath[]{#2 box.north}{#2 disk.north}
}

\newcommand{\emptydisk}[4][1]{ 
\node (#2) at (#3,#4) {};
\node [circle,draw,dashed,fit=(#2),inner sep = 5,scale=#1] (#2 disk) {};
}

\newcommand{\nstrands}[1][n]{ 
\stackanchor[1pt]{$#1$}{$\cdots$}
}

\newcommand{\nstrandsalongpath}[3][n]{ 
\path (#2)--(#3) node[pos=.5,align=center,scale=.7] {\nstrands[#1]};
}

\newcommand{\drawover}[2][thick]{
\draw[line width=2mm,white] #2
\draw[#1] #2
}

\newcommand{\negcros}[5][thick]{
\draw[#1] (#4,#3) to[out=-90,in=90] (#2,#5);
\drawover[#1]{(#2,#3) to[out=-90,in=90] (#4,#5);}
}

\newcommand{\negtwist}[5][thick]{
\tikzmath{
real \xstep,\ystep,\leftx \topy,\rightx,\boty;
\xstep=(#4-#2)/4;
\ystep=(#5-#3)/4;
}

\foreach \i in 
        {0,3}
    {
    \tikzmath{
        \leftx=#2 + (\i*\xstep);
        \topy=#3 + (\i*\ystep);
        \rightx=#2 + ((\i+1)*\xstep);
        \boty=#3 + ((\i+1)*\ystep);
        }
    \negcros[#1]{\leftx}{\topy}{\rightx}{\boty}
    \draw[#1] (\leftx,\boty)--(\leftx,#5);
    \draw[#1] (\rightx,\topy)--(\rightx,#3);
    }
\node at ($.5*(#2,#3)+.5*(#4,#5)+(0,.1)$) {$\ddots$};
\tikzmath{
    \leftx=#2+\xstep;
    \rightx=#4;
    }
\nstrandsalongpath[n-1]{\leftx,#3}{\rightx,#3}
\tikzmath{
    \leftx=#2;
    \rightx=#4-\xstep;
    }
\nstrandsalongpath[]{\leftx,#5}{\rightx,#5}
}

\newcommand{\FTfourEX}{
\tikzmath{
real \xstep,\ystep,\leftx \topy,\rightx,\boty;
\xstep=1;
\ystep=-1;
}

\foreach \i/\j in 
        {0/0,1/1,2/2,0/2,1/3,2/4,0/4,1/5,2/6,0/6,1/7,2/8}
    {
    \tikzmath{
        \leftx=\i*\xstep;
        \topy=\j*\ystep;
        \rightx=(\i+1)*\xstep;
        \boty=(\j+1)*\ystep;
        }
    \negcros{\leftx}{\topy}{\rightx}{\boty}
    }
\foreach \i/\j in {0/-1,0/-3,0/-5,0/-7,0/-8,1/-8,2/0,3/0,3/-1,3/-3,3/-5,3/-7}
    {\draw[thick] (\i,\j) -- (\i,\j-1);}
}

\newcommand{\whitebox}[3][]{ 
\draw[fill=white,#1] (#2) rectangle (#3);
}

\newcommand{\CKcolumnONE}{
    \vcenter{\hbox{\begin{tikzpicture}[xscale=.25,yscale=.4]
        \draw[thick] (0,2)--(0,1) to[out=-90,in=90] (2,0)
                (2,2)--(2,1) to[out=-90,in=90] (4,0)
                (3,2)--(3,1) to[out=-90,in=-90] (4,1)--(4,2)
                (0,0) to[out=90,in=90] (1,0);
        \whitebox{-.2,1}{3.2,1.5}
        \nstrandsalongpath[]{0,1.9}{2,1.9}
        \nstrandsalongpath[]{2,.2}{4,.2}
    \end{tikzpicture}}}
}

\newcommand{\CKcolumnTWO}{
    \vcenter{\hbox{\begin{tikzpicture}[xscale=.25,yscale=.4]
        \draw[thick] (0,2)--(0,0)
                (1,2)--(1,1) to[out=-90,in=90] (3,0)
                (3,2)--(3,1) to[out=-90,in=90] (5,0)
                (4,2)--(4,1) to[out=-90,in=-90] (5,1)--(5,2)
                (1,0) to[out=90,in=90] (2,0);
        \whitebox{-.2,1}{4.2,1.5}
        \nstrandsalongpath[]{1,1.9}{3,1.9}
        \nstrandsalongpath[]{3,.2}{5,.2}
    \end{tikzpicture}}}
}

\newcommand{\CKcolumnFOUR}{
\vcenter{\hbox{\begin{tikzpicture}[xscale=.25,yscale=.4]
        \draw[thick] (0,2)--(0,0)
                (2,2)--(2,0)
                (3,2)--(3,1) to[out=-90,in=90] (5,0)
                (4,2)--(4,1) to[out=-90,in=-90] (5,1)--(5,2)
                (3,0) to[out=90,in=90] (4,0);
        \whitebox{-.2,1}{4.2,1.5}
        \nstrandsalongpath[]{0,1.9}{2,1.9}
        \nstrandsalongpath[]{0,.2}{2,.2}
    \end{tikzpicture}}}
}

\newcommand{\CKcolumnFIVE}{
    \vcenter{\hbox{\begin{tikzpicture}[xscale=.25,yscale=.4]
        \draw[thick] (0,2)--(0,0)
                (2,2)--(2,0)
                (3,2)--(3,1) to[out=-90,in=-90] (4,1)--(4,2)
                (3,0) to[out=90,in=90] (4,0);
        \whitebox{-.2,1}{3.2,1.5}
        \nstrandsalongpath[]{0,1.9}{2,1.9}
        \nstrandsalongpath[]{0,.2}{2,.2}
    \end{tikzpicture}}}
}

\newcommand{\LiftDiagBL}{
    \vcenter{\hbox{\begin{tikzpicture}[xscale=.25,yscale=.8]
        \draw[thick] (0,2)--(0,0)
                (2,2)--(2,0)
                (3,2)--(3,1) to[out=-90,in=-90] (4,1)--(4,2)
                (3,0) to[out=90,in=90] (4,0);
        \whitebox{-.2,1}{3.2,1.5}
        \node[scale=.7] at (1.5,1.25) {$\cP_{n-1}$};
        \nstrandsalongpath[]{0,1.9}{2,1.9}
        \nstrandsalongpath[]{0,.2}{2,.2}
    \end{tikzpicture}}}
}

\newcommand{\LiftDiagBOT}{
\vcenter{\hbox{\begin{tikzpicture}[xscale=.25,yscale=.8]
        \draw[thick] (4,2)--(4,1.2) to[out=-90,in=90,looseness=.2] (-.5,.7);
        \drawover{(0,2)--(0,0);}
        \drawover{(2,2)--(2,0);}
        \drawover{(3,2)--(3,.7) to[out=-90,in=-90,looseness=.2] (-.5,.7);}
        \draw[thick] (3,0) to[out=90,in=90] (4,0);
        \whitebox{-.2,1.2}{4.2,1.7}
        \node[scale=.7] at (2,1.45) {$\augJM_n^{k-1}$};
        \nstrandsalongpath[]{0,1.9}{2,1.9}
        \nstrandsalongpath[]{0,.2}{2,.2}
    \end{tikzpicture}}}
}

\newcommand{\LiftDiagBR}{
\vcenter{\hbox{\begin{tikzpicture}[xscale=.25,yscale=.8]
        \draw[thick] (4,2)--(4,1.2) to[out=-90,in=90,looseness=.2] (-.5,.7);
        \drawover{(0,2)--(0,0);}
        \drawover{(2,2)--(2,0);}
        \drawover{(3,2)--(3,.7) to[out=-90,in=90,looseness=.4] (4,0);}
        \drawover{(3,0) to[out=90,in=-90,looseness=.4] (-.5,.7);}
        \whitebox{-.2,1.2}{4.2,1.7}
        \node[scale=.7] at (2,1.45) {$\augJM_n^{k-1}$};
        \nstrandsalongpath[]{0,1.9}{2,1.9}
        \nstrandsalongpath[]{0,.2}{2,.2}
    \end{tikzpicture}}}
}

\newcommand{\TLiftDiagBL}{
    \vcenter{\hbox{\begin{tikzpicture}[xscale=.25,yscale=.8]
        \draw[thick] (0,2)--(0,0)
                (2,2)--(2,0)
                (3,2)--(3,0);
        \whitebox{-.2,1}{3.2,1.5}
        \node[scale=.7] at (1.5,1.25) {$\cP_{n-1}$};
        \nstrandsalongpath[]{0,1.9}{2,1.9}
        \nstrandsalongpath[]{0,.2}{2,.2}
    \end{tikzpicture}}}
}

\newcommand{\TLiftDiagBOT}{
\vcenter{\hbox{\begin{tikzpicture}[xscale=.25,yscale=.8]
        \draw[thick] (4,1.2) to[out=-90,in=90,looseness=.2] (-.5,.7);
        \drawover{(0,2)--(0,0);}
        \drawover{(2,2)--(2,0);}
        \drawover{(3,2)--(3,.7) to[out=-90,in=-90,looseness=.2] (-.5,.7);}
        \draw[thick] (3,0) to[out=90,in=-90] (5,1)--(5,1.7) to[out=90,in=90] (4,1.7);
        \whitebox{-.2,1.2}{4.2,1.7}
        \node[scale=.7] at (2,1.45) {$\augJM_n^{k-1}$};
        \nstrandsalongpath[]{0,1.9}{2,1.9}
        \nstrandsalongpath[]{0,.2}{2,.2}
    \end{tikzpicture}}}
}

\newcommand{\TLiftDiagBR}{
\vcenter{\hbox{\begin{tikzpicture}[xscale=.25,yscale=.8]
        \draw[thick] (4,1.2) to[out=-90,in=90,looseness=.2] (-.5,.7);
        \drawover{(0,2)--(0,0);}
        \drawover{(2,2)--(2,0);}
        \drawover{(3,2)--(3,.7) to[out=-90,in=-90,looseness=.5] (5,.7)--(5,1.7) to[out=90,in=90] (4,1.7);}
        \drawover{(3,0) to[out=90,in=-90,looseness=.4] (-.5,.7);}
        \whitebox{-.2,1.2}{4.2,1.7}
        \node[scale=.7] at (2,1.45) {$\augJM_n^{k-1}$};
        \nstrandsalongpath[]{0,1.9}{2,1.9}
        \nstrandsalongpath[]{0,.2}{2,.2}
    \end{tikzpicture}}}
}

\newcommand{\ILtikzpic}[2][]{
\vcenter{\hbox{\begin{tikzpicture}[#1]
#2
\end{tikzpicture}}}
}

\newcommand{\ILvres}[1][scale=.35]{
\vcenter{\hbox{   
\begin{tikzpicture}[#1]
\draw[thick] (0,0) to[out=-70,in=70] (0,-1);
\draw[thick] (1,0) to[out=250,in=110] (1,-1);
\end{tikzpicture}
}}
}

\newcommand{\ILfres}[1][xscale=.25,yscale=.4]{
\vcenter{\hbox{
\begin{tikzpicture}[#1]
\draw[thick] (.2,2.5) to[out=-70,in=250] (1.8,2.5);
    \draw[thick] (-.5,.8) to[out=90,in=90,looseness=.3] (2.5,.8);
   \drawover{(.2,0)--(.2,1) to[out=90,in=90] (1.8,1) --(1.8,0);}
    \drawover{(-.5,.8) to[out=-90,in=-90,looseness=.3] (2.5,.8);}
\end{tikzpicture}
}}}

\newcommand{\ILdisres}[1][xscale=.25,yscale=.4]{
\vcenter{\hbox{
\begin{tikzpicture}[#1]
\draw[thick] (.2,2.5) to[out=-70,in=250] (1.8,2.5);
    \draw[thick] (.2,0)--(.2,1) to[out=90,in=90] (1.8,1) --(1.8,0);
    \draw[thick] (-.3,1) to[out=90,in=90,looseness=.8] (-1.7,1) to[out=-90,in=-90,looseness=.8] (-.3,1);
\end{tikzpicture}
}}}

\newcommand{\ILkres}[1][scale=.2]{
\vcenter{\hbox{
        \begin{tikzpicture}
        \emptydisk{M1}{0}{0}
        \node[fit={(M1 disk)}, circle, dashed, draw,yscale=1.3,xscale=1] (TM1) {};
        \draw[thick] (TM1.105) to[out=-90,in=-90,looseness=1.75] (TM1.75);
        \draw[thick] (M1 disk.\disksw) -- (TM1.255);
        \draw[thick] (M1 disk.\diskse) -- (TM1.285);
        \end{tikzpicture}}}
}

\newcommand{\ILhres}[1][scale=.35]{
\vcenter{\hbox{
\begin{tikzpicture}[#1]
\draw[thick] (0,0) to[out=-70,in=250] (1,0);
\draw[thick] (0,-1) to[out=70,in=110] (1,-1);
\end{tikzpicture}
}}
}

\newcommand{\ILhreshres}[1][scale=.35]{
\ILtikzpic[#1]{
    \draw[thick] (0,0) to[out=-70,in=250] (1,0);
    \draw[thick] (0,-2) to[out=70,in=110] (1,-2);
    \draw[thick] (.5,-1) circle (.4);
    }
}

\newcommand{\ILhresDotT}{
\ILtikzpic[scale=.35]{
    \draw[thick] (0,0) to[out=-70,in=250] node[pos=.3,fill=black,circle,scale=.4]{} (1,0);
    \draw[thick] (0,-1) to[out=70,in=110] (1,-1);
    }
}

\newcommand{\ILhresDotB}{
\ILtikzpic[scale=.35]{
    \draw[thick] (0,0) to[out=-70,in=250] (1,0);
    \draw[thick] (0,-1) to[out=70,in=110] node[pos=.7,fill=black,circle,scale=.4]{} (1,-1);
    }
}

\newcommand{\ILrhotop}[1][scale=.35]{
\ILtikzpic[#1]{
    \draw[thick] (-.5,1) to[out=-70,in=250] node[pos=.3,fill=black,circle,scale=.4]{} (.5,1);
    \draw[thick] (-.5,-1) to[out=70,in=110] (.5,-1);
    \draw[thick] (0,0) circle (.4);
    \foreach \i in {45,135,225,315} {
        \draw (\i:.4)--(\i:.55);
        }
    }
}

\newcommand{\ILrhobot}[1][scale=.35]{
\ILtikzpic[#1]{
    \draw[thick] (-.5,1) to[out=-70,in=250] (.5,1);
    \draw[thick] (-.5,-1) to[out=70,in=110] (.5,-1);
    \draw[thick] (0,0) circle (.4);
    \node[fill=black,circle,scale=.4] at (.4,0){};
    \foreach \i in {45,135,225,315} {
        \draw (\i:.4)--(\i:.55);
        }
    }
}

\newcommand{\ILhreshrestop}[1][scale=.35]{
\ILtikzpic[#1]{
    \draw[thick] (-.5,1) to[out=-70,in=250] node[pos=.3,fill=black,circle,scale=.4]{} (.5,1);
    \draw[thick] (-.5,-1) to[out=70,in=110] (.5,-1);
    \draw[thick] (0,0) circle (.4);
    }
}

\newcommand{\ILhreshresmid}[1][scale=.35]{
\ILtikzpic[#1]{
    \draw[thick] (-.5,1) to[out=-70,in=250] (.5,1);
    \draw[thick] (-.5,-1) to[out=70,in=110] (.5,-1);
    \draw[thick] (0,0) circle (.4);
    \node[fill=black,circle,scale=.4] at (.4,0){};
    }
}

\newcommand{\ILhreshressad}[1][scale=.35]{
\ILtikzpic[#1]{
    \draw[thick] (0,0) to[out=-70,in=250] node[pos=.5,inner sep=0](pt1){} (1,0);
    \draw[thick] (0,-2) to[out=70,in=110] (1,-2);
    \draw[thick] (.5,-1) circle (.3);
    \node[inner sep=0](pt2) at (.5,-.7){};
    \draw[thick,red] (pt1)--(pt2);
    }
}

\newcommand{\ILnegcros}{
\,
    \vcenter{\hbox{\begin{tikzpicture}[scale=.35]
    \draw[thick] (1,0) to[out=-90,in=90] (0,-1);
    \node[fill=white,circle,scale=.8] at (.5,-.5){};
    \draw[thick](0,0) to[out=-90,in=90] (1,-1);
    \end{tikzpicture}}}
\,
}

\newcommand{\ILmodbox}[2][]{  
\begin{tikzpicture}[baseline=-.5ex]
\straightmodboxindisk[#1]{mod}{0}{0}{#2}
\end{tikzpicture}
}

\newcommand{\ILbirth}{
\begin{tikzpicture}[baseline=-.5ex,scale=.15]
\draw (0,0) ellipse (.4 and 1);
\draw (0,1) to[out=180,in=90] (-1.75,0) to[out=-90,in=180] (0,-1);
\end{tikzpicture}
}

\newcommand{\ILhorizsaddle}{
\begin{tikzpicture}[baseline=-.5ex,scale=.3]
\draw[thick] (0,.6) to[out=-70,in=70] node[pos=.5,inner sep=0](pt1){} (0,-.6) ;
\draw[thick] (1,.6) to[out=250,in=110] node[pos=.5,inner sep=0](pt2){} (1,-.6) ;
\draw[thick,red] (pt1)--(pt2);
\end{tikzpicture}
}

\newcommand{\ILvertsaddle}[1][scale=.35]{
\vcenter{\hbox{
\begin{tikzpicture}[#1]
\draw[thick] (0,0) to[out=-70,in=250] node[pos=.5,inner sep=0](pt1){} (1,0);
\draw[thick] (0,-1) to[out=70,in=110] node[pos=.5,inner sep=0](pt2){} (1,-1);
\draw[thick,red] (pt1)--(pt2);
\end{tikzpicture}
}}
}

\newcommand{\MW}[1]
{\todo[color=cyan!30,linecolor=cyan!40!black,size=\small]{MW: #1}}

\newcommand{\MWinline}[1]
{\todo[inline,color=cyan!30,linecolor=cyan!40!black,size=\normalsize]{MW: #1}}

\newcommand{\N}{\mathbb{N}}
\newcommand{\Z}{\mathbb{Z}}

\newcommand{\Sp}{\mathtt{Sp}}
\newcommand{\gSp}{\mathtt{g}\Sp}
\newcommand{\Tang}{\mathbb{T}}
\newcommand{\thTang}{\widetilde{\Tang}}
\newcommand{\SpmMod}{\Sp\mathrm{MultiM}}

\newcommand{\Burn}{\mathrm{Burn}}
\newcommand{\gBurn}{\mathtt{g}\Burn}
\newcommand{\divcob}{\mathrm{Cob}_d}
\newcommand{\TL}{\mathcal{TL}}
\newcommand{\Rmod}[1][R]{#1-\mathrm{Mod}}
\newcommand{\Hnmod}{\Rmod[H_n]}
\newcommand{\Kom}{\mathrm{Kom}}

\newcommand{\sarc}{\mathcal{H}} 
\newcommand{\SpArcAlg}[1][2n]{\sarc_{#1}} 

\newcommand{\Kh}{\mathit{Kh}} 
 
\newcommand{\X}{\mathscr{X}} 
\newcommand{\vertcomp}{\stackrel{v}{\otimes}}
\newcommand{\horizcomp}{\sqcup}
\newcommand{\q}{\mathrm{q}} 
\newcommand{\atq}[2][\Big]{#1\rvert_{q=#2}}
\newcommand{\hocolim}{\mathrm{hocolim}}
\newcommand{\holim}{\mathrm{holim}}
\newcommand{\THH}{\mathrm{THH}}
\newcommand{\cube}{\underline{2}}
\newcommand{\basedcube}{\cube_+}

\newcommand{\chainsfunc}{\mathcal{C}_h}

\newcommand{\T}{\mathcal{T}} 
\newcommand{\FT}{\mathcal{FT}} 
\newcommand{\TWfilt}{\mathscr{F}_{tw}}
\newcommand{\simpT}{\widetilde{\T}} 

\newcommand{\JM}{\mathcal{J}}
\newcommand{\augJM}{\widetilde{\JM}}  
\newcommand{\augJMcone}{\widetilde{K}}  
\newcommand{\closure}[1]{\overline{#1}}  
\newcommand{\CK}[1][n]{\mathcal{CK}_{#1}}  
\newcommand{\CKch}[1][n]{\mathrm{CK}_{#1}}  
\newcommand{\CKfilt}{\mathscr{F}_{\CK[]}}
\newcommand{\CKchfilt}{\mathscr{F}_{\CKch[]}}

\newcommand{\eitop}[1][i]{e_{#1}^{\text{top}}}
\newcommand{\eibot}[1][i]{e_{#1}^{\text{bot}}}

\newcommand{\proj}{\cP}
\newcommand{\projmap}{\mathcal{U}}

\newcommand{\ssimpproj}{\proj^{\mathrm{ssimp}}}
\newcommand{\ssimpBar}{\mathcal{B}^\mathrm{ssimp}}

\newcommand{\Itang}{\mathbb{I}}
\newcommand{\Imod}{\mathcal{I}}
\newcommand{\Imap}{\mathrm{id}}
\newcommand{\sphere}{\mathbb{S}}

\DeclarePairedDelimiter{\ceil}{\lceil}{\rceil}

\newcommand{\bL}{\mathbb{L}}

\newcommand{\bR}{\mathbb{R}}
\newcommand{\bS}{\mathbb{S}}

\newcommand{\cA}{\mathcal{A}}
\newcommand{\cB}{\mathcal{B}}
\newcommand{\cC}{\mathcal{C}}

\newcommand{\cF}{\mathcal{F}}

\newcommand{\cI}{\mathcal{I}}

\newcommand{\cM}{\mathcal{M}}
\newcommand{\cN}{\mathcal{N}}

\newcommand{\cP}{\mathcal{P}}

\newcommand{\cT}{\mathcal{T}}

\newcommand{\Id}{\mathrm{Id}}
\newcommand{\id}{\mathrm{id}}
\newcommand{\HOM}{\mathrm{HOM}}
\newcommand{\Hom}{\mathrm{Hom}}
\newcommand{\End}{\mathrm{End}}
\newcommand{\Cone}{\mathrm{Cone}}
\newcommand{\Ob}{\mathrm{Ob}}

\newcommand{\Ccone}{\mathrm{CCone}}

\newtheorem{theorem}{Theorem}[section]
\newtheorem{Theorem}{Theorem}
\newtheorem{lemma}[theorem]{Lemma}
\newtheorem{proposition}[theorem]{Proposition}

\newtheorem{corollary}[theorem]{Corollary}

\newtheorem{Conjecture}[Theorem]{Conjecture}

\theoremstyle{definition}
\newtheorem{definition}[theorem]{Definition}

\theoremstyle{remark}
\newtheorem{remark}[theorem]{Remark}

\usepackage{mathrsfs}

\newcommand{\Kc}{\mathit{Kc}}
\newcommand{\st}{\mathscr{T}}
\newcommand{\thst}{\widetilde{\mathscr{T}}}

\newsavebox\mybox
\newsavebox\myotherbox
\newsavebox\mythirdbox

\numberwithin{equation}{section}

\newcommand\rightthreearrow{%
        \mathrel{\vcenter{\mathsurround0pt
                \ialign{##\crcr
                        \noalign{\nointerlineskip}$\rightarrow$\crcr
                        \noalign{\nointerlineskip}$\rightarrow$\crcr
                        \noalign{\nointerlineskip}$\rightarrow$\crcr
                }%
        }}%
}
\newcommand\rightfourarrow{%
        \mathrel{\vcenter{\mathsurround0pt
                \ialign{##\crcr
                        \noalign{\nointerlineskip}$\rightarrow$\crcr
                        \noalign{\nointerlineskip}$\rightarrow$\crcr
                        \noalign{\nointerlineskip}$\rightarrow$\crcr
                        \noalign{\nointerlineskip}$\rightarrow$\crcr
                }%
        }}%
}

\title[Infinite Twist]{Jones-Wenzl projectors and the Khovanov homotopy of the infinite twist}
\author[M. Stoffregen]{Matthew Stoffregen}
\address {Department of Mathematics, Michigan State University, East Lansing, MI 48824}
\email{stoffre1@msu.edu}

\author[M. Willis]{Michael Willis}
\address {Department of Mathematics, Texas A\&M University, College Station, TX  77845}
\email{msw188@tamu.edu}

\begin{document}
	
	\begin{abstract}
		We construct and study a lift of Jones-Wenzl projectors to the setting of Khovanov spectra, and provide a realization of such lifted projectors via a Cooper-Krushkal-like sequence of maps.  We also give a polynomial action on the 3-strand spectral projector allowing a complete computation of the $3$-colored Khovanov spectrum of the unknot, proving a conjecture of Lobb-Orson-Sch\"{u}tz.  As a byproduct, we disprove a conjecture of Lawson-Lipshitz-Sarkar on the topological Hochschild homology of tangle spectra.
	\end{abstract}
	
	\maketitle
	\tableofcontents

\section{Introduction}
\label{sec:intro}
The Temperley-Lieb algebra $TL_n$ on $n\geq 0$ strands \cite{TLalg}, and its categorification $\TL_n:=\Kom(H_{2n}-\mathrm{pmod})$,\footnote{We follow \cite[Definition 2.7]{Cooper-Krushkal}, where the category $\mathrm{Cob}(n)$ corresponds to the category $H_{2n}-\mathrm{pmod}$ of projective modules over Khovanov's arc algebra $H_{2n}$ \cite{Kh-ring_H}, and $\Kom(\cdot)$ indicates the category of complexes which are bounded below in homological degree and of finite rank in each quantum degree.} are large and important areas of study in low dimensional topology with connections to higher representation theory, quantum physics, and more.  To an $(n,n)$-tangle $T$ one may associate the Kauffman bracket $\langle T\rangle\in TL_n$ \cite{Kauff}, which is categorified by the Khovanov complex $\Kc(T)\in\TL_n$ \cite{Kh-ring_H}.  In the case that $n=0$, a $(0,0)$-tangle is just a link $L$, with $\langle L \rangle$ recovering the celebrated Jones polynomial \cite{JonesPoly}.  The categorification is then the usual Khovanov complex $\Kc(L)$ \cite{Khovanov-homology}, whose topological applications include a combinatorial proof of the Milnor conjecture on genus of torus knots \cite{Ras} and, more recently, the proof that the Conway knot is not slice \cite{LPConway}.

In \cite{LLS_tangles}, Lawson-Lipshitz-Sarkar provide a further lift of these invariants to the setting of spectra, defining the spectral arc algebras $\sarc_{2n}$ (for $n\geq 0$), and associating to each $(n,n)$-tangle $T$ a spectral $\sarc_{2n}$-module $\X(T)$, the \emph{Khovanov spectrum} of $T$, such that the reduced chain complex of $\X(T)$ is quasi-isomorphic to $\Kc(T)$.  It is then natural to ask what structural properties of $\Kc(T)$ lift to similar properties of $\X(T)$.  Many properties have been lifted successfully, and in this paper we continue this program, but also find some properties that \emph{fail} to lift.

Our first result concerns the Jones-Wenzl projector $p_n\in TL_n$ \cite{Wenzl}, a special idempotent element used to define the colored Jones polynomial and quantum invariants of 3-manifolds \cite{KL,ReshTuraev}.  Cooper-Krushkal \cite{Cooper-Krushkal} showed that there is a categorical idempotent $P_n\in \TL_n$, the \emph{categorified Jones-Wenzl projector}, whose decategorification is the Jones-Wenzl projector $p_n$.  Rozansky \cite{Rozansky} showed that the categorified Jones-Wenzl projector $P_n\in\TL_n$ can be described via $n$-strand torus braids $T_n^k$ as
\begin{equation}\label{intro eq: Pn is inf torus braid}
	P_n \simeq \Kc(T_n^\infty) := \lim_{k\rightarrow\infty} \Kc(T_n^k).
\end{equation}
A remarkable feature of the Cooper-Krushkal projector is its connection to the Khovanov homology of torus links $T(n,k)$ (see \cite{Rozansky}):
\begin{equation}\label{intro eq:End is Tinfty}
	q^{-n}\End_{H_{2n}}(P_n)\simeq \Kc(T(n,\infty)):= \lim_{k\rightarrow\infty} \Kc(T(n,k)).
\end{equation}
Here $q^{-n}$ denotes a shift in quantum grading, while the limits indicate a type of stable limiting procedure of Khovanov complexes.  The torus link complexes $\Kc(T(n,k))$ have been widely studied \cite{Bar-Natan,shumakovitch-khoho,stosic}, and deep conjectures of Gorsky-Oblomkov-Rasmussen and Gorsky-Oblomkov-Rasmussen-Shende \cite{GOR,GORS} relate $\Kc(T(n,\infty))$ to algebraic geometry and higher representation theory; see also Section \ref{sec:GOR conjecture} further below.  Substantial progress on these conjectures came in \cite{Hog_polyaction} where Hogancamp proved the following theorem.

\begin{theorem}[{\cite{Hog_polyaction}}]
	\label{thm:hog-polyaction}
	Fix $n\geq 2$.  Then there is a representative for $\End(P_n)$ which deformation retracts onto a differential bigraded $\Z[u_1,\dots,u_n]$-module $W_n=\Z[u_1,\dots,u_n]/(u_1^2) \otimes \Lambda[\xi_2,\dots,\xi_n]$ with differential satisfying
	\begin{enumerate}
		\item $d(u_k)=0$ for each $k=1,\dots,n$; and
		\item $d(\xi_k)\in 2u_1u_k + \Z[u_2,\dots,u_{k-1}]$ for each $k=2,\dots,n$.
	\end{enumerate}
	The differential is $\Z[u_1,\dots,u_n]$-linear.\footnote{We note that this a deformation retraction of chain complexes; both $\End(P_n)$ and $W_n$ have algebra structures, but the retraction is not an algebra map.}
\end{theorem}

The element $u_n$ in Theorem \ref{thm:hog-polyaction} corresponds to an endomorphism (more precisely, a homotopy class of endomorphism) via Equation \eqref{intro eq:End is Tinfty} which, in our conventions, will be written as a map
\[U_n:\q^{2n}\Sigma^{2n-2}P_n\to P_n.\]
These endomorphisms are closely related to a smaller, recursive model for $P_n$ due to Cooper-Krushkal \cite{Cooper-Krushkal} in terms of $P_{n-1}$.

In this paper, we first follow the notions introduced in \cite{LLS_tangles},\cite{LLS_func} to define the spectral Temperley-Lieb category $\Sp\TL_n$, a subcategory of the category $\sarc_{2n}-\mathrm{Mod}$ of modules over the spectral arc algebra $\sarc_{2n}$.  The category $\Sp\TL_n$ is enriched in (graded) spectra and comes equipped with a \emph{chains} functor $\chainsfunc:\Sp\TL_n\to\TL_n$ which induces a functor on the corresponding homotopy categories (which we also denote $\chainsfunc$).  We then define the notion of spectral Jones-Wenzl projectors $\proj_n\in\Sp\TL_n$, in analogy with the definition of categorified Jones-Wenzl projectors $P_n\in\TL_n$ (for the homotopy theorist, we note that a spectral Jones-Wenzl projector is an instance of Bousfield localization).  We show that any spectral Jones-Wenzl projector $\proj_n$ is a lift of $P_n$ to $\Sp\TL_n$, in the sense that $\chainsfunc(\proj_n)\simeq P_n$.  Furthermore, we show that each $\proj_n$ for $n\geq 1$ is unique up to stable homotopy equivalence.  The details of the definition are delayed to Section \ref{subsec:spectral-projectors}, but an important consequence of the definitions is that a spectral Jones-Wenzl projector $\proj_n$ satisfies an analog of the categorical idempotency of $P_n$.  We show in Theorem \ref{intro thm: spectral Pn is tori} that both Equations \eqref{intro eq: Pn is inf torus braid} and \eqref{intro eq:End is Tinfty} lift to $\Sp\TL_n$ as in \cite{Wil_TorusLinks,Wil_colored} (the notion of stabilization will be explained in Section \ref{sec:construction}).  We note that for any object of $\Sp\TL_n$, the endomorphism algebra is naturally a spectrum, on account of the spectral enrichment of $\Sp\TL_n$.

\begin{theorem}\label{intro thm: spectral Pn is tori}
For fixed $n$, the spectral $\mathcal{H}_{2n}$-modules $\T_n^k:=\X(T_n^k)$ stabilize as $k\rightarrow\infty$, and the corresponding limiting spectra $\T_n^\infty$ are spectral projectors.  Moreover, there is an equivalence
\begin{equation}\label{intro eq: spectral END is Tinfty}
q^{-n}\End_{\sarc_{2n}}(\mathcal{P}_n)\simeq \X(T_n^{\infty}):=\lim_{k\rightarrow\infty}\X(T_n^k).
\end{equation}
\end{theorem}

More crucially, in our first main result we present a smaller, recursive model for $\proj_n$ in terms of $\proj_{n-1}$ which serves as a lift of the Cooper-Krushkal model \cite{Cooper-Krushkal}.  We emphasize that the existence of the small Cooper-Krushkal model for $P_n$ does not imply \emph{a priori} that there is a corresponding model for $\proj_n$.  We summarize the model as follows (for the more precise version, see Theorem \ref{thm:Pn simplified in detail}).

\begin{theorem}\label{intro thm:Pn simplified}
	There is a filtration $\CKfilt$ of the spectral projector $\mathcal{P}_n$, viewed as a spectral module over  $\End_{\mathcal{H}_{2(n-1)}}(\mathcal{P}_{n-1})$, which lifts the Cooper-Krushkal recursive formula for the categorified projector $P_n$ illustrated in Figure \ref{fig:CK sequence}.
\end{theorem}

This smaller model allows us to investigate potential lifts of Hogancamp's endomorphisms $U_n^k$ for various $k$, which we will summarize in Section \ref{sec:more on endomorphism lifts}.  Here we focus on the cases $n=2,3$ where we can use Equation \eqref{intro eq: spectral END is Tinfty}, together with computations of $\X(T(2,\infty))$ \cite{Wil_TorusLinks} and $\X(T(3,k))$ due to Lobb-Orson-Sch\"utz \cite{LOS}, to arrive at the following results.

\begin{theorem}\label{intro thm: n=3 computations}
	In the following statements, we call a spectral map $F$ a lift of a chain map $f$ if $\chainsfunc(F)\simeq f$.
\begin{enumerate}
	\item \label{intro thm item: yes U2} There exists an endomorphism
	\[\q^4\Sigma^2 \proj_2 \xrightarrow{\mathcal{U}_2} \proj_2\]
	which lifts the map $U_2$ corresponding to $u_2$ in Theorem \ref{thm:hog-polyaction}.
	\item \label{intro thm item: no U3 no Hogancamp} There is no endomorphism
	\[\q^6\Sigma^{4}\proj_3 \to \proj_3\]
	which lifts the map $U_3$ corresponding to $u_3$ in Theorem \ref{thm:hog-polyaction}, and thus no general sense in which Theorem \ref{thm:hog-polyaction} can be lifted to spectra.
	\item \label{intro thm item: yes U3sq} There exists an endomorphism
	\[\q^{12}\Sigma^8 \proj_3 \xrightarrow{\mathcal{U}_3^2} \proj_3\]
	which lifts the map $U_3^2$ corresponding to $u_3^2$ in Theorem \ref{thm:hog-polyaction}.
\end{enumerate}
\end{theorem}

As a corollary to Theorem \ref{intro thm: n=3 computations}\eqref{intro thm item: yes U3sq}, we are able to complete the computation of $\X(T(3,\infty))$, proving a conjecture of Lobb-Orson-Sch\"utz.

\begin{corollary}\label{intro cor:3-strand-torus-link}
	The spectrum $\X(T(3,\infty))$ is given by Lobb-Orson-Sch\"{u}tz \cite[Conjecture 4.1]{LOS_colored}.  Namely:
	\begin{center}
		\begin{tabular}{|c|c|}
			\hline
			q-grading & $\X^q(T(3,\infty))$ \\ \hline 
			$-3$ & $S^0$ \\ \hline
			$-1$ & $S^0$ \\ \hline
			$1$ & $S^2$ \\ \hline
			$3$ & $\Sigma^2 M(\mathbb{Z}/2,2)\vee S^4$ \\ \hline
			$6j+5$ for $j\geq 0$ & $S^{4j+3}\vee S^{4j+4}$ \\ \hline
			$6j+1$ for $j\geq 1$ & $S^{4j+1}\vee S^{4j+2}$ \\ \hline
			$12j-3$ for $j\geq 1$ & $X(\eta 2, 8j-3)\vee S^{8j}$ \\ \hline
			$12j+3$ for $j\geq 1$ & $S^{8j+1}\vee X({}_{2}\eta,8j+2).$\\ \hline 
		\end{tabular}
	\end{center}
	Here, $M(\mathbb{Z}/2,2)$ is the Moore spectrum of $\mathbb{Z}/2$ in degree $2$; i.e. the mapping cone of $2\colon S^2\to S^2$.  For the notation $X$, see \cite{LOS_colored}.
	
\end{corollary}
The only new part of the corollary is the $12j\pm 3$ cases. 

Our last main result concerns a property of the Hochschild homology of Khovanov complexes due to Rozansky.

\begin{theorem}[{\cite{Roz_S1S2}}]
Let $T$ be an $(n,n)$-tangle with $n$ even.  Then the Hochschild homology $HH(H_{2n};\Kc(T))$ is an invariant of the corresponding link $\hat{T}$ in $S^1\times S^2$, illustrated in Figure \ref{fig:s1s2}.
\end{theorem}

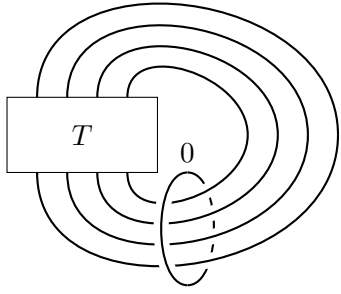
\begin{figure}
	\[\begin{tikzpicture}[xscale=.4,yscale=.5]
		\draw[thick] (0,-1) to[out=0,in=0] (0,-4);
		
		\foreach \i in {2,3,4,5}
		{
			\drawover{(-\i,-1) to[out=-90,in=-90] (\i,0);}
			\draw[thick] (-\i,1) to[out=90,in=90] (\i,0);
		}
		
		\drawover{(0,-1) to[out=180,in=180] (0,-4);}
		\node[above] at (0,-1) {0};
		
		\draw (-6,1) rectangle (-1,-1);
		\node at (-3.5,0) {$T$};
		
	\end{tikzpicture}\]
	\caption{An $(n,n)$-tangle $T$ determines the link $\hat{T}\subset S^1\times S^2$ shown here, where the 0-surgery on the meridian unknot gives $S^1\times S^2$.}
	\label{fig:s1s2}
\end{figure}

In \cite{LLS_tangles} after defining the spectral lift $\X(T)$ for such tangles, Lawson-Lipshitz-Sarkar asked whether or not the topological Hochschild homology similarly gives an invariant of $\hat{T}\subset S^1\times S^2$.  In this paper we show, perhaps surprisingly, that the answer is negative.

\begin{theorem}\label{intro thm:THH bad}
The topological Hochschild homology $THH(\sarc_{2n};\X(T))$ is not an invariant of the link $\hat{T}$ in $S^1\times S^2$.  
\end{theorem}

Although Theorem \ref{intro thm:THH bad} may appear unrelated at first, we will see that ideas involving 2-strand projectors, together with computations of Lobb-Orson-Sch\"utz \cite{LOS}, will provide many of the essential ingredients of the proof.

\subsection{More on spectral lifts of the endomorphisms $U_n^k$}\label{sec:more on endomorphism lifts}

In Section \ref{sec:Obstructing U_n^k}, we will give an obstruction to lifting each of Hogancamp's endomorphisms $U_n$ to \emph{spectral} endomorphisms in $\End_{\mathcal{H}_{2n}}(\mathcal{P}_n)$.  For each $U_n^k$ we identify an obstruction class $\beta_{n,k}$; $U_n^k$ lifts to a map of spectra from $q^{2n}\Sigma^{2n-2}\mathcal{P}_n$ to $\mathcal{P}_n$ (instead of only chain complexes) exactly when the obstruction class vanishes.  The $\beta_{n,k}$ are (homotopy types of) maps of Khovanov spectra induced by cobordisms, which are well-defined by the recent \cite{LLS_func}.  For a precise statement, see Corollary \ref{cor:interesting-cobordism-map}.

A key consequence of Hogancamp's Theorem \ref{thm:hog-polyaction} is that there is a categorification $Q_n$ (the ``symmetric projector") of a `normalization' $(\prod_{k=2}^n(1-q^{2k}))p_n$ of $p_n\in TL_n$, by clearing denominators.  This $Q_n$ is a finite chain complex, obtained from $P_n$ by taking the iterated mapping cone of all of the $U_i$ for $1\leq i \leq n$.  The $\beta_{n,k}$ obstructions do not all vanish, so that there does not appear to be a spectral version of these $Q_n$.  Indeed, since $\beta_{3,1}\not\simeq 0$, we have that $Q_3$ does not have a lift to spectra in a suitable sense.  However it seems natural to ask if there is a polynomial $f(q)$ such that there is a natural finite $\mathcal{H}_{2n}$-module spectrum, the decategorification of whose chain complex is $f(q)p_n\in TL_n$.  As a particular form of this, we conjecture:
\begin{Conjecture}\label{intro conj:finiteness}
	For each $n\geq 2$, there exists some $k>0$ for which $\beta_{n,k}\simeq 0$.  Equivalently, some power $U_n^k$ of the chain map $U_n\colon \q^{2n}\Sigma^{2n-2} P_n\to P_n$ lifts to a map of spectra $\mathcal{P}_n\to \mathcal{P}_n$. 
\end{Conjecture}
Theorem \ref{intro thm: n=3 computations} verifies Conjecture \ref{intro conj:finiteness} for $n=2,3$.  Note however that even if spectral lifts of the $U_n^k$ exist, unlike in the chain setting, they are not guaranteed to be well-defined.

\subsection{The GOR conjecture}
\label{sec:GOR conjecture}
The conjecture of Gorsky-Oblomkov-Rasmussen referenced in Section \ref{sec:intro}, which we will refer to as the \emph{GOR Conjecture}, can be stated as follows.
\begin{Conjecture}[{\cite{GOR}}]\label{conj:gor}
	The unreduced Khovanov chain complex $\mathit{Kc}(T(n,\infty))$ is equivalent as a chain complex to the graded-commutative differential graded algebra $A_n$ generated by variables $u_1,\ldots, u_{n}$ and $\xi_2,\ldots,\xi_{n-1}$, subject to $u_1^2=0$, and with differential 
	\[
	d(\xi_m)=\sum^{m}_{i=1} u_{i}u_{m+1-i},\qquad \mbox{ and } \qquad d(u_k)=0.
	\]
	The $u_i$ have grading $(2i-2,2i)$ and the $\xi_i$ have grading $(2i-1,2i+2)$.
\end{Conjecture}
\begin{remark}
	The complex $\mathit{Kc}(T(n,\infty))$ is itself a differential graded algebra (using `stacking' of torus braids); Conjecture \ref{conj:gor} does not give a homotopy equivalence of dg-algebras.  Indeed, the conjecture has been established for $n=2$, but the Khovanov chains $\xi_2,u_2$ in $T(2,\infty)$ satisfy $\xi_2^2=  u_2^3$.
\end{remark}
\begin{remark}
	For reduced $\mathfrak{sl}_N$ theories with $N\geq 2$, there are related conjectures; see \cite{GORS}.
\end{remark}

We were unable to find a generalization of Conjecture \ref{conj:gor} to a description of $\End_{\sarc_{2n}}(\cP_n)=\X(T(n,\infty))$, but we think this merits further attention, given the explicit nature of Conjecture \ref{conj:gor}.

In search of $\X(T(n,\infty))$ we list two descriptions of the algebra $A_n$ from Conjecture \ref{conj:gor}, both from \cite{GOR}.  First,  $H_*(A_n)$ is the Hochschild homology of the category of matrix factorizations for a certain potential $W$; this description does not have a concrete restatement in spectra, as far as the authors are aware.

Perhaps closer to topology, $A_n$ is a certain subquotient of the affine Lie algebra $\widehat{\mathfrak{sl}_2}$, associated to the integrable representation of $\widehat{\mathfrak{sl}_2}$ at level $1$.  Integrability here means that this representation also arises as a representation of the affine loop group of $SU(2)$, see \cite{pressley-segal}.  This raises the question if it is possible to naturally associate, to the representation $V_{(0,0,1)}$ of $\widehat{\mathfrak{sl}_2}$, a naturally-occurring spectrum to fit into a spectral version of Conjecture \ref{conj:gor}.

\subsection{Outline}
This paper is organized as follows.  In section \ref{sec:preliminaries} we review what we will need on spectra and homotopy colimits.  In Section \ref{sec:planar algebraic spectra} we will recall some facts about Khovanov homology and spectra of tangles, which will be the technical building blocks for the paper.  In Section \ref{sec:spTL cat}, we introduce the spectral Temperley-Lieb category and its operations and properties, as well as define the \emph{spectral projectors}, which are our objects of interest.  In Section \ref{sec:construction} we construct the spectral projectors.  In Section \ref{sec:CK-recursion} we show that the Cooper-Krushkal \cite{Cooper-Krushkal} presentation of the ordinary projector admits a generalization to spectral projectors; this together with Section \ref{sec:Obstructing U_n^k}, forms the technical heart of the paper.  In Section \ref{sec:Obstructing U_n^k} we use the spectral Cooper-Krushkal recursion to describe certain endomorphisms of the spectral projector, associated with cobordism maps of links.  In Section \ref{sec:3-strands}, we apply the work of the previous sections to determine the stable Khovanov spectra of $3$-stranded torus links.  In Section \ref{sec:THH}, we use the same calculations to show that topological Hochschild homology of Khovanov spectra of tangles does not yield a link invariant for links in $S^1\times S^2$.  

\subsection{Acknowledgments}
It is a pleasure to thank Anna Marie Bohmann, Teena Gerhardt, Kristen Hendricks, Matt Hogancamp, Inbar Klang, Slava Krushkal, Tyler Lawson, Robert Lipshitz, Cary Malkiewich, Mona Merling, Sucharit Sarkar, Hirofumi Sasahira, Dirk Sch\"{u}tz, Ian Zemke, and Melissa Zhang  for helpful conversations.  
The first author was supported by NSF DMS-2203828.

\section{Preliminaries on spectra}\label{sec:preliminaries}
	

\subsection{Symmetric spectra and homotopy colimits}\label{sec:spectra}

We will require certain facts about spectra, see e.g. \cite{barnes-roitzheim}.  Following \cite{LLS-Burnside,LLS_tangles}, we work in the category of symmetric spectra \cite{HSS-symmetric}.  We list some of the basic facts and notations we will use below.
\begin{itemize}
    \item The category of symmetric spectra will be denoted $\Sp$.  This category is symmetric monoidal, where the monoidal structure is given by smash product, denoted $\wedge$.  
    \item The sphere spectrum will be denoted $\sphere$.  The (trivial) basepoint spectrum will be denoted by $*$.
    \item Given $X,Y\in\Sp$, we let $\Hom(X,Y)$ denote the set of maps of spectra $X\to Y$.  Meanwhile, we use $[X,Y]$ to denote the set of homotopy classes of maps $X\to Y$.
    \item The category $\Sp$ is \emph{enriched} over itself; given $X,Y\in\Sp$, we let $\HOM(X,Y)$ denote the morphism \emph{spectrum} between them. If $X$ and $Y$ are cofibrant, fibrant, respectively, then $\pi_0(\HOM(X,Y)):=[\sphere,\HOM(X,Y)]=[X,Y].$ (We will often have made implicit choices of cofibrant, or fibrant, replacements of the relevant $X,Y$).  
    \item The (reduced) \emph{chains functor} described in \cite[Section 2.7]{LLS_tangles} will be denoted by $\Sp \xrightarrow{\chainsfunc} \Kom$, where $\Kom$ denotes the category of chain complexes of abelian groups.  This satisfies $\chainsfunc(\sphere)\simeq\Z$, where $\simeq$ indicates quasi-isomorphism.
    \item The (reduced) \emph{homology functor} will be denoted by $\Sp\xrightarrow{H_*}\Rmod[\Z]$, and satisfies $H_*(X) \cong H_*(\chainsfunc(X))$.
\end{itemize}

To identify equivalences of spectra we will often make use of Whitehead's theorem.

\begin{theorem}[Whitehead's theorem]\label{thm:Whitehead for spectra}
Let $X\xrightarrow{f}Y$ be a map of bounded below symmetric spectra such that the induced map on homology
\[H_*(X)\xrightarrow{f_*}H_*(Y)\]
is an isomorphism.  Then $f$ is a homotopy equivalence of spectra.
\end{theorem}

A commuting diagram of spectra is just an ordinary functor $\cC\xrightarrow{}\Sp$ from a small category $\cC$ to spectra.  However, we will also require diagrams of spectra which commute only up to homotopies, which themselves are coherent only up to higher homotopies, and so on.  Following \cite[Section 4]{LLS-Burnside}, we have the following summary.

\begin{definition}[{\cite[Definition 4.6]{LLS-Burnside}, following \cite{vogt}}]\label{def:homotopy coherent diagram}
A \emph{homotopy coherent diagram} of spectra, denoted $\cC\xrightarrow{F}\Sp$, consists of a small category $\cC$ and the following assignments.
\begin{itemize}
    \item For each object $x\in\cC$, a spectrum $F(x)\in\Sp$.
    \item For each $n\geq 1$ and each sequence of composable maps
    \[x_0\xrightarrow{f_1}x_1\xrightarrow{f_2}\cdots\xrightarrow{f_n}x_n\]
    in $\cC$, a morphism (homotopy)
    \[[0,1]^{n-1}_+\wedge F(x_0) \xrightarrow{F(f_1,\cdots,f_n)} F(x_n)\]
    which acts as a homotopy between $F(f_n)\circ \cdots \circ F(f_1)$ and $F(f_n\circ\cdots\circ f_1)$.  For a full list of the properties required see \cite[Definition 4.6]{LLS-Burnside}.
\end{itemize}
\end{definition}

We do not distinguish between homotopy coherent diagrams and ordinary functors (which are just homotopy coherent diagrams with `identity homotopies') in our notation, but it will be clear in context whether or not a given diagram is only homotopy coherent.

We will often be working with functors (and homotopy coherent diagrams) out of certain cube-shaped categories, which we explain here.

\begin{definition}\label{def:cube}
    The category $\cube^1$ consists of two objects $\{0,1\}$ and one non-identity morphism, as shown below.
    \[ 0 \rightarrow 1 \]
    The \emph{$m$-dimensional cube category} $\cube^m$ is the $m$-fold product of $\cube^1$.  Objects are called vertices $v\in\cube^m$, and are written as sequences of $0$'s and $1$'s (we write $v_i$ for the $i$-th term in such a sequence).  `Length one' morphisms between two vertices whose sequences differ only in a single entry are called \emph{edges}.
    
    The \emph{$m$-dimensional based cube category $\basedcube^m$} has objects
    \[\Ob(\basedcube^m) = \Ob(\cube^m)\sqcup \{*\},\]
    and all of the morphisms of $\cube^m$, together with one morphism $v\rightarrow *$ for each vertex $v\in\cube^m\setminus \{1^m\}$. Homotopy coherent diagrams $\basedcube^m\xrightarrow{F}\Sp$ will always be required to satisfy $F(*)=*\in\Sp$.

    Note that, associated to any homotopy coherent diagram $\cube^m\xrightarrow{F}\Sp$, there is a \emph{based extension} $\basedcube^m\xrightarrow{F_+}\Sp$ defined by assigning $*\in\Sp$ to the new object $*\in\basedcube^m$ (the extension of morphisms and higher homotopies, if present, is determined by this assignment).
\end{definition}

For $k\leq m$, there is a `projection' functor $\basedcube^{m-k}\times \basedcube^k \xrightarrow{}\basedcube^m$ sending vertices $(v,w)$ to $vw$ in $\cube^m$, and sending any term $(*,v)$ or $(v,*)$ to $*$.

\begin{lemma}\label{lem:products of cubes}
For any $k\leq m$, the projection $\basedcube^{m-k}\times\basedcube^k \xrightarrow{} \basedcube^m$ is homotopy cofinal.
\end{lemma}
\begin{proof}
    The undercategory of each entry $(i,j)\in \cube^{m-k}\times\cube^k\subset \basedcube^m$ is a (nonempty) product of categories of the form $X^\alpha\times Y^\beta$ where $X,Y$ are either of $\cube$ or $\basedcube$; the nerves of such categories are contractible.  The undercategory of $*\in \basedcube^m$ is the subcategory $*\times\basedcube^k\cup \basedcube^{m-k}\times *$; the nerve of which is also contractible.  
\end{proof}

We can also define subcubes of the cube category, as well as based versions.

\begin{definition}\label{def:subcube}
Fix a vertex $v\in\cube^N$.  Fix some subset of indices $I\subset \{1,\dots,N\}$ such that $v_i=0$ for all $i\in I$.  The \emph{subcube} $\cube^N|_{v,I}$ is the subcategory of $\cube^N$ containing all vertices $w$ such that $w_j=v_j$ for all $j\not\in I$, and all edges between such vertices.  In other words, it is the subcube which `starts' at $v$ and contains all edges in directions indicated by $I$.  Note that this is isomorphic as a category to $\cube^{|I|}$.

The \emph{based subcube} $\basedcube^N|_{v,I}$ is then the based version, consisting of $\cube^N|_{v,I} \sqcup \{*\}$ and including morphisms $w\rightarrow\{*\}$ for all $w$ other than the `ending' vertex determined by
\[w_j=\begin{cases}
    v_j & \text{if $j\not\in I$}\\
    1 & \text{if $j\in I$}
\end{cases}.\]
This based subcube is isomorphic as a category to $\basedcube^{|I|}$.
\end{definition}

Now a map of spectra $X\xrightarrow{}Y$ is equivalent to a functor $\cube^1\xrightarrow{}\Sp$.  A generalization of this observation provides the following notion of a natural transformation.

\begin{definition}\label{def:nat transf}
A \emph{natural transformation} from one functor (respectively homotopy coherent diagram) $\cC\xrightarrow{F}\Sp$ to another $\cC\xrightarrow{G}\Sp$ is a functor (respectively homotopy coherent diagram)
\[\cC\times\cube^1 \xrightarrow{\eta} \Sp\]
such that the restrictions $\eta|_{\cC\times\{0\}}$ and $\eta|_{\cC\times\{1\}}$ recover $F$ and $G$.
\end{definition}

With these ideas in place, we can discuss in slightly more detail various notions related to homotopy colimits.  See  \cite[Section 2.5, Proposition 2.10]{LLS_tangles} and \cite[Section 4.2, Definition 4.10]{LLS-Burnside} for further discussion on these constructions and properties.

\begin{proposition}\label{prop:hocolim properties}
    Let $\cC\xrightarrow{F}\Sp$ denote either a functor or a homotopy coherent diagram (henceforth \emph{diagram}) from a small category $\cC$ to symmetric spectra $\Sp$. Then there is a well-defined \emph{homotopy colimit} $\hocolim\, F\in \Sp$ which satisfies the following properties.
    \begin{enumerate}[label=(H-\arabic*),leftmargin=1cm]
        \item \label{it:hocolim functorial} Homotopy colimits are functorial, in that a natural transformation $F\to F'$ induces a map $\hocolim(F)\to \hocolim(F')$.  Moreover, if the natural transformation gives an equivalence on all objects $F(x)\simeq F'(x) \,\forall\, x\in\cC$, then the induced map is also an equivalence $\hocolim(F)\simeq \hocolim(F')$.
        \item \label{it:hocolim of product} For a diagram $\cC\times \mathcal{D} \xrightarrow{F} \Sp$, there is a natural equivalence
        \[\hocolim_{c\in\cC} (\hocolim_{d\in\mathcal{D}} F(c,d) ) \simeq \hocolim_{\cC\times\mathcal{D}} F.\]
        \item \label{it:smash preserves hocolim} The smash product $\wedge$ preserves homotopy colimits in each variable.
        \item \label{it:cofinal preserves hocolim} Cofinal functors preserve homotopy colimits, in the sense that if $\cC\xrightarrow{G}\mathcal{D}$ is cofinal, then for any diagram $\mathcal{D}\xrightarrow{F}\Sp$, we have
            \[\hocolim F \simeq \hocolim (F\circ G).\]
        In particular, hocolims of functors from $\basedcube^m$ can be computed using corresponding product functors from $\basedcube^{m-k}\times \basedcube^{k}$ for any $k\leq m$.
        \item \label{it:chains preserve hocolim} There is also a notion of homotopy colimits of chain complexes, and the reduced chains functor $\chainsfunc$ satisfies that there is a natural equivalence:
        \[
        \hocolim (\chainsfunc \circ F) \to \chainsfunc (\hocolim F).
        \]
    \end{enumerate}
\end{proposition}

\begin{definition}\label{def:mapping cone}
The \emph{mapping cone} of a map of spectra $X\xrightarrow{f} Y$ is defined to be the homotopy colimit over the following diagram of spectra
   \[
   \Cone(f) \simeq \hocolim\left(
        \begin{tikzcd}
            X \ar[r,"f"] \ar[d] & Y\\
            * & 
        \end{tikzcd}
         \right).
    \]
The mapping cone comes equipped with a cofibration sequence
\begin{equation}\label{eq:cone cofib general}
Y \to \Cone(f) \to \Sigma X
\end{equation}
and $f$ is an equivalence of spectra if and only if $\Cone(f)\simeq *$.
\end{definition}

In parallel with generalizing maps via natural transformations, we have the following proposition for mapping cones arising in that manner.

\begin{proposition}\label{prop:cones of hocolims}
Suppose $\cC\times\cube^1\xrightarrow{\eta}\Sp$ is a natural transformation between two diagrams $\cC\xrightarrow{F}\Sp$ and $\cC\xrightarrow{G}\Sp$.  Denote the induced map on hocolims via item \ref{it:hocolim functorial} of Proposition \ref{prop:hocolim properties} by
\[\hocolim \, F\xrightarrow{\eta_*}\hocolim \, G.\]
Define an extension $\cC\times\basedcube^1\xrightarrow{\eta_+}\Sp$ by setting $\eta_+(x,*)=*$ for all $x\in\cC$.  Then we have
\[\Cone(\eta_*) \simeq \hocolim_{\cC\times\basedcube^1} \eta_+.\]
\end{proposition}
\begin{proof}
    This follows from \ref{it:hocolim of product}.
\end{proof}

We will use the notation $\gSp$ to denote the category of \emph{$\Z$-graded spectra}, where all of the spectra and maps between them are graded.  All of the constructions and results of this section adapt to this setting readily.

\subsection{Sums and differences of maps}\label{sec:sums-and-differences}

For symmetric spectra $X,Y$, the set of homotopy classes of maps $[X,Y]$ is an abelian group by the following.  For maps $f,g\colon X \to Y$, we may suspend to obtain maps $\Sigma X\to \Sigma Y$.  We also have a map of spectra $\Sigma X \to \Sigma X\vee \Sigma X$ via the pinch map $p\colon S^1\to S^1\vee S^1$.  There is of course also an `unpinching' $u\colon S^1\vee S^1 \to S^1$, sending each circle factor to $S^1$.  Additionally, there is a well-defined map $\Sigma f\vee \Sigma g \colon (\Sigma X)\vee (\Sigma X)\to (\Sigma Y)\vee (\Sigma Y)$.  Combining these maps, we can add $f,g$ as follows:
\[
\Sigma f\vee \Sigma g\in [\Sigma X \vee \Sigma X,\Sigma Y\vee \Sigma Y]\xrightarrow{u_*} [\Sigma X\vee \Sigma X,\Sigma Y]\xrightarrow{p^*} [\Sigma X,\Sigma Y]\xrightarrow{\Sigma^{-1}} [X,Y]
\]
That is, $f+g=\Sigma^{-1}(u \circ (\Sigma f\vee \Sigma g)\circ p)$.

Additive inverses are defined using the map $r\colon S^1\to S^1$ sending $e^{i\theta}\to e^{-i\theta}$.  We leave it to the reader to check that $f+(-f)\in [X,Y]$ is homotopic to zero.

It seems a rather trivial point, but will be consequential for us, that $f+(-f)$ is not canonically zero.  That is, there is not a canonical null-homotopy of the sum $f+(-f)$, or in particular the map $1-1=\mathrm{Id}+(-\mathrm{Id})\colon \mathbb{S}\to\mathbb{S}$.  Indeed, the set of nullhomotopies of $1-1\simeq 0\colon \mathbb{S}\to \mathbb{S}$ is non-canonically identified with $\pi_1(\mathbb{S})=\mathbb{Z}/2$, in that there are essentially two choices for this nullhomotopy, up to homotopy.  We must choose nullhomotopies for this map several times in the simplification of spectral Cooper-Krushkal projectors. 

\subsection{The Spectral Module Category}\label{sec:spectral-modules}
The invariants we study are graded spectral modules over certain graded spectral algebras (or, equally well, spectral categories).  We review the relevant definitions here.

\begin{definition}\label{def:spectral cat}
A \emph{spectral category} is a small category $\cA$ enriched in symmetric spectra.  The morphism spectrum from object $a_i$ to $a_j$ will be denoted $\HOM(a_i,a_j)\in\Sp$.  A spectral category is called \emph{finite} if it has finite object set, and its morphism spectra are weakly equivalent to finite CW spectra.  A spectral category is \emph{graded} if its morphism spectra are $\Z$-graded.  We will often view a finite graded spectral category $\cA$ as a graded \emph{spectral algebra}, abusing notation slightly to write
\begin{equation}\label{eq:spectral cat as algebra}    
\cA:=\bigvee_{a_i,a_j\in \mathrm{Ob}(\cA)} \HOM_{\cA}(a_i,a_j) \in \gSp,
\end{equation}
where morphism composition defines graded \emph{multiplication} $\cA\wedge\cA\rightarrow\cA$.  For two morphisms which are not composable in the spectral category, the product is trivial.  Said differently, the multiplication map
\[
\HOM_{\cA}(a_i,a_j)\wedge \HOM_{\cA}(a_k,a_\ell)\to \HOM_{\cA} (a_i,a_\ell)
\]
is trivial if $a_j\neq a_k$.

The \emph{smash product} of two spectral categories $\cA\wedge\cB$ is the spectral category with $\mathrm{Ob}(\cA\wedge\cB):= \mathrm{Ob}(\cA)\times\mathrm{Ob}(\cB)$ and morphism spaces
\[\HOM_{\cA\wedge\cB}((a,b),(a',b')):= \HOM_{\cA}(a,a') \wedge \HOM_{\cB}(b,b').\]
\end{definition}

\begin{definition}\label{def:spec multimodule}
    A \emph{spectral multimodule $\cM$} over a set of specified spectral categories $\cA_1,\dots,\cA_k$ and $\cB$ is a functor
    \[(\cA_1\wedge\cdots\wedge \cA_k)^{\mathrm{op}}\wedge \cB \xrightarrow{\cM} \Sp.\]
    A \emph{graded} spectral multimodule $\cM$ is one whose source and target categories are all graded, with $\cM$ respecting this grading.
    
    In more detail, for each (ungraded) object $(a_1,\dots,a_k,b)$, we have a graded spectrum
    \[\cM(a_1,\dots,a_k,b) \in \gSp.\]
    Then for each (graded) morphism spectrum
    \[\left(\bigwedge_i \cA_i(a_i',a_i)\right) \bigwedge \cB(b,b') = \HOM_{(\cA_1\wedge\cdots\wedge \cA_k)^{\mathrm{op}}\wedge \cB}((a_1,\dots,a_k,b),(a_1',\dots,a_k',b')),\]
    we have a grading-preserving map
    \[\left(\bigwedge_i \cA_i(a_i',a_i)\right) \bigwedge \cB(b,b') \xrightarrow{\cM} \HOM_{g\Sp}( \cM(a_1,\dots,a_k,b) , \cM(a_1',\dots,a_k',b') ),\]
    which itself is equivalent to the data of a grading-preserving `multi-action map' of (graded) spectra
    \[\left(\bigwedge_i \cA_i(a_i',a_i) \right) \bigwedge \cM(a_1,\dots,a_k,b) \bigwedge \cB(b,b') \rightarrow \cM(a_1',\dots,a_k',b').\]

    Abusing notation slightly as in Equation \eqref{eq:spectral cat as algebra}, we will often view a spectral multimodule as the (graded) spectrum $\cM$ built out of the values of the corresponding functor
    \begin{equation}\label{eq:spectral functor as multim}
    \cM:=\bigvee_{(a_1,\dots,a_k,b)\in\mathrm{Ob}((\cA_1\wedge\cdots\wedge\cA_k)^{\mathrm{op}}\wedge\cB)} \cM(a_1,\dots,a_k,b) \in\gSp,
    \end{equation}
    with commuting actions of the $\cA_1,\dots,\cA_k$ and $\cB$ given by the values on morphism spectra.  
\end{definition}

\begin{definition}\label{def:cat of spec multimod over fixed algs}
Fix spectral categories $\cA_1,\dots,\cA_k$ and $\cB$.  The \emph{derived category of (graded) spectral multimodules over $\cA_1,\dots,\cA_k$ and $\cB$}, which we denote by $\SpmMod(\cA_1,\dots,\cA_k;\cB)$, is defined as follows.
\begin{itemize}
    \item Objects are (graded) spectral multimodules $\cM$ over $\cA_1,\dots,\cA_k$ and $\cB$ as in Definition \ref{def:spec multimodule}.
    \item Using the viewpoint of Equation \eqref{eq:spectral functor as multim}, a natural transformation is a collection of morphisms of (graded) spectra, running over $x\in \mathrm{Ob}(\cA_1\wedge\cdots\wedge \cA_k)^{\mathrm{op}}\wedge \cB $, by $F(x)\colon \cM(x)\xrightarrow{}\mathcal{N}(x)$, so that the following diagram commutes:
    \begin{equation}\label{eq:nat-trans-spec-cat}
    \begin{tikzpicture}
        \node (a0) at (0,0) {$\left(\bigwedge_i \cA_i(a_i',a_i) \right) \bigwedge \cM(a_1,\dots,a_k,b) \bigwedge \cB(b,b')$};
        \node (a1) at (6,0) {$\cM(a_1',\dots,a_k',b')$};
        \node (b0) at (0,-2) {$\left(\bigwedge_i \cA_i(a_i',a_i) \right) \bigwedge \cN(a_1,\dots,a_k,b) \bigwedge \cB(b,b')$};
            \node (b1) at (6,-2) {$\cN(a_1',\dots,a_k',b')$};

        \draw[->] (a0) -- (a1) ;
        \draw[->] (a0) -- (b0) node[pos=.5,anchor=east] {\scriptsize $\mathrm{id}\wedge F(a_1,\ldots, a_k,b)\wedge \mathrm{id}$};
        \draw[->] (a1) -- (b1) node[pos=.5,anchor= west] {\scriptsize $F(a_1',\ldots,a_k',b')$};
        \draw[->] (b0) -- (b1);
    \end{tikzpicture}
    \end{equation}
    A morphism $M\to N$ in the derived category $\SpmMod(\cA_1,\dots,\cA_k)$ is a homotopy class of maps as in (\ref{eq:nat-trans-spec-cat}) between a cofibrant replacement for $M$ and a fibrant replacement for $N$.

\end{itemize}

More generally, we define the (graded) \emph{derived multicategory of spectral multimodules $\SpmMod$} to be the multicategory with:
    \begin{itemize}
        \item Objects are finite, graded spectral categories.
        \item Each multimorphism category $\SpmMod(\cA_1,\dots,\cA_k;\cB)$ is defined as above.
        \item Multicomposition is given by derived smash product over the relevant spectral categories (denoted by $\otimes_{\cA}$).
    \end{itemize}
    Strictly speaking, this multicategory is built by choosing cofibrant replacement functors for each tuple $(\cA_1,\dots,\cA_k,\cB)$, and making a slight further adjustment (see \cite[Section 4.2.1]{LLS_func}) so that the derived smash product is strictly associative, functorial, and unital.  See \cite[Section 4.2.1 and Definition 4.20]{LLS_func}; in \cite{LLS_func}, the multicategory written $\SpmMod$ here is denoted $\mathrm{SBim}$.  We will elide these details here.
\end{definition}

  More generally, for a ring-spectrum $R$, and module spectra $M,N$ over $R$, there is a \emph{morphism spectrum} $\HOM_{R}(M,N)$ of $R$-equivariant morphisms $M\to N$.  If $R$ is as in \ref{eq:spectral cat as algebra} and $M,N$ are (cofibrant/fibrant replacements of the objects) as in (\ref{eq:spectral functor as multim}), then $\pi_0(\HOM_{R}(M,N))$ is the set of homotopy classes as in (\ref{eq:nat-trans-spec-cat}).  We will sometimes replace the module spectra with which we work with cofibrant/fibrant replacements without further comment.

All of the results of Sections \ref{sec:spectra} and \ref{sec:sums-and-differences} have analogues in the setting of spectral multimodules.  In particular, we can consider homotopy colimits of spectral multimodules, and Proposition \ref{prop:hocolim properties} continues to hold, as does the discussion of signs in Section \ref{sec:sums-and-differences}.  

\section{Planar Algebraic Structure of Khovanov Spectra}\label{sec:planar algebraic spectra}

In this section we review the gluing theory from \cite[Section 4]{LLS_func}.  This is very similar in spirit to the framework of \cite{Bar-Natan} and especially Jones' planar algebras from \cite{Jones}, as used in \cite{roberts-planar}.

\subsection{Punctured Tangles and Planar Spectra}

In what follows we will often be manipulating tangles and their corresponding Khovanov spectra using the structure of a planar algebra as in \cite{Jones,Bar-Natan}. We summarize the main ideas here, starting with the tangles themselves.

\begin{definition}[{\cite[Definition 4.1]{LLS_func}}]\label{def:punc tang}
    A \emph{punctured tangle} (called a \emph{diskular tangle} in \cite{LLS_func}) is a (unoriented) tangle diagram $T$ drawn in a disk $D^2\backslash (D_1^{\circ}\cup\ldots \cup D_k^{\circ})$ with some collection $\{D_i\}_{i=1,\ldots,k}$ of subdisks removed, with an even number of endpoints on each boundary component.  In addition, on each boundary component, the tangle endpoints are required to miss the right-most point of the boundary (this induces an ordering on the tangle endpoints).  See Figure \ref{fig:punc tang diag} for an example.  As a matter of convention, if the tangle $T$ has $n$ boundary points on a particular boundary circle of $D^2\backslash (D_1^{\circ}\cup\ldots \cup D_k^{\circ})$, we require that the $n$ boundary points lie at the non-unity $(n+1)$-first roots of unity (this assumption is made so that tangles that have the same number of boundary points may be readily glued).  For a more complete definition, see \cite[Definition 4.1]{LLS_func}.
\end{definition}

\begin{figure}[h!]
\centering \def\svgwidth{6in}
\begingroup%
  \makeatletter%
  \providecommand\color[2][]{%
    \errmessage{(Inkscape) Color is used for the text in Inkscape, but the package 'color.sty' is not loaded}%
    \renewcommand\color[2][]{}%
  }%
  \providecommand\transparent[1]{%
    \errmessage{(Inkscape) Transparency is used (non-zero) for the text in Inkscape, but the package 'transparent.sty' is not loaded}%
    \renewcommand\transparent[1]{}%
  }%
  \providecommand\rotatebox[2]{#2}%
  \newcommand*\fsize{\dimexpr\f@size pt\relax}%
  \newcommand*\lineheight[1]{\fontsize{\fsize}{#1\fsize}\selectfont}%
  \ifx\svgwidth\undefined%
    \setlength{\unitlength}{569.06510961bp}%
    \ifx\svgscale\undefined%
      \relax%
    \else%
      \setlength{\unitlength}{\unitlength * \real{\svgscale}}%
    \fi%
  \else%
    \setlength{\unitlength}{\svgwidth}%
  \fi%
  \global\let\svgwidth\undefined%
  \global\let\svgscale\undefined%
  \makeatother%
  \begin{picture}(1,0.31423623)%
    \lineheight{1}%
    \setlength\tabcolsep{0pt}%
    \put(0,0){\includegraphics[width=\unitlength,page=1]{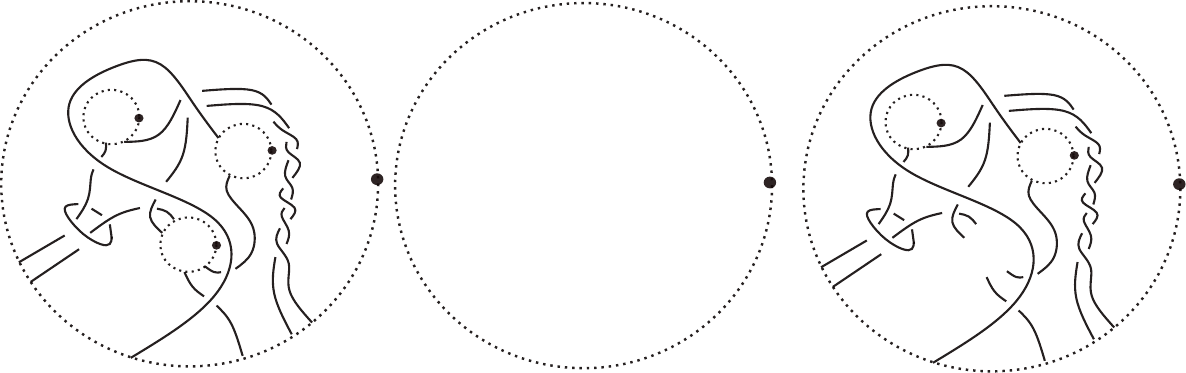}}%
    \put(0.08530229,0.21180946){\color[rgb]{0.03921569,0.01568627,0.01568627}\transparent{0.88888901}\makebox(0,0)[lt]{\lineheight{1.25}\smash{\begin{tabular}[t]{l}1\end{tabular}}}}%
    \put(0.76099285,0.21197732){\color[rgb]{0.03921569,0.01568627,0.01568627}\transparent{0.88888901}\makebox(0,0)[lt]{\lineheight{1.25}\smash{\begin{tabular}[t]{l}1\end{tabular}}}}%
    \put(0.20204993,0.18192203){\color[rgb]{0.03921569,0.01568627,0.01568627}\transparent{0.88888901}\makebox(0,0)[lt]{\lineheight{1.25}\smash{\begin{tabular}[t]{l}2\end{tabular}}}}%
    \put(0.87469563,0.17622518){\color[rgb]{0.03921569,0.01568627,0.01568627}\transparent{0.88888901}\makebox(0,0)[lt]{\lineheight{1.25}\smash{\begin{tabular}[t]{l}2\end{tabular}}}}%
    \put(0.15348291,0.10346761){\color[rgb]{0.03921569,0.01568627,0.01568627}\transparent{0.88888901}\makebox(0,0)[lt]{\lineheight{1.25}\smash{\begin{tabular}[t]{l}3\end{tabular}}}}%
    \put(0,0){\includegraphics[width=\unitlength,page=2]{tangle-multi-2.pdf}}%
  \end{picture}%
\endgroup%

\caption{Left is a diskular $(2,2,4;6)$ tangle $S$, in the middle is a diskular $(;4)$ tangle $T$, and on the right is the composition $T\circ_3 S$, which is a diskular $(2,2;6)$ tangle; compare with Figure 4.1 of \cite{LLS_func}.  }
\label{fig:punc tang diag}
\end{figure}

It is clear that one punctured tangle can be `inserted' into the puncture of a second punctured tangle, assuming that the corresponding boundary points can be matched up.  Indeed, the insertion of a tangle $S$ into the $i$-th missing disk of a tangle $T$ will be denoted $T\circ_i S$ (see Figure \ref{fig:punc tang diag}; note that the tangle $T$ being inserted may also have `holes', i.e. may be a $(i_1,\ldots,i_k;j)$-tangle with $k>0$).  Intuitively then, we would like to assign so-called \emph{planar spectra} to our punctured tangles, together with gluing instructions on how the planar spectum of a tangle obtained by `inserting' is related to the spectra of the constituent tangles.  In the style of a TQFT, we accomplish this by assigning:
\begin{center}
\begin{tabular}{| c | c |}
\hline 
    Points on boundaries & Spectral categories (equivalently, spectral algebras) \\ \hline
    Punctured tangles &  Spectral multimodules over the spectral algebra of the boundary\\ \hline
    Insertion & Smash product\\ \hline
\end{tabular}
\end{center}

We will also require some extra data to encode possible grading shifts once orientations are taken into account.  The precise terminology is summarized below.  See \cite[Section 4]{LLS_func} for more details.

\begin{definition} \label{def:tang multicat}
    The (unoriented) \emph{tangle movie multicategory} $\Tang$ is the multicategory (enriched in categories) defined as follows.

    \begin{itemize}
        \item The objects of $\Tang$ are the non-negative, even integers $n$.
        \item An object in the multimorphism category $\Tang(m_1,\dots, m_k;n):=\Hom_{\Tang}(m_1,\dots, m_k;n)$ is a pair $(T,P)$ where $P$ is an integer (to be used to encode grading shifts), and $T$ is a punctured tangle having the correct number $m_i$ of endpoints on each puncture boundary, and $n$ endpoints on the outer boundary.
        \item A morphism $(T,P)\xrightarrow{\Sigma}(T',P')$ in $\Tang(m_1,\dots m_k;n)$ consists of a cobordism $\Sigma\colon T\rightarrow T'$, arranged as a sequence of elementary cobordisms such that
        \begin{gather*}P'-P=P(\Sigma):= \text{\#(R1/R2 moves creating positive crossings)}\\-\text{\#(R1/R2 moves removing positive crossings)}\end{gather*}
        Some further relations on sequences of elementary cobordisms are imposed in \cite{LLS_func}.
        \item Given two punctured tangles and integers
        \[(T',P')\in\Tang(p_1,\dots,p_\ell;m_i), \,\, (T,P)\in\Tang(m_1,\dots,m_k;n),\]
        the $i^{\text{th}}$ composition
        \[(T',P')\circ_i(T,P):=(T'\circ_i T,P+P')\in\Tang(m_1,\dots, m_{i-1}, p_1, \dots, p_\ell, m_{i+1}, \dots, m_k; n)\]
        is defined by inserting $T'$ into the $i^{\text{th}}$ puncture in $T$ (which has $m_i$ endpoints), resulting in a new punctured tangle $T'\circ_i T$ as illustrated in Figure \ref{fig:punc tang diag}.
    \end{itemize}
    See \cite[Definitions 4.1 and 4.4]{LLS_func} for a more complete definition. 
\end{definition}

Due to the lax nature of the constructions of Khovanov spectra, we must pass to a certain enlargement of $\Tang$ where we allow multimorphism objects to consist of entire trees of composable multimorphisms without actually performing the composition.  
 
\begin{definition}[{\cite[Definition 4.8]{LLS_func}}]\label{def:thickened tang multicat}
The \emph{thickened (unoriented) tangle movie multicategory} $\thTang$ is the canonical thickening of $\Tang$ as summarized below.
\begin{itemize}
\item Objects of $\thTang$ are the same as objects of $\Tang$.
\item A \emph{basic} multimorphism in $\thTang$ is a multimorphism category in $\Tang$.
\item A more general multimorphism object in $\thTang$ is a planar rooted tree, whose edges are labelled by the objects of $\thTang$ and whose vertices are labeled by basic multimorphisms, whose inputs and output correspond to the edge labels (see \cite{LLS_func} for more details).  For such a labeled tree $T$, we can form the composite of all the multimorphisms in $T$ (as a composition in $\Tang$) to obtain a multimorphism object $\circ(T)$ in $\Tang$.
\item For labeled trees $T_1,T_2$ in $\thTang(x_1,\ldots, x_n;y)$, a multimorphism morphism $T_1$ to $T_2$ is a $\Tang$ multimorphism morphism $\circ(T_1)\to \circ(T_2)$. 
\item Multicomposition of general multimorphism objects in $\thTang$ (trees of composable punctured tangles) is induced by stacking labeled trees.  Multi-composition of multimorphism morphisms is similar.
\end{itemize}
See \cite[Definition 4.8 and Remark 4.10]{LLS_func} for more details.
\end{definition}

Moreover, if we consider the multicategories $\Tang$ or $\thTang$ and identify, on the morphism morphisms, isotopic (Rel boundary) cobordisms, we obtain quotient multicategories $\st$ and $\thst$ respectively.

With these definitions in place, together with the definitions in Section \ref{sec:spectral-modules}, we can give a name to the type of object that we wish to consider.

\begin{definition}
A (graded) \emph{planar spectral algebra} is a functor $\thTang\xrightarrow{\cF}\SpmMod$.  We shall also refer to any specific multimodule $\cF(T,P)$, for such an $\cF$, as a \emph{planar spectrum}.
\end{definition}

\subsection{Khovanov spectra}

In \cite[Section 4.3]{LLS_func}, Lawson-Lipshitz-Sarkar constructed a planar Khovanov spectrum $\X(T,P)$ for pairs of punctured tangles $T$ and integers $P$.  When the tangle $T$ is oriented, $P$ can be chosen so that the homology of $\X(T,P)$ recovers the planar algebraic Khovanov invariant $Kh(m(T))$ of the mirror tangle $m(T)$ itself viewed as a multimodule over the relevant arc algebras assigned to the various boundaries.  In this section, we recall a few features and properties of their construction.  

Using the language developed above, we have the following.

\begin{definition}\label{def:Khov spec}
The \emph{(planar) Khovanov spectral algebra} is the functor
\[\X\colon \thTang\rightarrow\SpmMod\]
defined in \cite[Definition 4.24]{LLS_func} and summarized below.
\begin{itemize}
    \item On an object $n\in 2\Z_{> 0}$, $\X(n)$ is the graded spectral Khovanov arc algebra $\sarc_n$ on $n$ (ordered) points defined in \cite[Section 4]{LLS_tangles}.
    \item On a basic multimorphism object (pair of punctured tangle diagram $T$ with integer $P$), $\X(T,P)$ is the Khovanov spectral multimodule defined in \cite[Section 4.3]{LLS_func} which we will summarize further below.
    \item On a more general multimorphism object (tree of composable elementary objects $(T_i,P_i)$), the value of $\X$ is given by the composition (derived smash product) of the values $\X(T_i,P_i)$ in $\SpmMod$.
    \item On a multimorphism morphism (movie of punctured tangle diagrams) $\Sigma$, $\X(\Sigma)$ is a composition of elementary movie multimodule morphisms also defined in \cite[Section 4.3]{LLS_func}.
\end{itemize}
The functor $\X$ descends to define a projective functor $\X\colon \thst\to \SpmMod$ (projective meaning that composition is well-defined only up to a sign).  
\end{definition}

Before giving a summary of the main points of the construction, we highlight some key features of the Khovanov spectra.  Throughout this summary we will be viewing $\X(n)$ and $\X(T,P)$ as graded spectra with extra structures through the lens of Equations \eqref{eq:spectral cat as algebra} and \eqref{eq:spectral functor as multim}.

First, the Khovanov spectra are $\Z$-graded; we refer to this internal grading as the \emph{quantum degree} or $q$-degree.  Thus we may write, for each $n\in2\Z$ and pair $(T,P)$ of punctured tangle and integer,
\begin{equation}\label{eq:gradings overview}
\X(n)=\bigvee_{j=0}^{2n} \X^j(n),\quad \X^j(T,P)=\bigvee_{j\in\Z} \X^j(T,P),
\end{equation}
where $\X^j(n)$ and $\X^j(T,P)$ are the component of $\X(n)$ (resp. $\X(T,P)$) in grading $j$.  Note that each $\X(n)$ is supported in finitely many $q$-degrees.  The same is true for any $\X(T,P)$, where all but finitely many of the summands $\X^j(T,P)$ are a point.  Note also that any action map $\X(n)\wedge \X(T,P)\to\X(T,P)$ is a map of graded spectra, so that it decomposes into components $\X^j(n) \wedge \X^k(T,P) \to \X^{j+k}(T,P)$.  Note that the wedge sum decomposition for $\X(T,P)$ in \eqref{eq:gradings overview} is really as a spectrum, not as a multimodule.  For a graded spectrum $S$, we will also use the notation $S\atq{j}$ to refer to the spectrum of $S$ in grading $j$.  

In addition, the integer $P$ corresponds to a shift in both homological and $q$-degree.  Using the notations $\Sigma$ to denote suspension and $\q$ to denote $q$-degree shift, we have, by definition:
\[\X(T,P) = \q^{-3P}\Sigma^{-P}\X(T,0), \quad \text{or equivalently,} \quad \X(T,P)^{j} = \Sigma^{-P}\X(T,0)^{j+3P}.\]
These gradings allow us to recover the graded algebraic Khovanov invariants as follows.  We write $\Kc^{i,j}(T)$ for the Khovanov complex of a tangle $T$, with $i$ referring to the homological grading, and $j$ the quantum grading.  We write $\Kh^{i,j}(T)$ for the Khovanov homology.

\begin{proposition}\label{prop:X lifts Kh}
Let $T^o$ be an oriented punctured tangle diagram, with $P(T^o)$ positive crossings.  Let $T$ denote the punctured tangle diagram obtained from $T^o$ by forgetting orientation.  Then we define
\[\X(T^o):=\X(T,P(T^o))\]
which satisfies
\begin{equation}\label{eq:X lifts Khov}
\widetilde{H}_i(\X(T^o)\atq[\big]{j})\cong Kh^{i,j}(m(T^o))
\end{equation}
where $m(T^o)$ is the mirror of the oriented tangle diagram $T^o$.
\end{proposition}

\begin{remark}
The grading conventions chosen here match those in \cite{LLS_func}, and are chosen in part to ensure that spectral algebras are built, rather than coalgebras.  Mirrors are then taken in order to recover Khovanov's original invariants which were built via cochain complexes.
\end{remark}

Next, being a functor, $\X$ must naturally respect multicomposition, implying that for any punctured tangles $T_1,T_2$ composable along a boundary with $m_i\in 2\Z$ points, we must have a canonical equivalence
\begin{equation}\label{eq:gluing equiv}
\X(T_1\circ T_2,P_1+P_2) \simeq \X(T_1,P_1)\otimes_{\X(m_i)}\X(T_2,P_2)
\end{equation}
for any choice of integers $P_1,P_2$, where $T_1\circ T_2$ indicates the proper insertion of punctured tangle diagrams.

Furthermore, much like Khovanov's algebraic invariants, the planar Khovanov spectra are well-defined (up to stable equivalence) for tangles, not just tangle diagrams.  In our unoriented tangle language, one way of writing this is as follows.

\begin{theorem}\label{thm:X well-defined}
Let $T\xrightarrow{\Sigma}T'$ be a punctured tangle isotopy (i.e. a corbordism of punctured tangle diagrams consisting of only Reidemeister moves).  Then for any integer $P$, the corresponding morphism $(T,P)\xrightarrow{\Sigma}(T',P+P(\Sigma))$ induces a stable homotopy equivalence
\[\X(\Sigma):\X(T,P)\xrightarrow{\simeq}\X(T',P+P(\Sigma)).\]
In particular, if both $T$ and $T'$ are isotopic rel boundary as oriented punctured tangles, then $\X(T)\simeq\X(T')$ as graded planar spectra.
\end{theorem}

Finally, any crossing in $T$ gives rise to a cofibration sequence lifting the corresponding short exact sequence of Khovanov complexes.

\begin{proposition}\label{prop:crossing cofib seq}
    Let $T$ be a punctured tangle diagram with a specified crossing, and let $T_i$ ($i=0,1$) be the punctured tangle diagram resulting from taking the $i$-resolution of this crossing.  Then for any integer $P$, we have a cofibration sequence
    \begin{equation}\label{eq:crossing cofib seq}
        \X(T_1,P)\rightarrow \q^{-1}\X(T,P) \rightarrow \q\Sigma\X(T_0,P).
    \end{equation}
    Equivalently, the planar spectrum $\X(T,P)$ can be written as a mapping cone
    \begin{equation}\label{eq:crossing mapping cone}
        \X(T,P) \simeq \Cone\left( \q^2\X(T_0,P) \xrightarrow{\X(s)} \q\X(T_1,P) \right)
    \end{equation}
    where $s$ is the saddle cobordism going from $T_0$ to $T_1$.
\end{proposition}
We will use both Theorem \ref{thm:X well-defined} and Proposition \ref{prop:crossing cofib seq} extensively throughout our arguments.

We now present a very brief outline of the construction of the spectral multimodules $\X(T,P)$.  A punctured disk with specified boundary point counts $m_1,\dots,m_k,n$ on the various boundaries gives rise to a so-called \emph{tangle-shape multicategory} $\cT_{m_1,\dots,m_k;n}$ (see \cite[Section 4.3]{LLS_func}, based upon \cite[Section 3.2.2]{LLS_tangles}).  In short, the objects of this category correspond to the idempotent decompositions of the various Khovanov arc algebras at each boundary, together with the decomposition of the Khovanov multimodule for \emph{any} flat  tangle that would fit in the given punctured disk.  The multimorphisms are either single element sets or empty sets, depending on whether or not the chosen idempotents are lined up properly for non-zero multiplication.  Then, as described in \cite{LLS_func, LLS_tangles}, the data of a spectral multimodule over the relevant spectral arc algebras is equivalent to a multifunctor $\cT_{m_1,\dots,m_k;n}\xrightarrow{\X_T} \gSp$.

The multifunctor $\X_T$ is built in stages.  Let $C$ be the number of crossings $T$.  Then the tangle $T$ (together with some auxiliary data of \emph{pox}, see \cite{LLS_tangles}) can be used to define a functor
\[\cT_{m_1,\dots,m_k;n}\tilde{\times} \basedcube^C \xrightarrow{F_T} \divcob'\]
where $\cT_{m_1,\dots,m_k;n}\tilde{\times} \basedcube^C$ denotes a type of thickened product of the two categories (see \cite[Section 3.2.4]{LLS_tangles}), and $\divcob'$ denotes (a canonical thickening of) the \emph{divided cobordism category} defined in \cite[Section 3.1]{LLS_tangles}.  The thickened $\divcob'$ admits a further multifunctor $V_{HKK}$ to the graded Burnside multicategory $\gBurn$ based upon the arguments of Hu-Kriz-Kriz \cite{HKK}.  From $\gBurn$, we can reach $\gSp$ via Elmendorff-Mandell $K$-theory \cite{Elmendorf-Mandell} as in \cite{LLS_tangles}, giving us the composition of multifunctors
\[\cT_{m_1,\dots,m_k;n}\tilde{\times} \basedcube^C \xrightarrow{F_T} \divcob' \xrightarrow{V_{HKK}} \gBurn \xrightarrow{K} \gSp.\]

At this stage, a rectification result of Elmendorff-Mandell produces a strict (un-thickened) multifunctor $(\cT_{m_1,\dots,m_k;n}\tilde{\times} \basedcube^C)^0 \rightarrow \gSp$, and we take a hocolim over the cubical maps in $\basedcube^C$ to arrive at a multifunctor
\[\cT_{m_1,\dots,m_k;n} \xrightarrow{\X_T} \gSp\]
as desired.  The final invariant is a shifted version of this, which we may write loosely as
\[\X(T,P) = \q^{\frac{n}{2}+2C-3P}\Sigma^{-P}\X_T.\]

\section{The spectral Temperley-Lieb category}\label{sec:spTL cat}
In this section we define the spectral Temperley-Lieb category and describe some of its properties.  We close the section by defining spectral Cooper-Krushkal projectors.

\subsection{Definition and basic properties}\label{sec:spTL basics}

Often we want to think of tangles within a single, non-punctured disk as having $n$ endpoints each at the `top' and `bottom' boundaries.  Such tangles then have various operations available, such as concatenation by stacking one tangle vertically (in the page) atop the other, and (partial) traces.  Similarly, we can view spectral $\X(2n)=\SpArcAlg$-modules $\cM\in\SpmMod(\varnothing;\X(2n))$ as sitting in a disk with $n$ endpoints each on the top and bottom, and then having these same operations available.  These ideas amount to special cases of the more general planar algebraic framework, which we codify now.  The following definition is a special case of Definition \ref{def:punc tang}, but we make it a separate statement to emphasize that the $(n,n)$-tangles, defined below, will be used for vertical stacking.

\begin{definition} \label{def:TLn tang}
Recall that a \emph{flat} tangle is a tangle diagram with no crossings.  An $(n,n)$-tangle (called an \emph{$n$-diskular tangle} in \cite{LLS_func}) is a punctured tangle $T\in\Tang(\varnothing;2n)$ where $n$ of the boundary points are drawn on the `upper semi-circle' (i.e. viewing the boundary of the disk as $\partial D^2=[0,2\pi]/(0\sim 2\pi)$, these $n$ points lie in $(0,\pi)$), while the remaining $n$ boundary points are drawn on the `lower semi-circle' (i.e. $(\pi,2\pi)$).  A flat $(n,n)$-tangle $\delta$ is also referred to as a \emph{Temperley-Lieb diagram}; the subset of $\Tang(\varnothing;2n)$ consisting of all Temperley-Lieb diagrams will be denoted $\TL_n$.  Such a $\delta\in\TL_n$ has a corresponding spectral $\X(2n)$-module $\X(\delta)\in\SpmMod(\varnothing;\X(2n))$, which we sometimes call a \emph{flat-tangle module}.  In line with Definition \ref{def:punc tang}, the endpoints of the tangle should be at roots of unity, however, we typically do not draw the tangle this way, to  emphasize stacking.  
\end{definition}

Our preferred category will be spectral modules built out of flat tangle modules in the following manner.

\begin{definition}\label{def:spectral TL cat}
Let $\Sp\TL_n$, the \emph{spectral Temperley-Lieb category}, be the full subcategory of $\SpmMod(\varnothing;\X(2n))$ consisting of objects (right) spectral $\X(2n)$-modules constructed in the following way.  If $\delta$ is any flat $(n,n)$-tangle and $f\colon \Sigma^\ell \q^{k}X(\delta)\to \cM_0$ is a map of $\X(2n)$-modules, we say that $\Cone(f)$ is the result of attaching a \emph{cell} $\Sigma^\ell \q^{k}\X(\delta)$ to $\cM_0$ along $f$.  The objects of $\Sp\TL_n$ are the spectra that are homotopy equivalent to any spectrum produced by a (possibly infinite) sequence of attachings as above, where for all $k_0$, there are only finitely many cells attached with $k\leq k_0$. 

	Any such object $\cM$ is referred to as a \emph{Temperley-Lieb spectrum}, and will be represented visually as $\ILmodbox{\cM}$.  We write $\q^k\cM$ for the object $\cM\in \Ob(\Sp\TL_n)$ with grading shifted up by $k$.  
\end{definition}
In particular, note that the Khovanov spectrum of any $(n,n)$-tangle $T$ is a Temperley-Lieb spectrum $\X(T,P)\in\Sp\TL_n$ (for any choice of integer $P$), since $\X(T,P)$ is built as a homotopy colimit (over the finite cube category $\underline{2}^N$) of various flat tangle modules (according to the various resolutions of the crossings of $T$).


\begin{remark}
    The spectral Temperley-Lieb category $\Sp\TL_n$ can be viewed as an analog of the category of spaces that have the homotopy type of a finite CW complex.  The role of ordinary cells is taken by ``cells'' $\Sigma^\ell\q^k\X(\delta)$.  One consequence of the finiteness condition is that the underlying spectra of our $\X(2n)$-modules are finite spectra in each individual quantum grading.
\end{remark}

We note, in particular, that the homotopy colimit of a diagram of \emph{cells} (or wedge sums of cells) $\mathbb{S}^\ell\wedge q^k\X(\delta)$ (where we treat the basepoint $*$ as a cell)  is an object of $\Sp\TL_n$.  Moreover, if $X,Y$ are in $\Sp\TL_n$, and $f\colon X \to Y$ is a map of $\X(2n)$-spectra, then $\Cone(f)$ is in $\Sp\TL_n$.  This can be seen inductively by attaching one cell of $X$ to $Y$ (along $f$) at a time.   

We will write $[X,Y]$ (or $[X,Y]_{\Sp\TL_n}$) for the set of homotopy classes of maps between objects $X,Y$ of $\Sp\TL_n$.  

The composition of multimorphisms in $\Tang$, as respected by the Khovanov functor $\X$, endows the categories $\Sp\TL_n$ with various spectral versions of the operations in the usual Temperley-Lieb category which we list now.  First we set some notation.

\begin{definition}\label{def:id tang and id mod}
    The \emph{identity $(n,n)$-tangle}, denoted $\Itang_n$, is the flat tangle consisting of $n$ strands going straight from top to bottom.  The corresponding \emph{identity spectral module} is $\Imod_n:=\X(\Itang_n,0).$ 
\end{definition}

\begin{definition}\label{def:vert comp}
    \emph{Vertical composition} is the operation \[\Sp\TL_n\times\Sp\TL_n\xrightarrow{\vertcomp}\Sp\TL_n\]
    defined via planar operations (that is, multicomposition) in $\Tang$ as follows.  Let $\cF^v_n\in\Tang(2n,2n;2n)$ denote the flat tangle illustrated in Figure \ref{fig:vertical comp}.  Given $\cM_1,\cM_2\in\Sp\TL_n$, we define $\cM_1\vertcomp\cM_2\in\Sp\TL_n$ via multicomposition with the spectral multimodule $\X(\cF^v_n,0)\in\SpmMod(\X(2n),\X(2n);\X(2n))$ also illustrated in Figure \ref{fig:vertical comp}.  We can write this multicomposition as
    \[\cM_1\vertcomp \cM_2 := (\cM_1,\cM_2) \otimes_{(\X(2n),\X(2n))} \X(\cF^v_n,0).\]
\end{definition}

\begin{remark}
The pair $(\Sp\TL_n,\vertcomp)$ is not quite a monoidal category, but instead can be viewed as a multifunctor from the canonical groupoid enrichment of the associative operad $\mathrm{Assoc}'$ to the category of categories, with $\Imod_n$ playing the role of the identity.
\end{remark}

\begin{remark}
One may also work with the category $\Sp\TL^m_n\subset\SpmMod(\varnothing;\X(m+n))$ with objects built from attaching cells $\Sigma^\ell\q^k\X(\delta)$ for $(m,n)$-tangles $\delta$.  In this setting there is a corresponding notion of vertical composition via stacking
\[\Sp\TL^m_n \times \Sp\TL^n_p \xrightarrow{\vertcomp} \Sp\TL^m_p\]
which we will occasionally make use of.  It will be clear from context what version of $\vertcomp$ is being used in any given situation.  All of the operations below also have similar generalizations to $\Sp\TL_n^m$.
\end{remark}

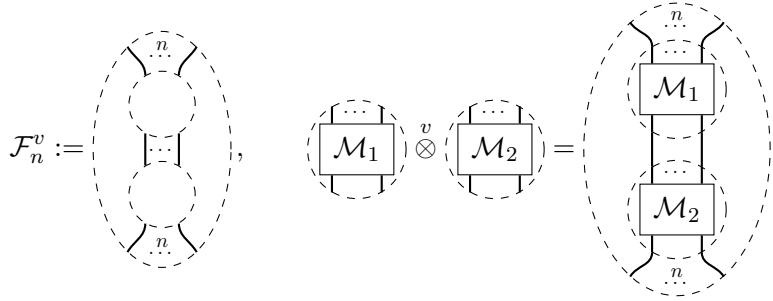
\begin{figure}
    \[ \cF^v_n:= \begin{tikzpicture}[baseline=-.5ex]
        \emptydisk{E1}{0}{.6}
        \emptydisk{E2}{0}{-.6}
        \node[align=center,fit={(E1 disk)(E2 disk)},draw,circle,xscale=.7,yscale=1.2,dashed] (E1E2) {};
        \draw[thick] (E1 disk.\diskne) to[out=90,in=\disksw] (E1E2.\diskne);
        \draw[thick] (E1 disk.\disknw) to[out=90,in=\diskse] (E1E2.\disknw);
        \draw[thick] (E2 disk.\disksw) to[out=-90,in=\diskne] (E1E2.\disksw);
        \draw[thick] (E2 disk.\diskse) to[out=-90,in=\disknw] (E1E2.\diskse);
        \draw[thick] (E1 disk.\disksw)--(E2 disk.\disknw);
        \draw[thick] (E1 disk.\diskse)--(E2 disk.\diskne);
        \foreach \innerdisk/\ns in {E1 disk/north,E2 disk/south}{
            \nstrandsalongpath{\innerdisk.\ns}{E1E2.\ns}
            }
        \nstrandsalongpath[]{E1 disk.south}{E2 disk.north}
    \end{tikzpicture}
    \,,\quad\quad
    \ILmodbox{\cM_1} \vertcomp \ILmodbox{\cM_2} = \begin{tikzpicture}[baseline=-.5ex]
        \straightmodboxindisk{M1}{0}{.8}{\cM_1}
        \straightmodboxindisk{M2}{0}{-.8}{\cM_2}
        \node[align=center,fit={(M1 disk)(M2 disk)},draw,circle,xscale=.7,yscale=1.1,dashed] (M1M2) {};
        \draw[thick] (M1 disk.\diskne) to[out=90,in=\disksw] (M1M2.\diskne);
        \draw[thick] (M1 disk.\disknw) to[out=90,in=\diskse] (M1M2.\disknw);
        \draw[thick] (M2 disk.\disksw) to[out=-90,in=\diskne] (M1M2.\disksw);
        \draw[thick] (M2 disk.\diskse) to[out=-90,in=\disknw] (M1M2.\diskse);
        \draw[thick] (M1 disk.\disksw)--(M2 disk.\disknw);
        \draw[thick] (M1 disk.\diskse)--(M2 disk.\diskne);
            \foreach \innerdisk/\ns in {M1 disk/north,M2 disk/south}{
            \nstrandsalongpath{\innerdisk.\ns}{M1M2.\ns}
            }
    \end{tikzpicture}
    \]
    \caption{The flat tangle $\cF^v_n\in\Tang(2n,2n;2n)$, together with a visual representation of the vertical composition $\cM_1\vertcomp\cM_2$.  The dots indicate $n$ total strands.}
    \label{fig:vertical comp}
\end{figure}

\begin{definition}\label{def:horiz comp}
    \emph{Horizontal composition} is an operation
    \[\Sp\TL_{n_1}\times \Sp\TL_{n_2} \xrightarrow{\horizcomp}{\Sp\TL_{n_1+n_2}}\]
    defined via planar operations in $\Tang$ as follows.  Let  $\cF^{\sqcup}_{n_1,n_2}\in\Tang(2n_1,2n_2;2(n_1+n_2))$ denote the flat tangle illustrated in Figure \ref{fig:horizontal comp}.  Given $\cM_i\in\Sp\TL_{n_i}$, $i=1,2$, we define $\cM_1\horizcomp\cM_2\in\TL_{n_1+n_2}$ via multicomposition with the spectral multimodule $\X(\cF^{\sqcup}_{n_1,n_2},0)$ also illustrated in Figure \ref{fig:horizontal comp}.  We can write this multicomposition as
    \[\cM_1\horizcomp\cM_2 := (\cM_1,\cM_2) \otimes_{(\X(2n_1),\X(2n_2))} \X(\cF^{\sqcup}_{n_1,n_2},0) .\]
\end{definition}

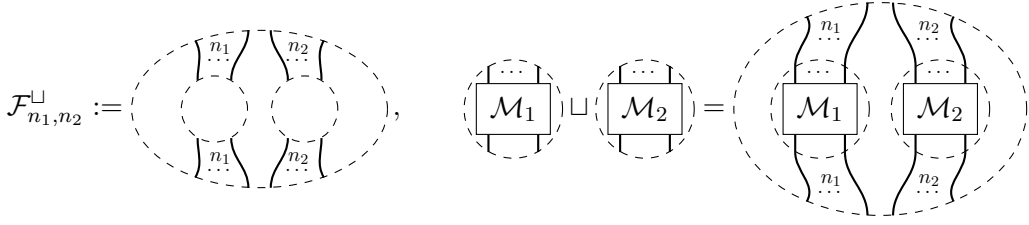
\begin{figure}
    \[  
    \cF^\sqcup_{n_1,n_2}:=
        \begin{tikzpicture}[baseline=-.5ex]
        \emptydisk{E1}{-.6}{0}
        \emptydisk{E2}{.6}{0}
        \node[fit={(E1 disk)(E2 disk)},draw,circle,yscale=.8,xscale=1.3,dashed] (E1E2) {};
        \foreach \a/\b/\c/\d in {\disknw/90/\diskse/\disknw, \diskne/90/-90/95, \disksw/-90/\diskne/\disksw, \diskse/-90/90/265}
            {
            \draw[thick] (E1 disk.\a) to[out=\b,in=\c] (E1E2.\d);
            }
        \foreach \a/\b/\c/\d in {\disknw/90/-90/85, \diskne/90/\disksw/\diskne, \disksw/-90/90/275, \diskse/-90/\disknw/\diskse}
            {
            \draw[thick] (E2 disk.\a) to[out=\b,in=\c] (E1E2.\d);
            }
        \foreach \pta/\ptb/\num in {E1 disk.north/E1E2.105/n_1, E1 disk.south/E1E2.255/n_1, E2 disk.north/E1E2.75/n_2, E2 disk.south/E1E2.285/n_2}
            {
            \path (\pta)--(\ptb) node[pos=.5,align=center,scale=.7] {\nstrands[\num]};
            }
        \end{tikzpicture}
    \,,\quad\quad
    \ILmodbox{\cM_1}\horizcomp\ILmodbox{\cM_2}=
        \begin{tikzpicture}[baseline=-.5ex]
        \straightmodboxindisk{E1}{-.8}{0}{\cM_1}
        \straightmodboxindisk{E2}{.8}{0}{\cM_2}
        \node[fit={(E1 disk)(E2 disk)},draw,circle,yscale=.8,xscale=1.1,dashed] (E1E2) {};
        \foreach \a/\b/\c/\d in {\disknw/90/\diskse/\disknw, \diskne/90/-90/95, \disksw/-90/\diskne/\disksw, \diskse/-90/90/265}
            {
            \draw[thick] (E1 disk.\a) to[out=\b,in=\c] (E1E2.\d);
            }
        \foreach \a/\b/\c/\d in {\disknw/90/-90/85, \diskne/90/\disksw/\diskne, \disksw/-90/90/275, \diskse/-90/\disknw/\diskse}
            {
            \draw[thick] (E2 disk.\a) to[out=\b,in=\c] (E1E2.\d);
            }
        \foreach \pta/\ptb/\num in {E1 disk.north/E1E2.105/n_1, E1 disk.south/E1E2.255/n_1, E2 disk.north/E1E2.75/n_2, E2 disk.south/E1E2.285/n_2}
            {
            \path (\pta)--(\ptb) node[pos=.5,align=center,scale=.7] {\nstrands[\num]};
            }
        \end{tikzpicture}
    \]
    \caption{The flat tangle $\cF^\sqcup_{n_1,n_2}\in\Tang(2n,2n;2n)$, together with a visual representation of the horizontal composition $\cM_1\horizcomp\cM_2$.  The dots indicate various numbers of total strands.}
    \label{fig:horizontal comp}
\end{figure}

\begin{definition}\label{def:partial trace}
    The \emph{partial trace} is an operation
    \[\Sp\TL_n\xrightarrow{T}\Sp\TL_{n-1}\]
    defined via planar operations in $\Tang$ as follows.  Let $\cF^{T}_n\in\Tang(2n;2(n-1))$ denote the flat tangle illustrated in Figure \ref{fig:partial trace}.  Given $\cM\in\Sp\TL_n$, we define $T(\cM)\in\Sp\TL_{n-1}$ via multicomposition with spectral multimodule $\X(\cF^T_n,0)$ also illustrated in Figure \ref{fig:partial trace}.  We can write this multicomposition as
    \[T(\cM):=\cM \otimes_{\X(2n)} \X(\cF^T_n,0).\]    
    More generally, the \emph{$k$th partial trace} $T^k(\cM)\in\Sp\TL_{n-k}$ (for $k\leq n$) is defined by applying $T$ $k$ times.
\end{definition}

\begin{figure}
    \[
    \cF^T_n:=
        \begin{tikzpicture}[baseline=-.5ex]
        \emptydisk{M1}{0}{0}
        \draw[thick] (M1 disk.\diskne) to[out=90,in=90] ++(.4,0) node[inner sep=0] (newpt1){};
        \draw[thick] (M1 disk.\diskse) to[out=-90,in=-90] ++(.4,0) node[inner sep=0] (newpt2){};
        \draw[thick] (newpt1.north)--(newpt2.south);
        \node[fit={(M1 disk)(newpt1)(newpt2)}, circle, dashed, draw,yscale=1.3,xscale=1.4] (TM1) {};
        \foreach \stpt/\stdir/\endpt/\enddir in {\disknw/90/\disknw/\diskse, \disksw/-90/\disksw/\disksw, 75/90/\diskne/\disksw, 285/-90/\diskse/\disknw}
            {
            \draw[thick] (M1 disk.\stpt) to[out=\stdir,in=\enddir] (TM1.\endpt);
            }
        \nstrandsalongpath[n-1]{M1 disk.100}{TM1.north}
        \nstrandsalongpath[]{M1 disk.260}{TM1.south}
        \end{tikzpicture}
    \,,\quad\quad T\left(\ILmodbox{\cM}\right)=
        \begin{tikzpicture}[baseline=-.5ex]
        \straightmodboxindisk{M1}{0}{0}{\cM}
        \draw[thick] (M1 disk.\diskne) to[out=90,in=90] ++(.4,0) node[inner sep=0] (newpt1){};
        \draw[thick] (M1 disk.\diskse) to[out=-90,in=-90] ++(.4,0) node[inner sep=0] (newpt2){};
        \draw[thick] (newpt1.north)--(newpt2.south);
        \node[fit={(M1 disk)(newpt1)(newpt2)}, circle, dashed, draw,yscale=1.3,xscale=1.3] (TM1) {};
        \foreach \stpt/\stdir/\endpt/\enddir in {\disknw/90/\disknw/\diskse, \disksw/-90/\disksw/\disksw, 70/90/\diskne/\disksw, 290/-90/\diskse/\disknw}
            {
            \draw[thick] (M1 disk.\stpt) to[out=\stdir,in=\enddir] (TM1.\endpt);
            }
        \draw[thick](M1 disk.70)--(M1 disk.70|-M1 box.north);
        \draw[thick](M1 disk.290)--(M1 disk.290|-M1 box.south);
        \nstrandsalongpath[n-1]{M1 disk.100}{TM1.north}
        \nstrandsalongpath[]{M1 disk.260}{TM1.south}
        \end{tikzpicture}
    \]
    \caption{The flat tangle $\cF^T_n\in\Tang(2n;2(n-1))$, together with a visual representation of the partial trace $T(\cM)$.  The dots indicate the various numbers of total strands.}
    \label{fig:partial trace}
\end{figure}
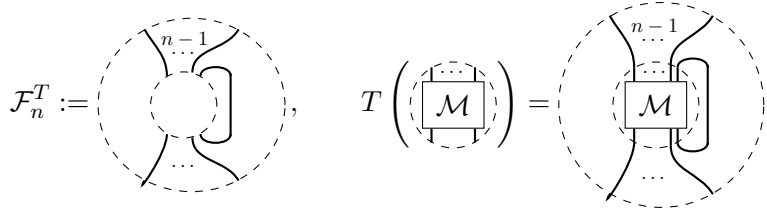

Finally, we have a fundamental adjunction in $\Sp\TL_n$ using traces and horizontal compositions as in \cite[Theorem 2.31]{Hog_polyaction}.

\begin{theorem}\label{thm:adjunction}
Let $\cM\in \Sp\TL_n$ and $\cN\in \Sp\TL_{n+1}$.  Then there are natural homotopy equivalences of morphism spectra:
	\[
	\HOM_{\Sp\TL_{n+1}}(\cM\horizcomp \Imod_1,\cN)\simeq \HOM_{\Sp\TL_{n}}(\cM,\q T(\cN))
	\]
	and
	\[
	\HOM_{\Sp\TL_{n+1}}(\cN,\cM\horizcomp \Imod_1)\simeq \HOM_{\Sp\TL_{n}}(T(\cN),\q \cM).
	\]
\end{theorem}
\begin{proof}
We prove the first equivalence only; the second is analogous.  The proof is visual in nature, making use of the planar algebraic structure of $\Sp\TL_n$ within $\SpmMod$ heavily.  As a first example of such proofs in this paper, we will provide some details which will be glossed over in later arguments.

We fix Temperley-Lieb spectra $\ILmodbox{\cM}\in\Sp\TL_n$ and $\ILmodbox{\cN}\in\Sp\TL_{n+1}$, and seek to define maps
\begin{gather*}
    \HOM_{\Sp\TL_{n+1}}(\cM\horizcomp \Imod_1,\cN)\xrightarrow{\phi} \HOM_{\Sp\TL_{n}}(\cM,\q T(\cN)),\\
    \HOM_{\Sp\TL_{n}}(\cM,\q T(\cN))\xrightarrow{\psi}\HOM_{\Sp\TL_{n+1}}(\cM\horizcomp \Imod_1,\cN).
\end{gather*}
To define $\phi$, we fix a map 
\[
\left(
    \begin{tikzpicture}[baseline=-.5ex]
    \node(M) at (0,0) {$\cM$};
    \node[fit=(M),draw,rectangle] (M box) {};
    \node[inner sep=0](pt1) at ($(M box.north east) +(.2,0)$){};
    \node[inner sep=0](pt2) at ($(M box.south east) +(.2,0)$){};
    \draw[thick] (pt1.north)--(pt2.south);
    \node[fit={(M box)(pt1)(pt2)},draw,circle,dashed] (M disk){};
    \foreach \stpt/\stdir/\endpt/\enddir in
            {M box.\disknw/90/M disk.\disknw/\diskse,
            M box.\diskne/90/M disk.\diskne/\disksw,
            pt1.north/90/M disk.50/-90,
            M box.\disksw/-90/M disk.\disksw/\diskne,
            M box.\diskse/-90/M disk.\diskse/\disknw,
            pt2.south/-90/M disk.310/90}
        {
        \draw[thick] (\stpt) to[out=\stdir,in=\enddir] (\endpt);
        }
    \nstrandsalongpath{M box.north}{M disk.north}
    \nstrandsalongpath[]{M box.south}{M disk.south}
    \end{tikzpicture}
\overset{f}{\longrightarrow}
    \begin{tikzpicture}[baseline=-.5ex]
    \modboxindisk[1.3]{N}{0}{0}{\cN}{n+1}
    \end{tikzpicture}
\right)
\in\HOM_{\Sp\TL_{n+1}}(\cM\horizcomp \Imod_1,\cN)
\]
We define $\phi(f)\in\HOM_{\Sp\TL_{n}}(\cM,qT(\cN))$ to be the following composition:
\begin{equation}\label{eq:adj-pic}
    \begin{tikzpicture}[baseline=-.5ex]
    \modboxindisk{M}{0}{0}{\cM}{n}
    \end{tikzpicture}    
\xrightarrow{\Imap\otimes\X\left(\ILbirth\right)}
    \begin{tikzpicture}[baseline=-.5ex]
    \modbox{M}{0}{0}{\q\cM}
    \node[fit=(M box),draw,dashed,circle]{};
    \emptydisk{E}{1.5}{0}
    \draw[thick] (E) circle(.3);
    \node[fit={(M disk)(E disk)},draw,dashed,circle,yscale=.7,xscale=.9] (O disk){};
    \strandsboxtodisk{M}{O}
    \nstrandsalongpath{M box.north}{M disk.85}
    \nstrandsalongpath[]{M box.south}{M disk.275}

    \end{tikzpicture}
\simeq
    \begin{tikzpicture}[baseline=-.5ex]
    \node (M) at (0,0){$\q\cM$};
    \node[fit=(M),draw,rectangle] (M box) {};
    \node[inner sep=0] at ($(M box.north east)+(.2,0)$) (NWpt){};
    \node[inner sep=0] at ($(M box.north east)+(1,0)$) (NEpt){};
    \node[inner sep=0] at ($(M box.south east)+(.2,0)$) (SWpt){};
    \node[inner sep=0] at ($(M box.south east)+(1,0)$) (SEpt){};
    \node[fit={(M box)(NWpt)(SWpt)},draw,dashed,circle,yscale=.8] (TM disk){};
    \draw[thick] (SWpt.south)--(NWpt.north) to[out=90,in=90] (NEpt.north) -- (SEpt.south) to[out=-90,in=-90] (SWpt.south);
    \node[fit={(TM disk)(NEpt)(SEpt)},draw,dashed,circle,xscale=.9,yscale=.85](O disk){};
    \strandsboxtodisk{M}{O}
    \nstrandsalongpath{M disk.north}{O disk.85}
    \nstrandsalongpath[]{M disk.south}{O disk.275}
    \end{tikzpicture}
\xrightarrow{f\otimes \Imap}
    \begin{tikzpicture}[baseline=-.5ex]
    \node (N) at (0,0){$\q\cN$};
    \node[fit=(N),draw,rectangle] (N box) {};
    \node[fit=(N box),draw,dashed,circle,scale=.9] (N disk) {};
    \node[inner sep=0] at ($(N box.north east)+(.4,0)$) (NEpt){};
    \node[inner sep=0] at ($(N box.south east)+(.4,0)$) (SEpt){};
    \node[fit={(N box)(NEpt)(SEpt)},draw,dashed,circle,scale=1.1] (O disk){};
    \strandsboxtodisk{N}{O}
    \foreach \stpt/\stdir/\endpt/\enddir in
            {N box.50/90/NEpt.north/90,
            N box.310/-90/SEpt.south/-90}
        {
        \draw[thick] (\stpt) to[out=\stdir,in=\enddir] (\endpt);
        }
        \draw[thick] (SEpt.south)--(NEpt.north);
    \nstrandsalongpath{N disk.north}{O disk.85}
    \nstrandsalongpath[]{N disk.south}{O disk.275}
    \end{tikzpicture}.
\end{equation}
Really, to define a morphism 
\[\phi\colon \HOM_{\Sp\TL_{n+1}}(\cM\horizcomp\Imod,\cN)\to \HOM_{\Sp\TL_n}(\cM,qT(\cN))
\]
of spectra, we should define more than the action of $\phi$ on particular morphisms in $\HOM_{\Sp\TL_{n+1}}(\cM\horizcomp\Imod,\cN)$, however, the definition as in (\ref{eq:adj-pic}) can be translated to a morphism between the $\HOM$ spectra without any complication; we prefer the description in (\ref{eq:adj-pic}) as it makes the idea of the construction more apparent.  The first map is the image of the birth cobordism $\ILbirth$ under the Khovanov spectrum functor $\X$.  The planar algebraic structure allows our maps $\X(\ILbirth)$ and $f$ to be viewed locally within any given subdisk of the overall disk, which we may adjust at our leisure as in the middle equivalence.  The initial map also implicitly uses the fact that $\cM\simeq\cM\horizcomp\X(\emptyset,0)$ since $\X(\emptyset,0)$ is just the standard sphere spectrum.  The grading shifts are included within the boxes to avoid clutter, but in fact this is equivalent to shifting the entire module by $q$.

Similarly, to define $\psi$ we fix a map
\[
\left(
    \begin{tikzpicture}[baseline=-.5ex]
    \modboxindisk{M}{0}{0}{\cM}{n}
    \end{tikzpicture}
\overset{g}{\longrightarrow}
    \begin{tikzpicture}[baseline=-.5ex]
    \node (N) at (0,0){$\q\cN$};
    \node[fit=(N),draw,rectangle] (N box) {};
    \node[inner sep=0] at ($(N box.north east)+(.4,0)$) (NEpt){};
    \node[inner sep=0] at ($(N box.south east)+(.4,0)$) (SEpt){};
    \node[fit={(N box)(NEpt)(SEpt)},draw,dashed,circle] (O disk){};
    \strandsboxtodisk{N}{O}
    \foreach \stpt/\stdir/\endpt/\enddir in
            {N box.50/90/NEpt.north/90,
            N box.310/-90/SEpt.south/-90}
        {
        \draw[thick] (\stpt) to[out=\stdir,in=\enddir] (\endpt);
        }
        \draw[thick] (SEpt.south)--(NEpt.north);
    \nstrandsalongpath{N box.north}{O disk.north}
    \nstrandsalongpath[]{N box.south}{O disk.south}
    \end{tikzpicture}
\right)
\in\HOM_{\Sp\TL_n}(\cM,qT(\cN))
\]
and define $\psi(g)\in\HOM_{\Sp\TL_{n+1}}(\cM\horizcomp \Imod_1,\cN)$ to be the following composition:
\[
    \begin{tikzpicture}[baseline=-.5ex]
    \modbox{M}{0}{0}{\cM}
    \node[fit=(M box),draw,dashed,circle,xscale=.9,yscale=.7] (M disk){};
    \node[inner sep=0] at ($(M disk.east)+(.2,.5)$) (NEpt){};
    \node[inner sep=0] at ($(M disk.east)+(.2,-.5)$) (SEpt){};
    \node[fit={(M disk)(NEpt)(SEpt)},draw,dashed,circle,scale=.9] (O disk){};
    \strandsboxtodisk{M}{O}
    \foreach \stpt/\stdir/\endpt/\enddir in
            {O disk.40/-90/O disk.320/90}
        {
        \draw[thick] (\stpt) to[out=\stdir,in=\enddir] (\endpt);
        }
    \nstrandsalongpath{M box.north}{O disk.north}
    \nstrandsalongpath[]{M box.south}{O disk.south}
    \end{tikzpicture}
\xrightarrow{g\otimes \Imap}
    \begin{tikzpicture}[baseline=-.5ex]
    \modbox{N}{0}{0}{\q\cN}
    \node[inner sep=0] at ($(N box.north east)+(.4,0)$) (NEpt){};
    \node[inner sep=0] at ($(N box.south east)+(.4,0)$) (SEpt){};
    \node[fit={(N box)(NEpt)(SEpt)},draw,dashed,circle,xscale=.9] (I disk){};
    \nstrandsalongpath{N box.north}{O disk.north}
    \nstrandsalongpath[]{N box.south}{O disk.south}
    \node[inner sep=0] at ($(NEpt)+(.4,0)$) (NEEpt){};
    \node[inner sep=0] at ($(SEpt)+(.4,0)$) (SEEpt){};
    \node[fit={(I disk)(NEEpt)(SEEpt)},draw,dashed,circle,scale=.8] (O disk){};
    \strandsboxtodisk{N}{O}
    \foreach \stpt/\stdir/\endpt/\enddir in
            {N box.50/90/NEpt.north/90,
            N box.310/-90/SEpt.south/-90,
            O disk.40/-90/O disk.320/90}
        {
        \draw[thick] (\stpt) to[out=\stdir,in=\enddir] (\endpt);
        }
    \draw[thick] (SEpt.south)--(NEpt.north);
    \end{tikzpicture}
\simeq
    \begin{tikzpicture}[baseline=-.5ex]
    \modbox{N}{0}{0}{\q\cN}
    \node[inner sep=0] at ($(N box.north east)+(.4,-.2)$) (NEpt){};
    \node[inner sep=0] at ($(N box.south east)+(.4,.2)$) (SEpt){};
    \node[inner sep=0] at ($(NEpt)+(.4,0)$) (NEEpt){};
    \node[inner sep=0] at ($(SEpt)+(.4,0)$) (SEEpt){};
    \node[fit={(NEpt)(NEEpt)(SEpt)(SEEpt)},draw,dashed,circle,scale=.7] (I disk){};
    \node[fit={(I disk)(N box)},draw,dashed,circle] (O disk){};
    \strandsboxtodisk{N}{O}
        \foreach \stpt/\stdir/\endpt/\enddir in
            {N box.50/90/NEpt.north/90,
            N box.310/-90/SEpt.south/-90,
            O disk.40/-90/NEEpt.north/90,
            O disk.320/90/SEEpt.south/-90,
            NEpt.north/-90/SEpt.south/90,
            NEEpt.north/-90/SEEpt.south/90}
        {
        \draw[thick] (\stpt) to[out=\stdir,in=\enddir] (\endpt);
        }
    \nstrandsalongpath{N box.north}{O disk.95}
    \nstrandsalongpath[]{N box.south}{O disk.265}
    \end{tikzpicture}
\xrightarrow{\Imap\otimes\X\left(\ILhorizsaddle\right)}
    \begin{tikzpicture}[baseline=-.5ex]
    \modbox{N}{0}{0}{\cN}
    \node[inner sep=0] at ($(N box.north east)+(.4,-.2)$) (NEpt){};
    \node[inner sep=0] at ($(N box.south east)+(.4,.2)$) (SEpt){};
    \node[inner sep=0] at ($(NEpt)+(.4,0)$) (NEEpt){};
    \node[inner sep=0] at ($(SEpt)+(.4,0)$) (SEEpt){};
    \node[fit={(NEpt)(NEEpt)(SEpt)(SEEpt)},draw,dashed,circle,scale=.7] (I disk){};
    \node[fit={(I disk)(N box)},draw,dashed,circle] (O disk){};
    \strandsboxtodisk{N}{O}
        \foreach \stpt/\stdir/\endpt/\enddir in
            {N box.50/90/NEpt.north/90,
            N box.310/-90/SEpt.south/-90,
            O disk.40/-90/NEEpt.north/90,
            O disk.320/90/SEEpt.south/-90,
            NEpt.north/-90/NEEpt.south/-90,
            SEpt.north/90/SEEpt.south/90}
        {
        \draw[thick] (\stpt) to[out=\stdir,in=\enddir] (\endpt);
        }
    \nstrandsalongpath{N box.north}{O disk.95}
    \nstrandsalongpath[]{N box.south}{O disk.265}
    \end{tikzpicture}
\]
where $\X\left(\ILhorizsaddle\right)$ indicates the image of a saddle cobordism under the Khovanov spectrum functor.

In this way, given $f\in\HOM_{\Sp\TL_{n+1}}(\cM\horizcomp \Imod_1,\cN)$, we see that $\psi\circ\phi(f)$ is the composition
\[
    \begin{tikzpicture}[baseline=-.5ex]
    \modbox{M}{0}{0}{\cM}
    \emptydisk[.8]{E}{1}{0}
    \node[inner sep=0] at ($(E disk.north east)+(.4,-.2)$) (NEpt){};
    \node[inner sep=0] at ($(E disk.south east)+(.4,.2)$) (SEpt){};
    \node[fit={(M box)(NEpt)(SEpt)},draw,dashed,circle,scale=.9] (O disk){};
    \strandsboxtodisk{M}{O}
    \draw[thick] (O disk.50) to[out=-90,in=90] (NEpt.north) -- (SEpt.south) to[out=-90,in=90] (O disk.310);
    \nstrandsalongpath{M box.north}{O disk.95}
    \nstrandsalongpath[]{M box.south}{O disk.265}
    \end{tikzpicture}
\xrightarrow{\X\left(\ILbirth\right)}
    \begin{tikzpicture}[baseline=-.5ex]
    \modbox{M}{0}{0}{\q\cM}
        \foreach \ptn/\stpt/\xoffset/\yoffset in 
            {NEpt/north east/.1/-.2,
            NEEpt/north east/.5/-.2,
            NEEEpt/north east/.6/-.2,
            SEpt/south east/.1/.2,
            SEEpt/south east/.5/.2,
            SEEEpt/south east/.6/.2}
        {
        \node[inner sep=0] at ($(M box.\stpt) + (\xoffset,\yoffset)$) (\ptn){};
        }
    \node[fit={(M box)(NEpt)(SEpt)},draw,dashed,circle,xscale=.85,yscale=.7] (I disk){};
    \node[fit={(M box)(NEEEpt)(SEEEpt)},draw,dashed,circle] (O disk){};
    \strandsboxtodisk{M}{O}
        \foreach \stpt/\outdir/\endpt/\enddir in
            {NEpt.north/90/NEEpt.north/90,
            SEpt.south/-90/SEEpt.south/-90,
            NEEEpt.north/90/O disk.50/-90,
            SEEEpt.south/-90/O disk.310/90,
            NEpt.north/-90/SEpt.south/90,
            NEEpt.north/-90/SEEpt.south/90,
            NEEEpt.north/-90/SEEEpt.south/90}
        {\draw[thick] (\stpt) to[out=\outdir,in=\enddir] (\endpt);}
        \nstrandsalongpath{I disk.north}{O disk.north}
        \nstrandsalongpath[]{I disk.south}{O disk.south}
    \end{tikzpicture}
\xrightarrow{f}
    \begin{tikzpicture}[baseline=-.5ex]
    \modbox{N}{0}{0}{\q\cN}
        \foreach \ptn/\stpt/\xoffset/\yoffset in 
            {NEpt/north east/.3/-.1,
            NEEpt/north east/.5/-.2,
            SEpt/south east/.3/.1,
            SEEpt/south east/.5/.2}
        {
        \node[inner sep=0] at ($(N box.\stpt) + (\xoffset,\yoffset)$) (\ptn){};
        }
    \node[fit={(NEpt)(SEpt)(NEEpt)(SEEpt)},draw,dashed,circle,scale=.5] (I disk){};
    \node[fit={(M box)(NEEpt)(SEEpt)},draw,dashed,circle] (O disk){};
    \strandsboxtodisk{N}{O}
        \foreach \stpt/\outdir/\endpt/\enddir in
            {N box.50/90/NEpt.north/90,
            N box.310/-90/SEpt.south/-90,
            NEpt.north/-90/SEpt.south/90,
            NEEpt.north/-90/SEEpt.south/90,
            NEEpt.north/90/O disk.50/-90,
            SEEpt.south/-90/O disk.310/90}
        {\draw[thick] (\stpt) to[out=\outdir,in=\enddir] (\endpt);}
    \nstrandsalongpath{N box.north}{O disk.95}
    \nstrandsalongpath[]{N box.south}{O disk.265}
    \end{tikzpicture}
\xrightarrow{\X\left(\ILhorizsaddle\right)}
    \begin{tikzpicture}[baseline=-.5ex]
    \modbox{N}{0}{0}{\cN}
        \foreach \ptn/\stpt/\xoffset/\yoffset in 
            {NEpt/north east/.3/-.2,
            NEEpt/north east/.5/-.2,
            SEpt/south east/.3/.2,
            SEEpt/south east/.5/.2}
        {
        \node[inner sep=0] at ($(N box.\stpt) + (\xoffset,\yoffset)$) (\ptn){};
        }
    \node[fit={(M box)(NEEpt)(SEEpt)},draw,dashed,circle] (O disk){};
    \strandsboxtodisk{N}{O}
        \foreach \stpt/\outdir/\endpt/\enddir in
            {N box.50/90/NEpt.north/90,
            N box.310/-90/SEpt.south/-90,
            NEpt.north/-90/NEEpt.north/-90,
            SEpt.south/90/SEEpt.south/90,
            NEEpt.north/90/O disk.50/-90,
            SEEpt.south/-90/O disk.310/90}
        {\draw[thick] (\stpt) to[out=\outdir,in=\enddir] (\endpt);}
    \nstrandsalongpath{N box.north}{O disk.95}
    \nstrandsalongpath[]{N box.south}{O disk.265}
    \end{tikzpicture}
\]
where we have suppressed the various tensor products with identity maps from the notations.  In each picture, the inner dashed disk indicates the location where the upcoming map is taking place.  The key point is then to notice that $f$ and $\X\left(\ILhorizsaddle\right)$ are being applied in \emph{disjoint disks}, so that they must commute.  Thus the map $\psi\circ\phi(f)$ above which we may write as $\X\left(\ILhorizsaddle\right) \circ f \circ \X\left(\ILbirth\right)$ is equivalent to the composition $f \circ \X\left(\ILhorizsaddle\right) \circ \X\left(\ILbirth\right)$ which itself is equivalent to $f$ since the birth and saddle cobordism is equivalent to an identity cobordism.

Similar reasoning shows that, for any $g\in\HOM_{\Sp\TL_n}(\cM,\q T(\cN))$, we have
\[\phi\circ\psi(g) = \X\left(\ILhorizsaddle\right) \circ g \circ \X\left(\ILbirth\right) = \X\left(\ILhorizsaddle\right) \circ \X\left(\ILbirth\right) \circ g = g\]
as required.
\end{proof}

The chains functor $\chainsfunc\colon \Sp\to \Kom$ also defines a functor from $\Sp\TL_{n}$ to $\TL_n=\Kom(H_{2n}-\mathrm{pmod})$ where $H_{2n}-\mathrm{pmod}$ is the category of projective modules over Khovanov's arc algebra on $2n$ points, and $\Kom(\cdot)$ indicates the category of chain complexes with the same finiteness conditions as in Definition \ref{def:spectral TL cat} (our $\TL_n$ corresponds to $\Kom^+(n)$ in Hogancamp's notation).  It is an important result, \cite[Proposition 4.2]{LLS_tangles}, that, for a tangle $T$, $\chainsfunc(\X(T))$ is quasi-isomorphic to Khovanov's tangle invariant of $T$ (compatibly with gluing \cite{LLS_func}), which we will frequently use without mention.

We note that the functors $(-\horizcomp \Imod_1)\colon \Sp\TL_{n}\to \Sp\TL_{n+1}$ and $T\colon \Sp\TL_{n+1}\to \Sp\TL_n$ are compatible with the chains functor $\mathcal{C}_h$; i.e. there are equivalences $\mathcal{C}_h(-\horizcomp \Imod_1)\to \mathcal{C}_h(-)\horizcomp\mathcal{C}_h(\Imod_1)$ and similarly for $T$.  We then have a commutative diagram, on $\pi_0$:
    \begin{equation}\label{eq:chains-adjunction}
    \begin{tikzpicture}[baseline=(current bounding box.center)]
        \node (a0) at (0,0) {$[\cM\horizcomp \Imod_1,\cN]_{\Sp\TL_{n+1}}$};
        \node (a1) at (6,0) {$[\cM,q T(\cN)]_{\Sp\TL_{n}}$};
       \node (b0) at (0,-2) {$[\mathcal{C}_h(\cM\horizcomp \Imod_1),\mathcal{C}_h\cN]_{\TL_{n+1}}$};
        \node (b1) at (6,-2) {$[\mathcal{C}_h\cM,\mathcal{C}_h(q T(\cN))]_{\TL_{n}}$};

        \draw[->] (a0) -- (a1);
        \draw[->] (b0) -- (b1);
        \draw[->] (a0) -- (b0) node[midway,anchor=east] {\scriptsize $\mathcal{C}_h$};
        \draw[->] (a1) -- (b1) node[midway,anchor=west] {\scriptsize $\mathcal{C}_h$};
    \end{tikzpicture}
    \end{equation}
\begin{remark}\label{rmk:adjunction-chain-compatibility}
    The proof of Theorem \ref{thm:adjunction} is compatible with the argument for chains \cite[Theorem 2.31]{Hog_polyaction} in an even stronger sense, namely that the homotopies used in the proof of the equivalence of Theorem \ref{thm:adjunction}, when passing to chains, agree with the (more tautalogical) homotopies in \cite{Hog_polyaction}.
\end{remark}
\subsection{Duals and mirrors}\label{sec:duals and mirrors}
We recall some material from \cite{LLS_func}.

\begin{definition}
    Let $A$ be a spectral algebra.  A left spectral $A$-module $X$ is \emph{dualizable} if, for all $A$-modules $Z$, the natural map
    \[
\HOM_A(X,A)\otimes_A Z \to \HOM_{A}(X,Z)
    \]
    is a weak equivalence.  There is a similar notion for a right spectral $A$-module $X$, namely that the natural map 
    \[
    Z\otimes_A \HOM_A(X,A)\to \HOM_{A}(X,Z)
    \]
    is a weak equivalence (here $\HOM_A(X,A)$ is the \emph{right} $A$-module morphisms, which inherits a left $A$-module structure by acting by the left on the range $A$).  
\end{definition}

In Section \ref{sec:spTL basics} we considered punctured tangles $T\in\Tang(\varnothing;2n)$ in a disc with $n$ endpoints on the `top' and `bottom' of the boundary.  In this section however we will briefly need to consider punctured tangles $T\in\Tang(m;n)$ in annuli, with $m$ endpoints on the `inner' puncture and $n$ endpoints on the `outer' puncture.

Such an $(m;n)$-tangle $T$ gives rise to an $(\X(m),\X(n))$-bimodule $\X(T,P)$, and \cite[Proposition 5.8]{LLS_func} establishes that $\X(T,P)$ is dualizable over both $\X(m)$ and $\X(n)$.  In addition, by flipping the radial direction in the annulus containing $T$, we obtain an $(n;m)$-tangle $\hat{T}$, the \emph{mirror} of $T$.  See Figure \ref{fig:mnTang and mirror}.

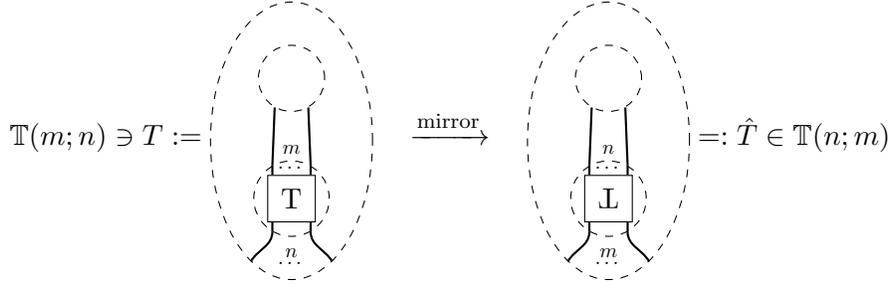
\begin{figure}
\[\Tang(m;n) \ni T:= \begin{tikzpicture}[baseline=-.5ex]
        \emptydisk{E}{0}{.8}
        \straightmodboxindisk{T}{0}{-.8}{\text{T}}
        \node[align=center,fit={(E disk)(T disk)},draw,circle,xscale=.7,yscale=1.2,dashed] (ET) {};
        \draw[thick] (T disk.\disksw) to[out=-90,in=\diskne] (ET.\disksw);
        \draw[thick] (T disk.\diskse) to[out=-90,in=\disknw] (ET.\diskse);
        \draw[thick] (E disk.\disksw)--(T disk.\disknw);
        \draw[thick] (E disk.\diskse)--(T disk.\diskne);
        \nstrandsalongpath[n]{T disk.south}{ET.south}
        \node[above,scale=.7] at (T disk.north) {$m$};
    \end{tikzpicture}
    \quad\xrightarrow{\mathrm{mirror}}\quad
    \begin{tikzpicture}[baseline=-.5ex]
        \emptydisk{E}{0}{.8}
        \straightmodboxindisk{T}{0}{-.8}{\scalebox{1}[-1]{T}}
        \node[align=center,fit={(E disk)(T disk)},draw,circle,xscale=.7,yscale=1.2,dashed] (ET) {};
        \draw[thick] (T disk.\disksw) to[out=-90,in=\diskne] (ET.\disksw);
        \draw[thick] (T disk.\diskse) to[out=-90,in=\disknw] (ET.\diskse);
        \draw[thick] (E disk.\disksw)--(T disk.\disknw);
        \draw[thick] (E disk.\diskse)--(T disk.\diskne);
        \nstrandsalongpath[m]{T disk.south}{ET.south}
        \node[above,scale=.7] at (T disk.north) {$n$};
    \end{tikzpicture}=:\hat{T}\in\Tang(n;m)
\]
\caption{An $(m;n)$-tangle $T$, viewed in the once-punctured disc (annulus), can be mirrored in the radial direction to arrive at an $(n;m)$-tangle $\hat{T}$.}
\label{fig:mnTang and mirror}
\end{figure}

Note that for a $(R,R')$-module spectrum $M$ and a $(R,R)$-module spectrum $N$, that $\HOM_{R}(M,N)$ is a $(R',R)$-module spectrum.  With this preparation, we have the following proposition:

\begin{proposition}[{\cite[Proposition 5.11]{LLS_func}}]\label{prop:lls-mirror}
    Let $T$ be an $(m;n)$-tangle with $N$ crossings and $\hat{T}$ its mirror.  Then there is a weak equivalence
    \[
    \HOM_{\mathscr{X}(m)}(\mathscr{X}(T,P),\mathscr{X}(m))_q\simeq \mathscr{X}^{q+\frac{n-m}{2}}(\hat{T},N-P).
    \]
\end{proposition}

\begin{lemma}[cf. {\cite[Corollary 2.32]{Hog_polyaction}}]
    Let $(-)^\vee$ denote the map (of sets) taking Temperley-Lieb diagrams to their reflection through the $x$-axis, and reversing $q$-grading shifts.  Then, for each $(n,n)$ Temperley-Lieb diagram $\delta$, we have a natural isomorphism
    \begin{equation}\label{eq:TL-adjunction}
    \HOM_{\Sp\TL_n}(A\vertcomp \X(\delta),B)\to \HOM_{\Sp\TL_n}(A,B\vertcomp \X(\delta^\vee)),
    \end{equation}
    for $A,B$ objects of $\Sp\TL_n$.
\end{lemma}
\begin{proof}
    Since $\mathscr{X}(\delta)$ is dualizable, we have
    \[
    B\otimes_{\mathscr{X}(n)}\HOM_{\mathscr{X}(n)}(\X(\delta),\mathscr{X}(n))\simeq \HOM_{{\mathscr{X}(n)}}(\X(\delta),B).
    \]
    Moreover, Proposition \ref{prop:lls-mirror} implies that $\HOM_{\mathscr{X}(n)}(\X(\delta),\mathscr{X}(n))\simeq \mathscr{X}(\delta^\vee).$

    We use the tensor-hom adjunction of ring spectra, as follows.  Let $R,R'$ be ring spectra, and $M$ a $(R,R')$-bimodule, $N$ an $R'$-module, and $P$ an $R$-module.  Then
    \[
    \HOM_R(M\wedge_{R'} N,P)\cong \HOM_{R'}(N,\HOM_{R}(M,P)).
    \]
By definition, $A\vertcomp \X(\delta)$ is $(A,\X(\delta)) \otimes \X(\cF^v_n)$.  We will distinguish the $\X(2n)$-algebras that enter into $\cF^v_n$ as a multimodule, writing them $\X(2n)_i$ with $i=0,1,2$, where $0$ refers to the outer circle of Figure \ref{fig:vertical comp}, $1$ the upper inner circle and $2$ the lower inner circle.   We have
\[
\HOM_{\X(2n)_0}((A,\X(\delta))\otimes \X(\cF^v_n),B)\cong\HOM_{\X(2n)_0}(A\otimes_{\X(2n)_1} (\X(\delta)\otimes_{\X(2n)_2}\X(\cF^v_n)),B).
    \]
    Applying the tensor-hom adjunction, that is:
    \[
    \HOM_{\X(2n)_1}(A,\HOM_{\X(2n)_0}(\X(\delta)\otimes_{\X(2n)_2} \X(\cF^v_n) ,B)).
    \]
    By the gluing rule, $\X(\delta)\otimes_{\X(2n)_2}\X(\cF^v_n)$ is weakly equivalent to $\X(\epsilon)$ for a flat tangle $\epsilon$ obtained by inserting $\delta$ into the $2$-labelled circle of $\cF^v_n$.  It follows that $\X(\delta)\otimes_{\X(2n)_2}\X(\cF^v_n)$ is right dualizable, so:
    \[
    \HOM_{\X(2n)_0}(\X(\delta)\otimes_{\X(2n)_2}\X(\cF^v_n) ,B)= B \otimes_{\X(2n)_0} \HOM_{\X(2n)_0}(\X(\delta)\otimes_{\X(2n)_2}\X(\cF^v_n),\X(2n)_0). 
    \]
    We are using $\X(2n)_0$ as a $(\X(2n)_0,\X(2n)_0)$-bimodule here.  The $\X(2n)_1$-module structure on both sides agree and are induced by that of $\X(\cF^v_n)$.  
    
    By Proposition \ref{prop:lls-mirror}, 
    \[
    \HOM_{\X(2n)_0}(\X(\delta)\otimes_{\X(2n)_2}\X(\cF^v_n),\X(2n)_0)=\X(\hat{\epsilon}),
    \]
    where $\epsilon$ is the tangle described above (note that there is no grading shift in this application of duality).  At this point, we have identified the left-hand side of (\ref{eq:TL-adjunction}) with 
    \begin{equation}\label{eq:almost-TL-adjunction}
    \HOM_{\X(2n)_1}(A, B \otimes_{\X(2n)_0} \X(\hat{\epsilon}) ).
    \end{equation}
    We note that $\hat{\epsilon}$ has its ``outer" boundary naturally identified with the ``inner" boundary of $\epsilon$, so that (roughly) the $\X(2n)_1$-action on $\X(\hat{\epsilon})$ comes from the outer boundary.
    
    Using the gluing formula for the tangles $\hat{\epsilon}$ and $\cF^v_n$, and the definition of $\vertcomp$, we have 
    $B\otimes_{\X(2n)_0}\X(\hat{\epsilon})= B \vertcomp \X(\delta^\vee).$ Thus (\ref{eq:almost-TL-adjunction}) is identified with
    \[
    \HOM_{\X(2n)}(A, B\vertcomp \X(\delta^\vee)),
    \]
    completing the proof.
\end{proof}

\subsection{Spectral projectors}\label{subsec:spectral-projectors}

In this section we provide the definition and formal properties of a (highest weight) spectral projector in $\Sp\TL_n$.  These are all lifts of similar definitions and properties of the categorified projectors in \cite{Cooper-Krushkal}.  Some of these results are also particular instances of the much more general situation of (smashing) Bousfield localizations; we have included (elementary) proofs for some results in this section that can also be deduced in much greater generality using the language of Bousfield localization - see \cite{Lawson-Bousfield} for a survey.

\begin{definition}
The \emph{through-degree} of a flat $(n,n)$-tangle $\delta$, denoted $\tau(\delta)$, is the number of strands with endpoints on opposite ends (top and bottom) of the disk.

We say that a Temperley-Lieb spectrum $\cM\in\Sp\TL_n$ has \emph{through-degree} $<k$, denoted $\tau(\cM)<k$, if $\cM$ is homotopy equivalent to a homotopy colimit of flat-tangle modules $\X(\delta)$ with each $\tau(\delta)< k$ (We allow some of the entries of the homotopy colimit to be $*$).  Note that $\tau(\cM\vertcomp\cN)\leq\min(\tau(\cM),\tau(\cN))$, using that homotopy colimits commute with smash product.
\end{definition}

\begin{definition}
We say that $\cM\in\Sp\TL_n$ \emph{kills turnbacks from above (respectively from below)} if, for all $\cN\in\Sp\TL_n$ with $\tau(\cN)<n$, we have
\[\cN\vertcomp \cM \simeq *\]
(respectively $\cM\vertcomp \cN \simeq *$).  We simply say $\cM$ \emph{kills turnbacks} if $\cM$ kills turnbacks from both above and below.
\end{definition}

As usual, all of this can be said using certain elementary turnback diagram modules.  Let $e$ denote the unique non-identity flat $\TL_2$-diagram consisting of a turnback on top and bottom.

\begin{definition}\label{def:ei turnbacks}
For $i=1,\dots,n-1$, the \emph{$i^\mathrm{th}$ elementary turnback module} is the flat-tangle module $\X(e_i)\in\Sp\TL_n$ where $e_i=\Itang_{i-1}\horizcomp e \horizcomp \Itang_{n-i-1}\in\TL_n$, illustrated in Figure \ref{fig:elem turnback}.
\end{definition}

\begin{figure}
    \[
    e_i := \vcenter{\hbox{
        \begin{tikzpicture}[scale=.8]
        \foreach \i in
                {0,1.5,3.5,5}
            {\draw[thick] (\i,0)--(\i,2);}
        \draw[thick] (2,0)--(2,.5) to[out=90,in=90] (3,.5)--(3,0);
        \draw[thick] (2,2)--(2,1.5) to[out=-90,in=-90] (3,1.5)--(3,2);
        \nstrandsalongpath[i-1]{0,1}{1.5,1}
        \nstrandsalongpath[n-i-1]{3.5,1}{5,1}
        \node[below,scale=.7] at (2,0) {$i$};
        \node[below,scale=.7] at (3,0) {$i+1$};
        \end{tikzpicture}
        }}
    \]
    \caption{The elementary turnback tangle $e_i\in\TL_n$.}
    \label{fig:elem turnback}
\end{figure}

\begin{proposition}\label{prop:killing turnbacks}
    An object $\cM\in\Sp\TL_n$ kills turnbacks from above (respectively from below) if and only if $\cM\vertcomp\X(e_i)\simeq *$ (respectively $\X(e_i)\vertcomp\cM\simeq *$) for all $i=1,\dots,n-1$.
\end{proposition}
\begin{proof}
It is a basic fact that all Temperley-Lieb diagrams $\delta$ with $\tau(\delta)<n$ can be built by concatenating the various $e_i$.  The result follows since homotopy colimits commute with tensor products, and homotopy colimits of diagrams with contractible objects are contractible.
\end{proof}

On the other hand, we also have the following Proposition which reduces turnback killing to statements on chain complexes.

\begin{proposition}\label{prop:turnback killing can be done on chains}
A Temperley-Lieb spectrum $\X\in\Sp\TL_n$ kills turnbacks if and only if its chain complex $\chainsfunc(\X)\in\TL_n$ kills turnbacks as in \cite{Cooper-Krushkal}.
\end{proposition}
\begin{proof}
    The chains functor commutes with tensor products, and a spectrum with contractible chain complex is itself contractible via the Whitehead theorem applied to a basepoint map.
\end{proof}

\begin{definition}\label{def:spectral projector}
A \emph{spectral Cooper-Krushkal projector} (or, when it will not cause confusion, simply a \emph{spectral projector}) is a pair $(\cP_n,\iota)$ consisting of an object $\cP_n\in \Sp\TL_n$ and a morphism $\Imod_n\xrightarrow{\iota} \cP_n$ so that: 
	\begin{enumerate}
		\item The mapping cone $\Cone(\iota)$ has through-degree $\tau(\Cone(\iota))<n$.\label{itm:CK1}
		\item The spectral multimodule $\cP_n$ kills turnbacks (or equivalently $\chainsfunc(\cP_n)$ kills turnbacks; see Proposition \ref{prop:turnback killing can be done on chains}). \label{itm:CK2}
	\end{enumerate}
Similarly, a \emph{Cooper-Krushkal projector} is a pair $(P_n,\iota)$, where $P_n\in \TL_n$, with the analogous properties in chain complexes.
\end{definition}

\begin{corollary}\label{cor:turnback-killing-chains}
    Let $(\mathcal{P}_n,\iota)$ be a pair with $\mathcal{P}_n\in \Sp\TL_n$ and $\iota\colon \Imod_n\to\mathcal{P}_n$.  If $(\mathcal{P}_n,\iota)$ is a spectral projector, then $(\mathcal{C}_h(\mathcal{P}_n),\mathcal{C}_h(\iota))$ is a Cooper-Krushkal projector.
\end{corollary}
\begin{proof}
    The second property in the definition of spectral and (chain) Cooper-Krushkal projectors are identified by Proposition \ref{prop:turnback killing can be done on chains}.  Since mapping cone commutes, up to quasi-isomorphism, with the chains functor, we have that item \eqref{itm:CK1} for spectra implies the analog in chains.  This completes the proof.  
\end{proof}

\begin{remark}
Note that, by definition of through-degree, the spectral multimodule $\Cone(\iota)$ is merely homotopy equivalent to a spectral multimodule built out of non-identity flat-tangle modules.  This is slightly weaker than the definition originally given in \cite{Cooper-Krushkal}.  In \cite{Hog_polyaction}, this is phrased as the difference between being a projector and a \emph{strong projector}.
\end{remark} 

\begin{lemma}\label{lem:END(P) = HOM(I,P)}
Given a spectral projector $(\cP_n,\iota)$, the pullback
\[\HOM(\cP_n,\cP_n)\xrightarrow{\iota^*}\HOM(\Imod_n,\cP_n)\]
is a stable homotopy equivalence.
\end{lemma}
\begin{proof}
    Because the $\HOM$ functor commutes past homotopy colimits (and cones in particular) in the first entry, we can use items \eqref{itm:CK1} and \eqref{itm:CK2} of Definition \ref{def:spectral projector} to see
    \begin{align*}
    \Cone(\iota^*) &\simeq \HOM(\Cone(\iota),\cP_n)\\
    &\simeq \HOM(\hocolim( \X(\delta) ), \cP_n)\quad\text{with $\tau(\delta)<n$ for all $\delta$ in the hocolim}\\
    &\simeq \holim(\HOM(\X(\delta),\cP_n))\quad\text{with $\tau(\delta)<n$ for all $\delta$ in the hocolim}\\
    &\simeq \holim( \HOM(\Imod_n, \cP_n\vertcomp\X(\delta^\vee)) )\quad\text{with $\tau(\delta^\vee)<n$ for all $\delta$ in the hocolim}\\
    &\simeq \holim( \HOM(\Imod_n,*) ) \simeq *,
    \end{align*}
    and thus $\iota^*$ is an equivalence.
\end{proof}

\begin{corollary}\label{cor:END(P) = closure(P)}
There is an equivalence of spectra
\[\HOM(\cP_n,\cP_n)\simeq \q^nT^n(\cP_n).\]
In other words, the endomorphism spectrum of $\cP_n$ is equivalent to the closure (full trace) of $\cP_n$.
\end{corollary}
\begin{proof}
    This follows from Lemma \ref{lem:END(P) = HOM(I,P)} and $n$ applications of Theorem \ref{thm:adjunction}
\end{proof}

\begin{proposition}[Uniqueness]\label{prop:projector unique}
If $(\cP_n,\iota),(\cP'_n,\iota')\in\Sp\TL_n$ are two spectral projectors, then $\cP_n\simeq\cP'_n$.  In fact, there is a homotopy commutative diagram, with a homotopy-equivalence $\phi$ (well-defined up to homotopy) as follows.
\begin{equation}\label{eq:projector-canonical}
\begin{tikzpicture}[baseline={([yshift=-.8ex]current  bounding  box.center)},xscale=1.5,yscale=1]
\node (a0) at (0,0) {$\cI_n$};
\node (a1) at (1,1) {$\cP_n$};
\node (b1) at (1,-1) {$\cP'_n$};

\draw[->] (a0) -- (a1) node[pos=0.5,anchor=south] {\scriptsize
  $\iota$}; 
\draw[->] (a0) -- (b1) node[pos=0.5,anchor=north] {\scriptsize$\iota'$};
\draw[->] (a1) -- (b1) node[pos=0.5,anchor=west] {\scriptsize$\phi$};

\end{tikzpicture}
\end{equation}
\end{proposition}
\begin{remark}
    In fact, the equivalence $\phi$ can be made even more coherent, but we will not need this.
\end{remark}
\begin{proof}
The map $\iota:\Imod_n\rightarrow\cP_n$ gives rise to a map
\[\Imod_n\vertcomp\cP'_n\xrightarrow{\iota\vertcomp \Imap} \cP_n\vertcomp\cP'_n\]
whose cone satisfies
\[\Cone(\iota\vertcomp \Imap) \simeq \Cone(\iota)\vertcomp \cP'_n\simeq *\]
where the last equivalence uses items \eqref{itm:CK1} and \eqref{itm:CK2} of Definition \ref{def:spectral projector}.  The result follows.

To see that we can choose $\phi$ essentially uniquely, consider the following commutative diagram, where the arrows either come from the monoidal structure of $\X(2n)$-modules or the morphisms $\iota,\iota'$:
\begin{equation}\label{eq:iotaiota}
\begin{tikzpicture}[baseline={([yshift=-.8ex]current  bounding  box.center)},xscale=2.5,yscale=1]
\node (a0) at (0,1) {$\cP_n$};
\node (a1) at (1,1) {$\cI_n\vertcomp\cP_n$};
\node (b0) at (-0.5,0) {$\cI_n$};
\node (b1) at (0.5,0) {$\cI_n\vertcomp\cI_n$};
\node (b2) at (1.5,0) {$\cP'_n\vertcomp\cP_n$};
\node (c0) at (0,-1) {$\cP'_n$};
\node (c1) at (1,-1) {$\cP'_n\vertcomp \cI_n$};

\draw[->] (a0) -- (a1); 
\draw[->] (b0) -- (a0);
\draw[->] (b1) -- (a1);
\draw[->] (a1) -- (b2);
\draw[->] (b0) -- (b1);
\draw[->] (b1) -- (b2);
\draw[->] (b0) -- (c0);
\draw[->] (b1) -- (c1);
\draw[->] (c1) -- (b2);
\draw[->] (c0) -- (c1);

\end{tikzpicture}
\end{equation}
Choosing a homotopy inverse for either $\id \vertcomp \iota$ or $\iota'\vertcomp \id$ determines the homotopy class $\phi$; using (\ref{eq:iotaiota}), it follows that the diagram of the Proposition statement is homotopy commutative.  For uniqueness of $\phi$ up to homotopy, suppose that $\phi_1,\phi_2$ both fit into (\ref{eq:projector-canonical}).  Let $\psi_2$ be a homotopy inverse for $\phi_2$.  Then $(\Id-\psi_2\phi_1)\iota\simeq *$ by construction, and Lemma \ref{lem:END(P) = HOM(I,P)} allows us to conclude that $\Id-\psi_2\phi_1$ is nullhomotopic, and so $\phi_1\simeq\phi_2$, as needed.
\end{proof}

\begin{proposition}\label{prop:projectors are idempotents}
If $(\cP_n,\iota)$ is a spectral projector, then $(\cP_n\vertcomp\cP_n,\iota\vertcomp\iota)$ is also a spectral projector.  In particular, $\cP_n\vertcomp\cP_n\simeq\cP_n$.
\end{proposition}
\begin{proof}
It is clear from the monoidal structure that $\cP_n\vertcomp\cP_n$ will kill turnbacks.  The diagram \eqref{eq:iotaiota} shows that $\Cone(\iota\vertcomp\iota) \simeq \Cone(\Id\vertcomp\iota)\simeq\Cone(\iota)$.
\end{proof}

We will also require a similar statement about `absorbing smaller projectors'.

\begin{proposition}\label{prop:projector absorbs smaller projectors}
For any $\ell \leq n$, let $\proj_\ell\in\Sp\TL_\ell$ and $\proj_n\in\Sp\TL_n$ be spectral projectors.  Then
\[\proj_n\vertcomp (\proj_\ell \horizcomp \Imod_{n-\ell}) \simeq \proj_n.\]
That is to say, `projectors absorb smaller projectors'.
\end{proposition}
\begin{proof}
The smaller projector $\proj_\ell$ comes equipped with a map $\Imod_\ell\xrightarrow{\iota_\ell}\proj_\ell$ whose cone is of lesser through-degree.  Therefore the map
\[\proj_n\vertcomp (\proj_\ell \horizcomp \Imod_{n-\ell}) \xrightarrow{ \id\vertcomp (\iota_\ell\horizcomp \id) } \proj_n\vertcomp\Imod_n\]
will be an equivalence since its cone will be contractible (since $\proj_n$ kills turnbacks).
\end{proof}

There are similar equivalences for composing $\proj_n$ with a smaller projector on any of the strands, not just the first $\ell$ of them.  The precise statements are left to the reader.

A similar proof shows that $\proj_n$ is able to `absorb crossings' as in the following proposition.
\begin{proposition}\label{prop:projectors absorb crossings}
Let $\cP_n\in\Sp\TL_n$ be a spectral projector, and let $\X(\sigma_i^{\pm 1})\in\Sp\TL_n$ denote the Khovanov spectrum associated to a braid group generator (crossing) $\sigma_i^{\pm 1}$.
\begin{itemize}
    \item For right-handed crossings, $\cP_n\vertcomp\X(\sigma_i) \simeq \X(\sigma_i)\vertcomp\cP_n \simeq \q^2\Sigma\cP_n$, and
    \item For left-handed crossings, $\cP_n\vertcomp\X(\sigma_i^{-1}) \simeq \X(\sigma_i^{-1})\vertcomp\cP_n \simeq \q\cP_n$.
\end{itemize}
\end{proposition}
\begin{proof}
The equivalences are induced by those in the cofibration sequence of Equation \eqref{eq:crossing cofib seq}, since one of the three terms contains a turnback.
\end{proof}

\section{Construction of the Spectral Projector via Infinite Twists}\label{sec:construction}

In \cite{Wil_colored}, the second author generalized an argument of Rozansky \cite{Rozansky} to show that the Khovanov spectrum of an infinite twist on $n$ strands satisfies the properties of a spectral projector.  At that time the Khovanov spectrum was not yet defined as a spectral multimodule, and so the theorems in \cite{Wil_colored} are stated for infinite twists with specified closures, and the proofs rely on restricting to specific $q$-gradings.  In $\Sp\TL_n$ however, the graded action of $\X(2n)=\sarc_{2n}$ prevents this approach from being directly applicable.

Nevertheless, the arguments in \cite{Wil_colored} can be adapted to the spectral multimodule setting, and we will use them to show that the infinite twist spectrum continues to play the role of a spectral projector in $\Sp\TL_n$.  In the process, we will improve some of the bounds involved.

In diagrams in this and following sections, we will sometimes write $T$ for the Khovanov spectrum $\X(T)$ associated to a tangle $T$, where it will not cause confusion.  In particular we will often write $e_i$ for the Khovanov spectrum $\X(e_i)$ associated to the elementary turnback (Definition \ref{def:ei turnbacks}) $e_i$.

\begin{definition}\label{def:frac twist}
The \emph{left-handed fractional twist spectrum on $n$ strands}, denoted $\T_n$, is the Khovanov spectrum associated to the $(n,n)$ tangle shown in Figure \ref{fig:frac twist} (with grading shift $P=0$).  Superscripts indicate vertical composition:
\[\T_n^m:=\T_n\vertcomp \T_n \vertcomp \cdots \vertcomp \T_n,\]
with $\T_n^0:=\Imod_n$, the spectrum for the identity tangle of vertical lines.  The \emph{left-handed full twist spectrum} is $\FT_n:=\T_n^n$, with similar superscript notation $\FT_n^k:=\T_n^{nk}$; an example is also shown in Figure \ref{fig:frac twist}.

By convention, $\T_0^k:=\sphere$ is the module assigned to the empty diagram.
\end{definition}

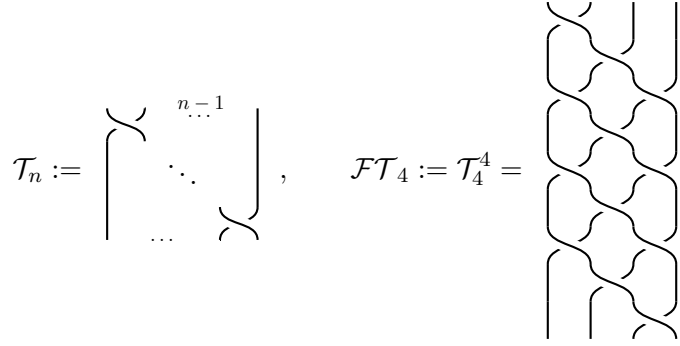
\begin{figure}
    \[
    \T_n:=
    \vcenter{\hbox{
        \begin{tikzpicture}
            \negtwist{0}{0}{2}{-1.75}
        \end{tikzpicture}
        }}
    \,,\quad\quad
    \FT_4:=\T_4^4=\vcenter{\hbox{
        \begin{tikzpicture}[xscale=.2mm,yscale=.175mm]
            \FTfourEX
        \end{tikzpicture}
        }}
    \]
    \caption{The left-handed fractional twist spectrum $\T_n$ is the (Khovanov spectrum associated to) the braid on the left, where one strand passes over the other $n-1$ strands, creating $n-1$ left-handed crossings in the process.  We also show an example of a full twist $\FT_n$ in the case $n=4$.}
    \label{fig:frac twist}
\end{figure}

The cofibration sequence of Proposition \ref{prop:crossing cofib seq} can be redrawn for our left-handed crossings appearing in $\T_n$ as

\begin{equation}\label{eq:cofib seq for negcros}
\X\left( \ILvres ,P\right) \rightarrow
\q^{-1}\X\left( \ILnegcros ,P \right) \rightarrow
\q\Sigma\X\left( \ILhres ,P \right).
\end{equation}

\begin{definition}
\label{def:inf twist}
For any $n\in\N$, the \emph{left-handed infinite twist spectrum} $\T_n^\infty$ is defined as the homotopy colimit of the sequence
\begin{equation}
\label{eq:inf twist seq def}
\T_n^\infty:= 
\hocolim \left( \T_n^0 \rightarrow \q^{-(n-1)}\T_n^1 \rightarrow\cdots \q^{-C(\T_n^m)}\T_n^m\rightarrow\cdots \right)
\end{equation}
where $C(\T_n^m)=m(n-1)$ denotes the number of crossings in $\T_n^m$, and each of the maps is a composition of maps arising from Equation \eqref{eq:cofib seq for negcros}.
\end{definition}

The key to understanding $\T_n^\infty$ comes from understanding the cones of the maps involved in the sequence \eqref{eq:inf twist seq def}.  The arguments are essentially equivalent to those given in \cite{Wil_colored}, but the notation and conventions in that paper are slightly different.  For completeness we provide a streamlined version here.

We begin with a well-known topological lemma on sliding turnbacks through finite twists.  Fix a number of strands $n$.  Let $\eitop,\eibot$ denote the individual turnbacks at the top and bottom of the standard turnback $e_i=\eitop\vertcomp\eibot$.  We also introduce the notational convention 
\[\eitop[0]:=
\ILtikzpic[scale=.6]{
\draw[thick] (0,2)--(0,0);
\draw[thick] (3,2)--(3,0);
\drawover{ (-1,2) to[out=-90,in=180] (0,1) -- (3,1) to[out=0,in=-90] (4,2); }
\nstrandsalongpath[n-2]{0,2}{3,2}
}.
\]

\begin{lemma}\label{lem:pull tb thru twist}
For any $0\leq r<n$ and $0<i<n$, we have
\[\T_n^{nk+r}\vertcomp e_i \simeq \q^{6k(n-1)+s_i} \eitop[i'] \vertcomp\T_{n-2}^{(n-2)k+r_i'} \vertcomp \eibot[i],\]
where $i'=i+r \mod n$ and $r'_i,s_i$ are determined as follows.
\begin{enumerate}
    \item If $i<n-r$, then $r'_i=r$ and $s_i=3r$. \label{it:pull tb case less than}
    \item If $i=n-r$, then $r'_i=r-1$ and $s_i=3r$. \label{it:pull tb case equal}
    \item If $i>n-r$, then $r'_i=r-2$ and $s_i=3(n+r-2)$. \label{it:pull tb case greater than}
\end{enumerate}
\end{lemma}
\begin{proof}
We pull the turnback through twists, recording the Reidemeister moves used.  We first pull through the full twists $\T_n^{nk}$, so that
\begin{align*}
\T_n^{nk+r} \vertcomp e_i &=
\T_n^r \vertcomp \T_n^{nk} \vertcomp \eitop \vertcomp \eibot \\
&\cong \q^{6k(n-1)} \T_n^r \vertcomp \eitop \vertcomp \T_{n-2}^{(n-2)k} \vertcomp \eibot.
\end{align*}
This isotopy consists of $2k$ Reidemeister I moves and $2k(n-2)$ Reidemeister II moves, each of which shifts $q$-grading by $3$.  Note that at this stage, the placement $i$ of the turnback has not changed.

We then pull through the remaining fractional twist $\T_n^r$, and check case by case the three statements above.  In case \eqref{it:pull tb case less than} the turnback does not need to `sweep around the back', and we accomplish the isotopy using only $r$ Reidemeister II moves.  In case \eqref{it:pull tb case greater than} the turnback must `sweep around the back', using $2$ Reidemeister I moves and $n+r-4$ Reidemeister II moves.  Finally, case \eqref{it:pull tb case equal} describes the case where the turnback only sweeps `halfway' around the back, leading to the need for $\eitop[0]$.  This case uses $1$ Reidemeister I move and $r-1$ Reidemeister II moves.  
\end{proof}

\begin{corollary}\label{cor:Cone of Tm to Tm+1}
Fix a number of strands $n$.  Then for any $m=nk+r$ (with $r<n$), the mapping cone
\[\mathcal{C}_{m+1}:=\Cone \left( \q^{-C(\T_n^{m})}\T_n^{m}\rightarrow \q^{-C(\T_n^{m+1})}\T_n^{m+1} \right)\]
is equivalent to a homotopy colimit
\[\mathcal{C}_{m+1}\simeq \hocolim \left(\basedcube^{n-1}\rightarrow \Sp\TL_n\right)\]
where every non-trivial term is of the form
\begin{equation}\label{eq:simplified terms in twist cone}
\q^{a_i-C(\T_n^{m+1})+6k(n-1)+s_i} \, \eitop[i'] \vertcomp\T_{n-2}^{(n-2)k+r_i'} \vertcomp \eibot[i] \vertcomp \delta
\end{equation}
for some flat $(n,n)$-tangle module $\delta$ and some $a_i\geq i-1$ (here $r_i',s_i$ are as defined in Lemma \ref{lem:pull tb thru twist}).

In particular $\mathcal{C}_{m+1}$ is equivalent to a homotopy colimit of shifted flat tangle modules $\q^{b_\epsilon}\Sigma^{c_\epsilon}\epsilon$ with a global bound on $\q$-grading shifts
\begin{equation}\label{eq:twist cone minimal q degree}
b_\epsilon \geq B_{m+1}:= m+nk+1-n.
\end{equation}
\end{corollary}
\begin{proof}
The desired cone $\mathcal{C}_{m+1}$ is a hocolim taken over the cube of resolutions for $\T_n$ (but with the identity term replaced by a trivial basepoint term) concatenated with the remaining $\T_n^m$.  Thus every non-trivial term in this hocolim has the form
\[\q^{a_i-C(\T_n^{m+1})} \T_n^{nk+r} \vertcomp \X(e_i) \vertcomp \X(\delta),\]
where $a_i\geq i-1$ (because the $i-1$ crossing resolutions `to the left' of $e_i$ were all 0-resolutions) and $\delta$ is some flat $(n,n)$-tangle.  We then apply Lemma \ref{lem:pull tb thru twist} to arrive at terms of the form \eqref{eq:simplified terms in twist cone}.

To arrive at the global bound for minimal $\q$-shifts appearing amongst all such terms, we consider that for any tangle $T$, the minimal $\q$-shift appearing in the cube of resolutions for $\X(T)$ is precisely $C(T)$, the number of crossings in $T$; see Equation \eqref{eq:crossing mapping cone}.  Therefore the minimum shift appearing for a term of the form \eqref{eq:simplified terms in twist cone} will be
\[\q_{\min}(i):= a_i-C(\T_n^{m+1})+6k(n-1)+s_i + C(\T_{n-2}^{(n-2)k+r_i'}) +C(\eitop[i']).\]
(Note that $\eibot$ has no crossings, but $\eitop[i']$ may have crossings.)  We simplify this expression down to
\[\q_{\min}(i)= 2nk-2r -(n-3)(r-r_i') +C(\eitop[i']) + s_i +a_i +1-n,\]
and then take the minimum across all $i$ using Lemma \ref{lem:pull tb thru twist}, while noting that $a_i\geq i-1\geq 0$ and $C(\eitop[i'])=n-2$ precisely in Case \eqref{it:pull tb case equal} of Lemma \ref{lem:pull tb thru twist} (and is zero otherwise), to show that
\[\min_i(\q_{\min}(i)) \geq 2nk+r+1-n = B_{m+1}\]
as desired.
\end{proof}

The next two theorems provide the proof of the first portion of Theorem \ref{intro thm: spectral Pn is tori} from the Introduction; the second portion of Theorem \ref{intro thm: spectral Pn is tori} about endomorphisms then follows from Corollary \ref{cor:END(P) = closure(P)}.

\begin{theorem}
\label{thm:inf twist is projector}
For any $n\in\N$, the infinite twist spectrum $\T_n^\infty$ is a spectral projector.
\end{theorem}
\begin{proof}
We can describe $\T_n^\infty$ as an iterated homotopy colimit in the following way.  We will write $\mathcal{S}^m$ for certain inductively-constructed spectra which are homotopy-equivalent to $\q^{-C(\T_n^{m})} \T_n^m$, but constructed with fewer cells.  Let $\mathcal{S}^1=\q^{-C(\T_n^{1})}\T_n^1$, and say that $\mathcal{S}^\ell$ have been constructed for $\ell\leq m$.  Using the notation of Corollary \ref{cor:Cone of Tm to Tm+1}, we have the exact triangle
\[\q^{-C(\T_n^{m})}\T^m_n \rightarrow \q^{-C(\T_n^{m+1})}\T^{m+1}_n\rightarrow \mathcal{C}_{m+1}\]
implying that $\q^{-C(\T_n^{m+1})}\T^{m+1}_n$ is homotopy-equivalent to the cone of a map
\[
\mathcal{C}_{m+1}\xrightarrow{\psi_m} \T^m_n.
\]
By Corollary \ref{cor:Cone of Tm to Tm+1}, there is a spectrum $\mathcal{C}'_{m+1}$ homotopy-equivalent to $\mathcal{C}_{m+1}$, but which is a cell complex all of whose cells have $q$-shift greater than $B_{m+1}$. Write $\mathcal{C}_{m+1}'\xrightarrow{\psi_m'} \mathcal{S}^m$ for the map obtained from $\psi_m$ by pre- and post-composition by the equivalences $\mathcal{C}_{m+1}'\simeq \mathcal{C}_{m+1}$ and $q^{-C(\T^m_n)}\T_n^m\simeq \mathcal{S}^m$.  We define $\mathcal{S}^{m+1}$ to be the cone of $\psi_m'$; it follows from the definitions that $\q^{-C(\T_n^{m+1})}\T^{m+1}_n\simeq \mathcal{S}^{m+1}$.  


Thus we have that $\mathcal{S}^{m+1}$ is obtained from $\mathcal{S}^m$ by attaching a finite collection of cells with $q$-shifts at least $B_{m+1}$.  Then because $B_{m+1}\rightarrow\infty$ as $m\rightarrow\infty$, iterating this procedure shows that $\mathcal{S}^\infty\simeq\T_n^\infty$ is obtained as an iterated attaching as in Definition \ref{def:spectral TL cat}.
\end{proof}


Compare the following theorem about stabilization of the underlying spectra in the sequence \eqref{eq:inf twist seq def} to \cite[Proposition 3.8]{Wil_colored} and particularly \cite[Theorem 1.1]{Wil_TorusLinks}, which gives a similar but slightly less sharp bound on the amount of twisting required.

\begin{theorem}
\label{thm:inf twist seq stabilizes}
Given a fixed $j\in\Z$, define the symbols $m_0,k_0,$ and $r_0<n$ by
\[m_0=nk_0+r_0 := \ceil*{\frac{j}{2}}.\]
Then for any $m\geq m_0+\min(r_0,n-r_0)$, the map
\[\q^{-C(\T_n^m)}\T_n^m\atq{j}\rightarrow \q^{-C(\T_n^{m+1})}\T_n^{m+1}\atq{j}\]
in the sequence \eqref{eq:inf twist seq def} is an equivalence of underlying spectra in $q$-degree $j$.  In particular we have
\begin{equation}
\label{eq:first stable twist}
\T_n^\infty\atq{j} \simeq \T_n^\mu\atq{j+C(\T_n^\mu)} \quad \text{for} \quad\mu=\mu(j):=m_0+\min(r_0,n-r_0),
\end{equation}
so that $\T_n^\infty\atq{-n}\simeq\mathbb{S}$ and $\T_n^\infty\Big\rvert_{q<-n}\simeq *$.
\end{theorem}
\begin{proof}
The proof follows from the bound of Equation \eqref{eq:twist cone minimal q degree}, together with the fact that a flat $(n,n)$-tangle module $\epsilon$ can have minimal non-empty $q$-grading $-n$ (corresponding to closures of $\epsilon$ having at most $n$ circles, which could all be labelled by $1$).
\end{proof}


\subsection{A smaller model for $\proj_2$}
\label{sec:P2 simplified}
Although the construction above for $\T_2^\infty$ can be used to build $\proj_2$, experts often use the following simplified expression in the homological setting (this diagram should be interpreted as a chain complex in the (chain) Temperley-Lieb category \cite{Hog_polyaction}):
\begin{equation}\label{eq:P2 simp homological}
P_2 \simeq \left( \cdots
\xrightarrow{\ILhresDotT \, - \, \ILhresDotB}
\q^5\ILhres[scale=.5]
\xrightarrow{\ILhresDotT \, + \, \ILhresDotB}
\q^3\ILhres[scale=.5]
\xrightarrow{\ILhresDotT \, - \, \ILhresDotB}
\q\ILhres[scale=.5]
\xrightarrow{\ILvertsaddle}
\ILvres[scale=.5]
\right).
\end{equation}

In this section we will show that this smaller model continues to make sense for the spectral projector $\proj_2$.  In this setting however, because we a priori need to keep track of higher nullhomotopies between various compositions of these maps, we will phrase things in terms of semisimplicial objects.  To be more precise, we will construct homotopy coherent diagrams from $\Delta_{inj}^{op}$ to spectra that are the analog of \eqref{eq:P2 simp homological}.  Write $\Delta_{inj}$ for the semisimplex category, whose objects are the finite totally ordered sets $[i]=\{0,1,\ldots,i\}$, together with order-preserving injections as morphisms.

\begin{definition}\label{def:P2 simplified}
A \emph{simplified infinite 2-twist functor} $F_2^\infty$ is a homotopy coherent diagram $\Delta_{inj}^{op}\xrightarrow{F_2^\infty} \Sp\TL_2$ satisfying the following properties.  We will abuse notation somewhat and refer to homotopy coherent diagrams from $\Delta_{inj}^{op}$ to spectra as \emph{semisimplicial spectra}.
\begin{enumerate}

    \item \label{it:P2 vertices} At each simplicial level $m\geq 0$ we have
        \[ F_2^\infty ([m]) = \begin{cases} \Imod_2 & \text{if $m=0$}\\
                        \q^{2m-1} e_1 & \text{otherwise}\\
        \end{cases}\]
    \item \label{it:P2 edge maps} The `zero-th' face maps on each level $[m+1]\xrightarrow{d_0} [m]$ are defined to be
    \[ F(d_0) = \begin{cases}
            \q e_1 \xrightarrow{\ILvertsaddle} \Imod_2 & \text{if $m=0$}\\
            \q^{2m+1}e_1 \xrightarrow{\ILhresDotT \, - \, \ILhresDotB} \q^{2m-1}e_1 & \text{if $m=2n+1$ for $n\geq0$}\\
            \q^{2m+1}e_1 \xrightarrow{\ILhresDotT \, + \, \ILhresDotB} \q^{2m-1}e_1 & \text{if $m=2n$ for $n>0$}
        \end{cases}\]
    All other face maps $F(d_i)$ for $i\neq 0$ are trivial.
    \item \label{it:P2 nullhomotopies} For each $m\geq 0$, the composition $F_2^\infty(m+2) \xrightarrow{F(d_0)} F_2^\infty(m+1) \xrightarrow{F(d_0)} F_2^\infty(m)$ gives a difference of two identical cobordism maps $\psi$.  The simplicial relations require a null-homotopy between such a composition and the trivial $d_0\circ d_1=*$, which we obtain as a tensor product of $\psi$ with a choice of nullhomotopy of $\sphere\xrightarrow{(1-1)}\sphere$.  There are no higher coherences to build between various nullhomotopies used in 3-fold compositions or more, because the relevant mapping spectra are contractible, on account of the $q$-shifts (to see this, use Theorem \ref{thm:adjunction}).  
\end{enumerate}
\end{definition}

Simplified infinite 2-twist functors are not unique, due to the choices of nullhomotopies indicated in Item \eqref{it:P2 nullhomotopies} of Definition \ref{def:P2 simplified}.  We do not keep track of these choices in the notation due to the following theorem.

\begin{theorem}\label{thm:P2 simp}
The 2-strand spectral projector is stably homotopy equivalent to the geometric realization of any choice of simplified infinite 2-twist functor:
\[\proj_2\simeq \left|F_2^\infty\right|.\]
Visually we represent this as
\begin{equation}\label{eq:simp P2 seq def}
\proj_2\simeq \left( \cdots
\xrightarrow{\ILhresDotT \, - \, \ILhresDotB}
\q^5\ILhres[scale=.5]
\xrightarrow{\ILhresDotT \, + \, \ILhresDotB}
\q^3\ILhres[scale=.5]
\xrightarrow{\ILhresDotT \, - \, \ILhresDotB}
\q^1\ILhres[scale=.5]
\xrightarrow{\ILvertsaddle}
\ILvres[scale=.5]
\right).
\end{equation}
\end{theorem}

Note that in Equation \eqref{eq:simp P2 seq def} the right-hand side should not be interpreted as a homotopy colimit of only the maps pictured.
\begin{proof}
We have a natural map
\[\Imod_2 = F_2^\infty([0]) \rightarrow \left|F_2^\infty\right|\]
whose cone consists of the terms $F_2^\infty([m])$ for $m>0$, each of which has through-degree zero, and thus $\left|F_2^\infty\right|$ satisfies Item \eqref{itm:CK1} of Definition \ref{def:spectral projector} for a spectral projector.  Meanwhile, on chains we see that $\chainsfunc(\left|F_2^\infty\right|) \simeq \left|\chainsfunc(F_2^\infty)\right|$ recovers $P_2$ as in Equation \eqref{eq:P2 simp homological}, which kills turnbacks.  Thus $\left| F_2^\infty \right|$ is a spectral projector on 2 strands, which is unique via Proposition \ref{prop:projector unique}.
\end{proof}

\subsection{A smaller model for finite twisting $\T_2^k$}

Although Theorem \ref{thm:P2 simp} provides the simplified model for $\proj_2\simeq \T_2^\infty$, it is also desirable to have correspondingly simple models for the finite twists $\T_2^k$ which `build' $\proj_2$.  The simplification of $\T_2^2$ in particular will also play a special role in Section \ref{sec:CK-recursion}.

Compare the following to Definition \ref{def:P2 simplified}.

\begin{definition}\label{def:T2k simplified}
For each $k\in\N$, a \emph{simplified 2-twist functor} is a homotopy coherent diagram $\basedcube^k\xrightarrow{F_2^k}\Sp\TL_2$ defined as follows.
\begin{enumerate}
    \item\label{it:T2k simp terms}  Non-basepoint objects (vertices) $v\in\basedcube^k$ are assigned spectra as follows:
        \[ F_2^k (v) = \begin{cases} \Imod_2 & \text{if $v=1^k$}\\
                        \q^{2t-1} e_1 & \text{if $v=0^t 1^{k-t}$ for some $1\leq t\leq k$}\\
                        * & \text{otherwise}
        \end{cases}\]

    \item\label{it:T2k simp edges}  Edge maps of the form $0^t1^{k-t} \xrightarrow{f_t} 0^{t-1}1^{k-t+1}$ are assigned maps as follows
    \[ F_2^k (f_t) = \begin{cases}
            \q e_1 \xrightarrow{\ILvertsaddle} \Imod_2 & \text{if $t=1$}\\
            \q^{2t-1}e_1 \xrightarrow{\ILhresDotT \, - \, \ILhresDotB} \q^{2t-3}e_1 & \text{if $t=2m$ for $m>0$}\\
            \q^{2t-1}e_1 \xrightarrow{\ILhresDotT \, + \, \ILhresDotB} \q^{2t-3}e_1 & \text{if $t=2m+1$ for $m>0$}
        \end{cases}\]
    Other edge maps are trivial.
    
    \item\label{it:T2k simp square faces} Two dimensional subcubes of the form
    \[\begin{tikzcd}
        \cdots 00\cdots \ar[r,"f"] \ar[d] & \cdots 01 \cdots \ar[d,"g"] \\
        \cdots 10\cdots \ar[r] & \cdots 11\cdots
    \end{tikzcd}\]
    must have the vertex $\cdots 10 \cdots$ assigned a basepoint.  So long as the other three vertices are non-trivial, the composition $g\circ f$ is canonically homotopic to a difference of two identical cobordism maps $\psi$.  A null-homotopy of this map is obtained as a tensor product of $\psi$ with a choice of nullhomotopy of $\sphere\xrightarrow{(1-1)}\sphere$.  
    This nullhomotopy allows us to `fill' the corresponding 2-dimensional subcube.  All other 2-dimensional subcubes are filled with trivial homotopies.

    \item\label{it:T2k simp no higher nullhomotopies} For any subcube of dimension $d>2$, the $d$-fold compositions of the various edges from the initial vertex to the final vertex canonically give the null map (by looking at $q$-degree).  All $d>2$ dimensional subcubes are filled with trivial homotopies (in fact, by $q$-degree, these are the only allowed fillings).  
\end{enumerate}
See Figure \ref{fig:simplified T22} for a diagram illustrating $F_2^2$.
\end{definition}

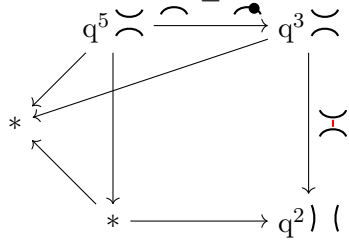
\begin{figure}
\[\begin{tikzpicture}[scale=1.3]

    \node (base) at (-1,-1) {$*$};
    \node (term 0) at (0,0){$\q^5\,\hres$};

    \node (term 1) at (2,0) {$\q^3\,\hres$};

    \node (term 2) at (0,-2) {$*$};

    \node (term 3) at (2,-2) {$\q^2\,\vres$};

    \savebox{\mybox}{$\ILhresDotT \, - \, \ILhresDotB$}
    \draw[->] (term 0) -- (term 1) node[pos=.5,above] {\usebox\mybox};

    \draw[->] (term 0) -- (term 2);
    \draw[->] (term 2) -- (term 3);

    \savebox{\mybox}{$\ILvertsaddle$}
    \draw[->] (term 1) -- (term 3) node[pos=.5,right] {\usebox\mybox};

    \draw[->] (term 0) -- (base);
    \draw[->] (term 1) -- (base);
    \draw[->] (term 2) -- (base);
\end{tikzpicture}\]
    \caption{The diagram for $\basedcube^2\xrightarrow{F_2^2}\Sp\TL_2$.  The composite of the two nontrivial maps in the diagram is the difference of two identical dotted saddle maps $\ILhres\xrightarrow{\psi}\ILvres$.  A null homotopy of this composite is obtained as a tensor product of a null homotopy of $\sphere \xrightarrow{(1-1)}\sphere$ with the cobordism map $\psi$.}
    \label{fig:simplified T22}
\end{figure}

As before, simplified 2-twist functors $F_2^k$ are not unique due to the choices of nullhomotopies in Item \ref{it:T2k simp square faces}.  We continue to suppress these choices from the notation since they will not effect the homotopy type of the homotopy colimit.

\begin{lemma}\label{lem:simp P2 induction}
Fix a simplified 2-twist functor $F_2^k$.  Then for any choice of simplified 2-twist functor $F_2^{k+1}$ which matches $F_2^k$ on the subcube $\cube^k\times\{1\}\subset\basedcube^{k+1}$ we have
\begin{equation}\label{eq:simp P2 induction}
\hocolim\, F_2^{k+1} \simeq (\hocolim\, F_2^k) \vertcomp \T_2^1.
\end{equation}
In particular we have maps
\begin{equation}\label{eq:simp P2 seq maps}
\hocolim F_2^k \rightarrow \hocolim F_2^{k+1}
\end{equation}
induced by the usual map for a crossing $\Imod_2\rightarrow \T_2^1$.
\end{lemma}
\begin{proof}

Fixing a choice of simplified 2-twist functor $F_2^k$, we can write the right hand side of Equation \eqref{eq:simp P2 induction} as a mapping cone
\[(\hocolim\, F_2^k) \vertcomp \T_2^1 \simeq \Cone( \hocolim\, F_2^k\vertcomp e_1 \xrightarrow{s} \hocolim\, F_2^k),\]
where we view $F_2^k\vertcomp e_1$ as a homotopy coherent diagram $\basedcube^k\rightarrow \Sp\TL_2$ and the map $s$ as being induced by a natural transformation $F_2^k\vertcomp e_1 \xrightarrow{s} F_2^k$.  Now we define an auxiliary functor $\basedcube^k\xrightarrow{E_2^k}\Sp\TL_2$
satisfying
\[E_2^k(v) = \begin{cases}
    \q^{2k+1} e_1 & \text{if $v=0^k$}\\
    * & \text{otherwise}
\end{cases},\]
with every morphism trivial.  Our goal is to build and study a natural transformation $E_2^k\xrightarrow{\rho^k}F_2^k\vertcomp e_1$ which induces a weak equivalence on hocolims.

The map $\rho^k$ is based upon iterated Reidemeister I equivalences.  On the only non-trivial object of $E_2^k$ we define $\rho^k$ as
\[E_2^k( 0^k) = \q^{2k+1} e_1 \xrightarrow{ \rho^\bullet + (-1)^{k} \rho_\bullet} \q^{2k-1} e_1\vertcomp e_1 = F_2^k(0^k)\vertcomp e_1,\]
where $\rho^\bullet$ consists of an undotted birth together with a dot on the `top sheet', while $\rho_\bullet$ consists of a dotted birth.  In a slight abuse of notation, we will use $\rho^k$ to indicate this sum/difference of cobordism maps as well.  See Figure \ref{fig:simp P2 induction} for the case of $\rho^2$.

Unlike in the case of the natural transformation for $s$ (and also unlike in the homological setting), we require a new homotopy to make the diagram for $\rho^k$ coherent.  For $k>1$ the composition
\[E_2^k( 0^k) = \q^{2k+1} e_1 \xrightarrow{ \rho^k} \q^{2k-1} e_1\vertcomp e_1 \xrightarrow{ F_2^k(f_k)} \q^{2k-3} e_1 \vertcomp e_1 = F_2^k(0^{k-1}1)\vertcomp e_1\]
is not trivial on the nose, but instead gives a sum of four dotted cobordism maps.  Two of them have two dots on single sheets, and are thus trivial (not just null-homotopic).  The other two will be the same dotted cobordism $\rho^\bullet_\bullet$ (consisting of a dotted birth together with a dot on the top sheet), but with differing signs.  Thus we can give a null-homotopy for this composition which consists of tensoring $\rho^\bullet_\bullet$ with any choice of null-homotopy for $\sphere\xrightarrow{(1-1)}\sphere$.  
Similarly when $k=1$, the composition
\[E_2^1( 0) = \q^{3} e_1 \xrightarrow{ \rho^1} \q e_1\vertcomp e_1 \xrightarrow{ F_2^1(f_1)} e_1 = F_2^1(1)\vertcomp e_1\]
consists of two identical dotted cobordisms with opposite sign, which can be null-homotoped in the same fashion.  Note that no choice of the new null-homotopies described here involves the `lower' copy of $e_1$.

Finally, we define $\rho^k$ as a morphism $E_2^k\to F_2^k\vertcomp e_1$, that is to say, a homotopy-coherent diagram 
\[
\basedcube^{k}\times \basedcube^1 \xrightarrow{\rho^k}\Sp\TL_2
\]
restricting to $E_2^k$ on $\basedcube^k\times \{0\}$ and $F_2^k\vertcomp e_1$ on $\basedcube^k \times \{1\}$.  For this, we need to define $\rho^k$ on all higher-dimensional faces.  However, the composites of length $1$ and $2$ maps constructed so far are identically null, and we choose the homotopies among these to be trivial as well, giving a definition of $\rho^k$ (again, there are no further choices, using the $q$-degree).

Verifying that $\rho^k$ induces an isomorphism on homology (of hocolims) is equivalent to performing $k$ Reidemeister I moves, and is a computation well-known to experts.  Whitehead's theorem then implies that $\rho^k$ induces a weak equivalence, allowing us to write 
\[(\hocolim\, F_2^k)\vertcomp \T_2^1 \simeq \Cone(\hocolim\, E_2^k \xrightarrow{s\circ\rho^k}\hocolim\, F_2^k).\]

From here, we note that this mapping cone is equivalent to taking the hocolim of a homotopy coherent diagram $\basedcube^{k+1} \xrightarrow{G} \Sp\TL_2$ defined using $E_2^k,F_2^k$ and $s\circ\rho^k$ as in Proposition \ref{prop:cones of hocolims} (using also Item \ref{it:cofinal preserves hocolim} of Proposition \ref{prop:hocolim properties}).  We claim that, for any choice of null-homotopy $\rho^k$ above, this functor $G$ \emph{is} a desired choice of simplified 2-twist functor $F_2^{k+1}$.  We can see directly that the non-trivial morphism $G(f_{k+1})=e_1\xrightarrow{s\circ\rho^k} e_1$ is precisely the map required for $F_2^{k+1}(f_{k+1})$, while the new null-homotopy for $F_2^k(f_k)\circ\rho^k$, which was disjoint from the `lower' copy of $e_1$, combines with the identity homotopies induced by $s$ to give a null-homotopy for $G$ of the correct form.  All of the other `length-2' null-homotopies for $G$ are directly inherited from $F_2^k$ as desired.   Finally, all higher compositions continue to be trivial (since they all involve at least one sheet with two or more dots).
\end{proof}

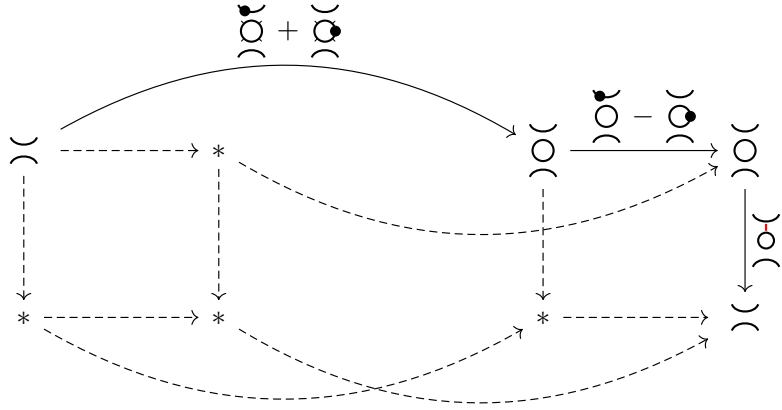
\begin{figure}
\[
\savebox{\mybox}{$\ILrhotop \, + \, \ILrhobot$}
\savebox{\myotherbox}{$\ILhreshrestop \, - \, \ILhreshresmid$}
\savebox{\mythirdbox}{$\ILhreshressad$}
\begin{tikzcd}[sep=huge]
\ILhres \ar[d,dashed] \ar[r,dashed] \ar[rrr,bend left,"\usebox\mybox"] & * \ar[d,dashed] \ar[rrr,bend right,dashed]
& &
\ILhreshres \ar[r,"\usebox\myotherbox"] \ar[d,dashed] & \ILhreshres \ar[d,"\usebox\mythirdbox"] \\
* \ar[r,dashed] \ar[rrr,bend right,dashed] & * \ar[rrr,bend right,dashed]
& &
* \ar[r,dashed] & \ILhres
\end{tikzcd}\]
    \caption{A diagram for the natural transformation $E_2^2 \xrightarrow{\rho^2} F_2^2\vertcomp e_1$.  The left square is $E_2^2$, and the right square is $F_2^2\vertcomp e_1$; the extra basepoint in $\basedcube^2$ is omitted to avoid clutter.  The dashed arrows are all trivial.  The curved arrows assemble to form the natural transformation $\rho^2$.  The tick marks on the circles indicate births.  There is one new non-trivial homotopy: a null-homotopy for the composition of the top curved solid with the top solid arrow.  This is obtained as a tensor product of a null-homotopy of $\sphere\xrightarrow{(1-1)}\sphere$ with an identity cobordism with one dot on the top sheet, and one dot on the circle.}
    \label{fig:simp P2 induction}
\end{figure}

\begin{theorem}\label{thm:T2k simp}
Any simplified 2-twist functor $F_2^k$ satisfies
\[\hocolim(F_2^k)\simeq \T_2^k.\]
Moreover, any choice of simplified infinite 2-twist functor $F_2^\infty$ gives rise to choices for each $F_2^k$ allowing the equivalences above to commute with the maps of Equations \eqref{eq:inf twist seq def} and \eqref{eq:simp P2 seq maps}.  That is, there is the following homotopy-coherent commutative diagram, where the rows have already been defined:
\[\begin{tikzcd}
    \q^{-1}\T_2^1 \ar[r] & \q^{-2}\T_2^2 \ar[r] & \q^{-3}\T_2^3 \ar[r] & \cdots\\
    \q^{-1} \hocolim F_2^1 \ar[u,"="] \ar[r] & \q^{-2}\hocolim F_2^2 \ar[u,"\simeq"] \ar[r] & \q^{-3}\hocolim F_2^3 \ar[u,"\simeq"] \ar[r] & \cdots
\end{tikzcd}.\]
In this way we can recover $\proj_2$ as
\begin{equation}\label{eq:P2 as hocolim of simplified twists}
\proj_2 \simeq \hocolim( \q^{-1}\hocolim F_2^1 \rightarrow \q^{-2}\hocolim F_2^2 \rightarrow \cdots).
\end{equation}

\end{theorem}

\begin{proof}
We have $\T_2^1=\hocolim F_2^1$ by definition.  From here, we can inductively build the desired equivalences using Lemma \ref{lem:simp P2 induction}.  The maps of Equations \eqref{eq:inf twist seq def} and \eqref{eq:simp P2 seq maps} are both defined via the exact triangle for the `new' crossing in $\T_2^1$, and thus the diagram shown will inductively commute.

At each inductive step in this process, we choose a new nullhomotopy (in the proof of Lemma \ref{lem:simp P2 induction} this is described as a choice of nullhomotopy for $F_2^k(f_k)\circ\rho^k$, which in turn determines a new nullhomotopy in the construction of $G=F_2^{k+1}$; all of the other nullhomotopies for our new $G=F_2^{k+1}$ match those of $F_2^k$).  We can see by inspection that these choices are precisely the choices of nullhomotopies in Item \eqref{it:P2 nullhomotopies} of Definition \ref{def:P2 simplified} for a simplified infinite 2-twist functor.  In this way, any choice of $F_2^\infty$ completely determines this inductive process choosing all of the finite $F_2^k$ which fit into the commuting diagram.  The last claim follows.

\end{proof}

Let us denote the right-hand side of Equation \eqref{eq:P2 as hocolim of simplified twists} by $\simpT_2^\infty$, which comes equipped with a natural non-negative filtration induced by the homotopy colimit.  We will denote this filtration by $\CKfilt$ in parallel with the Cooper-Krushkal filtration considered in Section \ref{sec:prove Pn simp}.  We can interpret Theorem \ref{thm:T2k simp}, together with the proof of Lemma \ref{lem:simp P2 induction}, as indicating that the filtration levels satisfy
\[\CKfilt^i(\simpT_2^\infty)\simeq \q^{-i}\T_2^i,\]
with associated graded levels (for $i>1$)
\[\CKfilt^i/\CKfilt^{i-1}(\simpT_2^\infty) = \Sigma \hocolim E_2^{i-1} \simeq \Sigma^{i} \q^{2i-1} e_1,\]
and attaching maps as determined by Equation \eqref{eq:simp P2 seq def} (here the notation $E_2^{i-1}$ comes from the proof of Lemma \ref{lem:simp P2 induction}).

We end this section with a recasting of Lemma \ref{lem:simp P2 induction} and Theorem \ref{thm:T2k simp} for the full twist $\T_2^2$ in particular.

\begin{corollary}\label{cor:T22 simplified as multicone}
There is a natural homotopy-coherent diagram, where the top row is the cofibration sequence \eqref{eq:crossing cofib seq} for the bottom crossing of the full twist spectrum $\T_2^2$:
\[\begin{tikzcd}
\T_2^2 \ar[r] & \Sigma\,\q^2\T_2^1\vertcomp e_1 \ar[d,"\simeq"] \ar[r,"\Sigma(\Id\vertcomp s)"] &
\Sigma\, \q\T_2^1\vertcomp\Imod\ar[d,"\simeq"] \\
&  \Sigma\Cone\left(\q^5 e_1 \rightarrow *\right)  \ar[r,"\Sigma({w,*})"] &
\Sigma\Cone\left(\q^3 e_1 \rightarrow \q^2\Imod\right)
\end{tikzcd}\]
where $s$ denotes the saddle map, and $(w,*)$ is the map induced on cones by the homotopy coherent square in Figure \ref{fig:simplified T22} (with $w$ denoting the difference of dotted identity maps).  In particular, $\T_2^2 \simeq \Cone(\Sigma(w,*))$.  
\end{corollary}
\begin{proof}
    This is a repackaging of Lemma \ref{lem:simp P2 induction} in the case $k=1$, as used to prove Theorem \ref{thm:T2k simp} for $\T_2^2$.
\end{proof}

\section{A Smaller Model for the Spectral Projector}\label{sec:CK-recursion}

In this section we prove Theorem \ref{intro thm:Pn simplified}.  The goal is to build a smaller model for the spectral projector $\proj_n$ out of diagrams involving $\proj_{n-1}$ as in the Cooper-Krushkal sequence \cite{Cooper-Krushkal} of Figure \ref{fig:CK sequence} for $P_n\in\TL_n$.  As in \cite{Hog_polyaction}, we will see in Section \ref{sec:Obstructing U_n^k} that this model is particularly well-suited to the search for certain endomorphisms of $\proj_n$ related to certain recursive properties of its presentation.

In the spectral setting, the Cooper-Krushkal sequence will require a reinterpretation via a non-negative filtration whose associated graded levels and attaching maps are determined by Figure \ref{fig:CK sequence}, and whose filtration levels will correspond to certain tangle diagrams involving $\proj_{n-1}$.  The case when $n=2$ (and $\proj_{n-1} = \proj_1 = \Imod_1$) is precisely Theorem \ref{thm:T2k simp} together with the discussion following the proof.  (The simpler approach of Theorem \ref{thm:P2 simp}, which presents $\proj_2$ via a semisimplicial model with no reference to any `finite stages' or filtration levels, does not easily generalize to $n>2$ due to the presence of nontrivial higher nullhomotopies.)  Thus for the remainder of this section we restrict to $n>2$.

\subsection{The spectral projector via the augmented infinite Jucys-Murphy spectrum $\augJM_n^\infty$}
\label{sec:augJMinf}

We begin with setting notation for two important families of Temperley-Lieb spectra.

\begin{definition}
\label{def:Jucys-Murphy}
The \emph{(left-handed) Jucys-Murphy spectrum $\JM_n$} (on $n$ strands) is the Khovanov spectrum assigned to the \emph{(left-handed) Jucys-Murphy braid}
\[\sigma_{n-1}^{-1} \cdots \sigma_{2}^{-1}\sigma_{1}^{-1} \sigma_{1}^{-1} \sigma_{2}^{-1} \cdots \sigma_{n-1}^{-1}\] 
shown in Figure \ref{fig:LHJM}.  Superscripts will be used to indicate vertical concatenations, as in
\[\JM_n^k := \JM_n\vertcomp\cdots\vertcomp\JM_n \quad \text{($k$ times)}.\]
The \emph{augmented Jucys-Murphy spectrum} is the vertical composition
\[\augJM_n:=(\cP_{n-1}\horizcomp \Imod_1) \vertcomp \JM_n\]
also shown in Figure \ref{fig:LHJM}, with superscripts used to indicate extra copies of $\JM$ within $\augJM$ as in
\[\augJM_n^k:=(\cP_{n-1}\horizcomp\Imod_1) \vertcomp \JM_n^k.\]
We also use the notation $\augJM_n^0:=\cP_{n-1}\horizcomp\Imod_1$.
\end{definition}

\begin{figure}
    \[
    \JM_n:=
    \vcenter{\hbox{
        \begin{tikzpicture}[xscale=.15mm,yscale=.3mm]
            \draw[thick] (0,0) -- (0,1);
            \draw[thick] (3,0) -- (3,1);
            \nstrandsalongpath[n-1]{0,1}{3,1}
            \drawover{
                (4,0) to[out=90,in=-90,distance=25] (-1,1);}
            \draw[thick] (-1,1) to[out=90,in=-90,distance=25] (4,2);
            \drawover{(0,1) -- (0,2);}
            \drawover{(3,1) -- (3,2);}
        \end{tikzpicture}
        }}
    \,,\quad\quad
    \augJM_n:=
    \vcenter{\hbox{
        \begin{tikzpicture}[xscale=.2mm,yscale=.3mm]
            \modbox{P}{0}{3}{\cP_{n-1}}
            \draw[thick] ( $(P box ne)+(1,0)$ ) -- ( $(P box se)+(1,0)$ ) to[out=-90,in=90,distance=20] ( $(P box sw)+(-1,-1)$ );
            \drawover{ (P box sw) -- +(0,-1); }
            \drawover{ (P box se) -- +(0,-1); }
            \draw[thick] ( $(P box sw)+(0,-1)$ ) -- +(0,-1);
            \draw[thick] ( $(P box se)+(0,-1)$ ) -- +(0,-1);
            \drawover{ ( $(P box sw)+(-1,-1)$ ) to[out=-90,in=90,distance=20] ( $(P box se)+(1,-2)$ ); }
            \nstrandsalongpath[n-1] { $(P box sw)+(0,-1)$ } { $(P box se)+(0,-1)$ }
        \end{tikzpicture}
        }}
    \]
    \caption{The Jucys-Murphy spectrum $\JM_n$ is the (Khovanov spectrum associated to) the braid on the left, where the right-most strand wraps around the other $n-1$ strands as shown.  The augmented $\augJM_n$ is similar, but includes a spectral projector $\cP_{n-1}$ as well.}
    \label{fig:LHJM}
\end{figure}
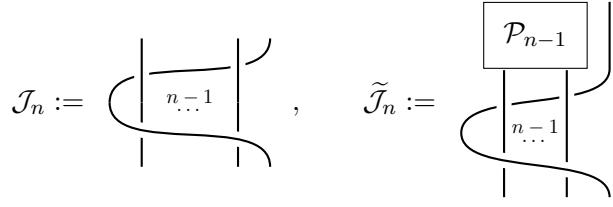

The $\cP_{n-1}$ present in the definition allows us to view $\augJM_n^k$ as an $\End_{\mathcal{H}_{2(n-1)}}(\mathcal{P}_{n-1})$-module.  In addition, the $\cP_{n-1}$ can slide above and below the looping strand, leading to Definition \ref{def:Jucys-Murphy} having some equivalent forms, as illustrated in the following lemma.

\begin{lemma}\label{lem:augJM equivalences}
The augmented Jucys-Murphy spectra satisfy
\begin{itemize}
    \item $\augJM_n \simeq \JM_n \vertcomp (\cP_{n-1}\horizcomp\Imod_1)$, and
    \item $\augJM_n^{k+\ell} \simeq \augJM_n^k\vertcomp\augJM_n^\ell \quad \text{(including for $\ell=0$)}$.
\end{itemize}
\end{lemma}
\begin{proof}
The first statement follows from modelling $\cP_{n-1}$ via $\T_n^m$ for large $m$, which can be seen to commute with $\JM_n$ via usual Reidemeister moves.  The second statement follows from the first together with the idempotence of $\cP_{n-1}$ (Proposition \ref{prop:projectors are idempotents}).
\end{proof}

Recall the notation $e_i$ for turnback diagrams introduced in Definition \ref{def:ei turnbacks}.  We establish properties for how the $\augJM_n^k$ behave when concatenated with turnbacks:

\begin{lemma}\label{lem:augmented JM and turnbacks}
Fix $n,k\in\N$.  Then we have
\begin{itemize}
    \item $\augJM_n^k \vertcomp e_i \simeq * \simeq e_i\vertcomp \augJM_n^k$ for all $i<n-1$,
    \item $\augJM_n^k \vertcomp e_{n-1} \simeq \q^{(4n-2)k}\Sigma^{(2n-2)k} \augJM_n^0 \vertcomp e_{n-1}$, and
    \item $e_{n-1} \vertcomp \augJM_n^k \simeq \q^{(4n-2)k}\Sigma^{(2n-2)k} e_{n-1}\vertcomp \augJM_n^0$.
\end{itemize}
\begin{proof}
The proof makes use of Lemma \ref{lem:augJM equivalences} throughout, and we do not call further attention to it.  The first point is a simple consequence of the turnback-killing property of $\cP_{n-1}$.  For the second point we induct on $k$, starting from the base case $k=0$ which is tautological.  The inductive step follows from the following sequence of equivalences:
\begin{align*}
\augJM_n^1\vertcomp e_{n-1}
=
\vcenter{\hbox{\begin{tikzpicture}[xscale=.25,yscale=.8]
    \draw[thick] (4,2)--(4,1.2) to[out=-90,in=90,looseness=.2] (-.5,.7);
    \drawover{(0,2)--(0,0);}
    \drawover{(2,2)--(2,0);}
    \drawover{(3,2)--(3,0);}
    \drawover{ (-.5,.7) to[out=-90,in=90,looseness=.4] (4,0);}
    \draw[thick] (4,0) to[out=-90,in=-90] (3,0);
    \whitebox{-.2,1.2}{3.2,1.7}
    \node[scale=.7] at (1.5,1.45) {$\cP_{n-1}$};
    \nstrandsalongpath[]{0,1.9}{2,1.9}
    \nstrandsalongpath[]{0,.2}{2,.2}
\end{tikzpicture}}}
\quad &\simeq \quad
\q^6\Sigma^2 \vcenter{\hbox{\begin{tikzpicture}[xscale=.25,yscale=.8]
    \draw[thick] (4,2)--(4,.7) to[out=-90,in=-90,looseness=.2] (-.5,.7);
    \drawover{(0,2)--(0,0);}
    \drawover{(2,2)--(2,0);}
    \drawover{ (3,2)--(3,1.2) to[out=-90,in=90,looseness=.2] (-.5,.7);}
    \whitebox{-.2,1.2}{3.2,1.7}
    \node[scale=.7] at (1.5,1.45) {$\cP_{n-1}$};
    \nstrandsalongpath[]{0,1.9}{2,1.9}
    \nstrandsalongpath[]{0,.2}{2,.2}
\end{tikzpicture}}}\\
&\simeq \quad
\q^6\Sigma^2 \q^{2(2n-4)}\Sigma^{2n-4}
\vcenter{\hbox{\begin{tikzpicture}[xscale=.25,yscale=.8]
    \draw[thick] (0,2)--(0,.5)
        (2,2)--(2,.5)
        (3,2)--(3,1) to[out=-90,in=-90] (4,1)--(4,2);
    \whitebox{-.2,1.2}{3.2,1.7}
    \node[scale=.7] at (1.5,1.45) {$\cP_{n-1}$};
    \nstrandsalongpath[]{0,1.9}{2,1.9}
    \nstrandsalongpath[]{0,.6}{2,.6}
\end{tikzpicture}}}\\ 
&= \quad \q^{4n-2}\Sigma^{2n-2} \augJM_n^0 \vertcomp e_{n-1},
\end{align*}
where the first equivalence comes from an isotopy involving Reidemeister III moves together with precisely two Reidemeister I moves, and the second equivalence comes from applying Proposition \ref{prop:projectors absorb crossings} to all $(2n-4)$ crossings present.  The third point is proved similarly.
\end{proof}
\end{lemma}

As was the case for the left-handed twist, we see that $\augJM_n$ has only left-handed crossings which allows us to use the left-hand map of Equation \eqref{eq:cofib seq for negcros} to build a sequence of maps similar to Equation \eqref{eq:inf twist seq def}

\begin{equation}\label{eq:inf augJM seq def}
(\cP_{n-1}\horizcomp \Imod_1) = \augJM_n^0 \rightarrow \q^{-2(n-1)}\augJM_n^1 \rightarrow \cdots \q^{-C(\JM_n^k)}\augJM_n^k \rightarrow \cdots,
\end{equation}
where $C(\JM_n^k):=2(n-1)k$ is the number of crossings in $\JM_n^k$, as before.  Similar to the infinite twist $\T_n^\infty$, we have the following definition.

\begin{definition}\label{def:infinite augJM}
The \emph{infinite augmented Jucys-Murphy spectrum}, denoted $\augJM_n^\infty$, is defined to be the homotopy colimit of the sequence \eqref{eq:inf augJM seq def}.
\end{definition}

With Definition \ref{def:infinite augJM} in place, our immediate goal is to prove the analogue of Theorem \ref{thm:inf twist is projector} for $\augJM_n^\infty$.

\begin{theorem}\label{thm:infinite augJM is a projector}
The homotopy colimit $\augJM_n^\infty$ (with the natural morphism $\Imod\xrightarrow{\iota} \augJM_n^\infty$ described in Equation \eqref{eq:Identity into augJMinf}) is a spectral projector.
\end{theorem}

Just as in the case of Theorem \ref{thm:inf twist is projector}, the proof will rely on understanding the cones of the maps involved.

\begin{proposition}\label{prop:inf augJM seq cones}
Let $f^k:\q^{-C(\JM_n^k)} \augJM_n^k \rightarrow \ q^{-C(\JM_n^{k+1})} \augJM_n^{k+1}$ denote the maps in the sequence \eqref{eq:inf augJM seq def}.  Then the cone of any $f^k$ can be written as
\begin{equation}\label{eq:inf augJM seq cones}
\Cone(f^k) \simeq \q^{(4n-2)k}\Sigma^{(2n-2)k} \augJMcone_n^k,
\end{equation}
where the spectrum $\augJMcone_n^k\in\Sp\TL_n$ is homotopy equivalent to a finite hocolim of basepoints and  terms of the form
\begin{equation}\label{eq:augJM cone terms of lower thru degree}
\q^a\,\,
\vcenter{\hbox{\begin{tikzpicture}[xscale=.2mm,yscale=.3mm]
    \modbox{P}{0}{3}{\cP_{n-1}}
    \draw[thick] (P box se) to[out=-90,in=-90] ( $(P box se) + (1,0)$ ) -- ( $(P box ne) + (1,0)$ );
    \node[inner sep=0, outer sep=0] (P box sse) at ( $(P box se) + (-.2,0)$ ){};
    \modbox[xscale=2.2]{t}{0.4}{1}{\delta}
    \draw[thick] (P box sse) to[out=-90,in=90] (t box ne);
    \draw[thick] (P box sw) to[out=-90,in=90] (t box nw);
    \nstrandsalongpath[n-2]{ $(P box.south) + (0,-.4)$} {t box.north}
    \draw[thick] (t box sw) -- +(0,-.5);
    \draw[thick] (t box se) -- +(0,-.5);
    \nstrandsalongpath{t box.south}{ $(t box.south) + (0,-.6)$ }
\end{tikzpicture}}}
\end{equation}
for some flat $(n-2,n)$ tangles $\delta$ of through-degree $n-2$, and with $q$-degree shift $a > 0$.
\end{proposition}
\begin{proof}
We consider $\augJM_n^{k+1}=\augJM_n^k\vertcomp \JM_n$, and resolve the $2(n-1)$ crossings in $\JM_n$ to view $\augJM_n^{k+1}$ as a hocolim over the cube $\cube^{2(n-1)}$.  The all-one resolution of $\JM_n$ is the identity tangle, so this vertex in the cube corresponds to $\augJM_n^k$.  Thus, the desired $\Cone(f^k)$ is homotopy equivalent to the corresponding hocolim over $\cube^{2(n-1)}$ where the all-one vertex is assigned a basepoint.

Meanwhile, each other term in this hocolim has the form
\[ \q^b \augJM_n^k \vertcomp \gamma \]
for some $b\geq 1$ and some flat tangles $\gamma$ having through-degree $\tau(\gamma)\leq n-2$.  Any terms with $\tau(\gamma)< n-2$ are contractible using Lemma \ref{lem:augmented JM and turnbacks}.  Meanwhile the terms with $\tau(\gamma)=n-2$ also contract except for the cases where $\gamma=e_{n-1}$ or $\gamma=e_{n-1}^2$.  In these cases, Lemma \ref{lem:augmented JM and turnbacks} produces the extra shifts $\q^{(4n-2)k}\Sigma^{(2n-2)k}$.  Finally, if there was a disjoint circle present (i.e. if $\gamma=e_{n-1}^2$), then this term can be decomposed into a wedge sum with extra $\q$-shifts of $\pm 1$.  Then since $b\geq 1$, we have $a \geq 0$ in the statement of the proposition.
\end{proof}

\begin{corollary}\label{cor:augJM cones in projector}
For any $k$ there exists a morphism $\q^{-C(\JM_n^k)}\augJM_n^k \xrightarrow{F_n^k} \augJM_n^\infty$ with
\[\Cone(F_n^k)\simeq \q^{2nk}\Sigma^{(2n-2)k} \augJMcone_n^{\infty-k},\]
where $\augJMcone_n^{\infty-k}$ denotes a spectrum which can be written as a homotopy colimit of basepoints and terms of the form \eqref{eq:augJM cone terms of lower thru degree}.
\end{corollary}

\begin{proof}[Proof of Theorem \ref{thm:infinite augJM is a projector}]
The proof that $\augJM_n^\infty$ is in $\Sp\TL_n$ is similar to the argument for $\T_n^\infty$.  The map $\iota$ is given by the composition
\begin{equation}\label{eq:Identity into augJMinf}
\Imod_n \xrightarrow{\iota_{n-1} \horizcomp 1} \proj_{n-1} \horizcomp \Imod_1 = \augJM_n^0 \xrightarrow{F_n^0} \augJM_n^\infty,
\end{equation}
and the cone of a composition fits into an exact triangle with the cones of the individual maps by the octahedral axiom.  The cones of $\iota_{n-1}\horizcomp 1$ and $F_n^0$ both have through-degree $<n$ (using Proposition \ref{prop:inf augJM seq cones} in the latter case), and so the cone of $F_n^0\circ (\iota_P\horizcomp 1)$ has through-degree $<n$ as well.  Thus $(\augJM_n^\infty,\iota)$ satisfies item \ref{itm:CK1} of Definition \ref{def:spectral projector}.

To show that $\augJM_n^\infty$ is killed by turnbacks (item \ref{itm:CK2} of Definition \ref{def:spectral projector}), we consider
\[e_i\vertcomp \augJM_n^\infty \simeq
\hocolim\left(e_i\vertcomp\augJM_n^0 \rightarrow \cdots \q^{-C(\JM_n^k)} e_i\vertcomp\augJM_n^k \rightarrow \cdots\right).\]
Lemma \ref{lem:augmented JM and turnbacks} shows that all of these terms are contractible for $i<n-1$, while for $i=n-1$ we have
\begin{align*}
e_{n-1}\vertcomp \augJM_n^\infty &\simeq
\hocolim\left(e_{n-1}\vertcomp\augJM_n^0 \rightarrow \cdots \q^{(4n-2)k-C(\JM_n^k)}\Sigma^{(2n-2)k} e_{n-1}\vertcomp\augJM_n^0 \rightarrow \cdots\right)\\
&\simeq
\hocolim\left(e_{n-1}\vertcomp\augJM_n^0 \rightarrow \cdots \q^{2nk}\Sigma^{(2n-2)k} e_{n-1}\vertcomp\augJM_n^0 \rightarrow \cdots\right),
\end{align*}
where we have used the count of crossings $C(\JM_n^k)=2(n-1)k$ in the last line.  Now as in Theorem \ref{thm:inf twist seq stabilizes}, the term $e_{n-1}\vertcomp\augJM_n^0$ has a finite lower bound for $q$-degrees beyond which all terms are contractible.  Thus this sequence must eventually become a sequence of contractible terms, and we are done.
\end{proof}

\subsection{Filtration levels of $\augJM_n^\infty$ and the proof of Theorem \ref{intro thm:Pn simplified}}
\label{sec:prove Pn simp}

Finally, we would like to relate our augmented Jucys-Murphy spectra with the Cooper-Krushkal sequence used by Hogancamp to study the projector $P_n\in\TL_n$, schematically illustrated in Figure \ref{fig:CK sequence}.  (Recall that, just as for spectral projectors, (chain) Cooper-Krushkal projectors are unique up to homotopy equivalence).  

\begin{figure}
\[\begin{tikzpicture}[xscale=2.8,yscale=2]
\node (term 0) at (5,5){$
    \begin{tikzpicture}[xscale=.25,yscale=.4]
        \draw[thick] (0,0) -- (0,2)
                (2,0) -- (2,2)
                (3,0) -- (3,2);
        \whitebox{-.2,.75}{2.2,1.25}
        \nstrandsalongpath[]{0,1.8}{2,1.8}
        \nstrandsalongpath[]{0,.2}{2,.2}
    \end{tikzpicture}$
    };
\node (term 1) at (5,4){ $\q\,\, \CKcolumnFIVE$};
\node (term 2) at (4,4){ $\q^2\,\, \CKcolumnFOUR$};
\node (dots 1) at (3,4) {$\cdots$};
\node (term n-2) at (2,4){ $\q^{n-2}\,\, \CKcolumnTWO$};
\node (term n-1) at (1,4){ $\q^{n-1}\,\, \CKcolumnONE$};
\node (term n) at (1,3){ $\q^{n+1}\,\, \CKcolumnONE$};
\node (term n+1) at (2,3){ $\q^{n+2}\,\, \CKcolumnTWO$};
\node (dots 2) at (3,3) {$\cdots$};
\node (term 2n-3) at (4,3) {$\q^{2n-2}\,\, \CKcolumnFOUR$};
\node (term 2n-2) at (5,3) {$\q^{2n-1}\,\, \CKcolumnFIVE$};
\node (term 2n-1) at (5,2) {$\q^{2n+1}\,\, \CKcolumnFIVE$};
\node (term 2n) at (4,2) {$\q^{2n+2}\,\, \CKcolumnFOUR$};
\node (dots 3) at (3,2) {$\cdots$};
\node (dots 4) at (3,1) {$\cdots$};
\node (term penultimate) at (4,1) {$\q^{2nk-2} \,\, \CKcolumnFOUR$};
\node (term ultimate) at (5,1) {$\q^{2nk-1} \,\, \CKcolumnFIVE$};
\node (dots end) at (5,0) {$\vdots$};
\foreach \dom/\ran in
{term 0/term 1, term 1/term 2, term 2/dots 1,
dots 1/term n-2, term n-2/term n-1, term n-1/term n,
term n/term n+1, term n+1/dots 2, dots 2/term 2n-3,
term 2n-3/term 2n-2, term 2n-2/term 2n-1, term 2n-1/term 2n,
term 2n/dots 3, dots 4/term penultimate, term penultimate/term ultimate, term ultimate/dots end}
    {\draw[<-] (\dom)--(\ran);}
\draw[<-] (dots 3) -- (1,2) -- (1,1) -- (dots 4);
\end{tikzpicture}\]
    \caption{The Cooper-Krushkal sequence \cite{Cooper-Krushkal} for $P_n\in\TL_n$ with $n>2$.  The white boxes indicate copies of $P_{n-1}$.  The horizontal maps are all saddle maps. The vertical maps are differences of dotted identity maps.  When lifting to spectra, we are forced to consider nullhomotopies for various compositions.}
    \label{fig:CK sequence}
\end{figure}
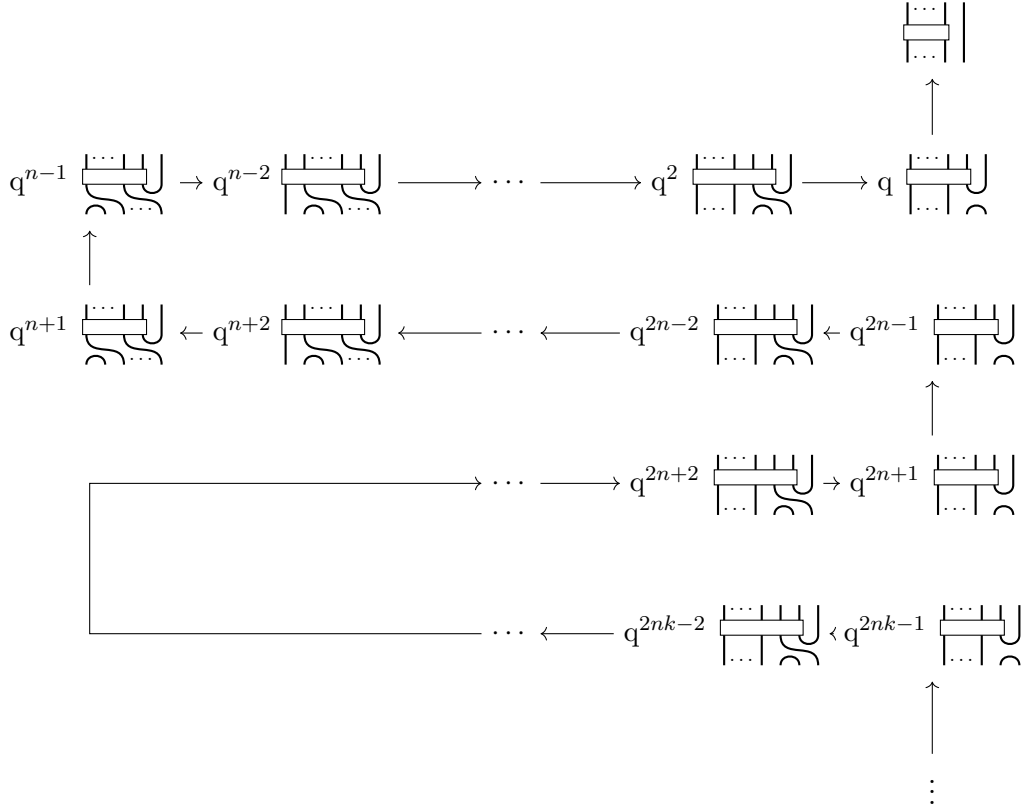

To begin, we establish some notation.  We will let $\CK(t)$ denote the Khovanov spectrum associated to the $t$-th term in Figure \ref{fig:CK sequence}, with $\CK(0)$ corresponding to the `ending' term $\augJM_n^0$, and increasing $t$ counting `backwards' along the sequence from there, with the $\q$-degree shifts included.  As in Section \ref{sec:augJMinf}, the presence of the projectors $\proj_{n-1}$ endows each $\CK(t)$ with a natural $\End_{\mathcal{H}_{2(n-1)}}(\proj_{n-1})$-module structure.  The reduced chain groups of these terms (with corresponding $\End_{H_{2(n-1)}}(P_{n-1})$-module structure) will be denoted $\CKch(t):=\chainsfunc(\CK(t))$.  The maps between these chain groups will be denoted
\[\CKch(t)\xrightarrow{f_t}\CKch(t-1).\]
When $t\not\equiv 1\mod n-1$, the map $f_t$ is induced by the obvious saddle map (this corresponds to the horizontal maps in Figure \ref{fig:CK sequence}).  The first vertical map $f_1$ is also induced by the obvious saddle map.  The other vertical maps $f_{m(n-1)+1}$ are induced by differences of dotted identity cobordisms (with dots on the lower turnback and the strand adjacent to it) when $m>0$.

In the homological setting, Cooper-Krushkal \cite{Cooper-Krushkal} built a convolution of Figure \ref{fig:CK sequence} and showed moreover that any convolution of Figure \ref{fig:CK sequence} gives a projector $P_n$.  Said differently: we interpret this convolution as providing a filtration $\CKchfilt$ on $P_n$ as an $\End_{H_{n-1}}(P_{n-1})$-module, whose associated graded levels and attaching maps are indicated in Figure \ref{fig:CK sequence}.

Now in order to lift this filtration to the spectral setting, we require one last bit of notation.  For $i\geq 0$ we let $\augJM_n^{<i>}$ denote the spectrum obtained by including the first $i$ crossings (from the top) of the diagram for $\augJM_n^k$ (for some $k>i$), so that $\augJM_n^{<0>}=\augJM_n^0$ and $\augJM_n^{<2n-2>}=\augJM_n^1$.  Applying Equation \eqref{eq:crossing cofib seq} to bottom-most crossings provides maps
\[\q^{-i}\augJM_n^{<i>} \rightarrow \q^{-(i+1)}\augJM_n^{<i+1>}\]
for all $i\geq 0$.  These maps compose to give the maps in the sequence \eqref{eq:inf augJM seq def} used to define $\augJM_n^\infty$, thereby giving a canonical equivalence
\begin{equation}\label{eq:inf augJM one crossing at a time}
\augJM_n^\infty \simeq \hocolim \left(
\augJM_n^{<0>} \rightarrow \q^{-1}\augJM_n^{<1>} \rightarrow \cdots
\q^{-i}\augJM_n^{<i>} \rightarrow \cdots \right).
\end{equation}
With this notation we are ready to restate Theorem \ref{intro thm:Pn simplified} in more detail.

\begin{theorem} \label{thm:Pn simplified in detail}
For $n>2$ the non-negative filtration $\CKfilt$, induced by the homotopy colimit \eqref{eq:inf augJM one crossing at a time}, on the spectral projector $\proj_n\simeq\augJM_n^\infty$ as an $\End_{\mathcal{H}_{2(n-1)}}(\proj_{n-1})$-module satisfies the following properties:
\begin{enumerate}
    \item \label{it:CKfilt levels} The filtration levels have canonical equivalences, of $\End_{\X(2n-2)}(\proj_{n-1})$-modules, $\CKfilt^i(\proj_n)\simeq \q^{-i}\augJM_n^{<i>}$ for all $i\geq 0$.
    \item \label{it:CKfilt ag levels} The associated graded levels have canonical equivalences $\CKfilt^i/\CKfilt^{i-1}(\proj_n)\to \Sigma^i\CK(i)$ for all $i\geq 0$.
    \item \label{it:CKfilt att maps} The attaching maps $\alpha_i$ on associated graded levels fit into a well-defined homotopy coherent diagram
    \[
    \begin{tikzpicture}[xscale=1.5]
        \node (a0) at (0,0) {$\CKfilt^i/\CKfilt^{i-1}(\proj_n)$};
        \node (a1) at (3,0) {$\Sigma\CKfilt^{i-1}/\CKfilt^{i-2}(\proj_n)$};
        \node (b0) at (0,-2) {$\Sigma^i\CK(i)$};
        \node (b1) at (3,-2) {$\Sigma^{i}\CK(i-1)$};

        \draw[->] (a0) -- (b0) node[pos=.5,left]{$\simeq$};
        \draw[->] (a1) -- (b1) node[pos=.5,right]{$\simeq$};
        \draw[->] (b0) -- (b1) node[pos=.5,above] {\scriptsize $\Sigma^{i}F_i$};
        \draw[->] (a0) -- (a1) node[pos=.5,anchor=south] {\scriptsize $\alpha_i$};
    \end{tikzpicture}
    \]
    where the vertical maps are the canonical equivalences of Item \eqref{it:CKfilt ag levels}, and the maps $\CK(i)\xrightarrow{F_i}\CK(i-1)$ are induced by saddles and differences of dotted identities in the same manner as the maps $f_i$ in Figure \ref{fig:CK sequence}.
\end{enumerate}
\end{theorem}
The proof will use the notation $[i]$ to denote the braid placement of the $i^{\text{th}}$ (bottom) crossing of $\augJM_n^{<i>}$, with corresponding turnback $e_{[i]}$.  This notation allows the following lemma.
\begin{lemma}\label{lem:Ji+1ei_1 is *}
For any $i\not\equiv 0 \mod n-1$ we have $\augJM_n^{<i-1>} \vertcomp e_{[i+1]}\simeq *$.
\end{lemma}
\begin{proof}
    This can be seen by inspection.  The upper turnback of $e_{[i+1]}$ can be pulled upwards and eventually reaches the $\proj_{n-1}$, which is killed by turnbacks.
\end{proof}

\begin{proof}[Proof of Theorem \ref{thm:Pn simplified in detail}]
Indeed, a homotopy colimit always has a filtration, associated to a filtration of the diagram, so that (\ref{eq:inf augJM one crossing at a time}) induces a filtration by truncating the sequence at finite stages; Item \eqref{it:CKfilt levels} follows immediately from the definition of the filtration.  The associated graded levels $\CKfilt^i/\CKfilt^{i-1}$ (we will omit the $\proj_n$ from the notation to avoid clutter) can be determined using the cofibration sequence \eqref{eq:crossing cofib seq} of the $i$-thcrossing of $\augJM_n^{<i>}$.  Namely, there is a cofibration sequence
\[
\augJM_n^{<i-1>}\to \augJM_n^{<i>}\to \Sigma q^{-i+2}\augJM_n^{<i-1>}e_{[i]},
\]
so that
\[\CKfilt^i/\CKfilt^{i-1} \simeq \Sigma\,\q^{-i+2}\augJM_n^{<i-1>}\vertcomp e_{[i]}.\]

For Items \eqref{it:CKfilt ag levels} and \eqref{it:CKfilt att maps} we induct on the filtration level.  For the base case $i=0$ we have $\CK(0)=\augJM_n^{<0>}$ and there is nothing to prove.  The case $i=1$ uses the cofibration sequence \eqref{eq:crossing cofib seq} for the lone crossing to see
\[\q^{-1}\augJM_n^{<1>} \rightarrow \Sigma\, \q \augJM_n^0\vertcomp e_{[1]} \xrightarrow{\Sigma s} \Sigma \augJM_n^0.\]
Noting that $[1]=n-1$, we see by inspection
\[\CKfilt^1/\CKfilt^{0} = \Sigma\, \q \augJM_n^0 \vertcomp e_{[1]} = \Sigma \CK(1)\]
with attaching map to $\CK(0)$ determined by the saddle $s$ as desired.

For the induction step we will assume we have proven both Items \eqref{it:CKfilt ag levels} and \eqref{it:CKfilt att maps} up through level $i$, and seek to prove the corresponding statements for level $i+1$.  This proof will be split into three cases.

In the case that $i\not \equiv 0,1\mod n-1$, we follow a similar pattern.  We use the cofibration sequence \eqref{eq:crossing cofib seq} applied to the bottom crossing in $\augJM_n^{<i+1>}$ to obtain a cofibration sequence: 
\[\q^{-(i+1)} \augJM_n^{<i+1>} \rightarrow \Sigma\, \q^{-i+1} \augJM_n^{<i>} \vertcomp e_{[i+1]} \xrightarrow{\Sigma (\Id\vertcomp s)} \Sigma\, \q^{-i} \augJM_n^{<i>}\vertcomp\Imod. \]
We simplify the second term (i.e. the associated graded term) and the third term simultaneously by using the cofibration sequence \eqref{eq:crossing cofib seq} applied to the bottom crossing of $\augJM_n^{<i>}$.  That is, we have a diagram of cofibration sequences:

\[\begin{tikzcd}
\CKfilt^{i+1}/\CKfilt^i = \Sigma\, \q^{-i+1} \augJM_n^{<i>} \vertcomp e_{[i+1]}
\ar[rr, "\Sigma(\Id\vertcomp s)"] \ar[d] & &
    \Sigma\, \q^{-i} \augJM_n^{<i>} \vertcomp \Imod
     \ar[d] \\
\Sigma^2\, \q^{-i+3} \augJM_n^{<i-1>} \vertcomp e_{[i]} \vertcomp e_{[i+1]}
\ar[rr,"\Sigma^2(\Id\vertcomp\Id\vertcomp s)"] \ar[d,"\Sigma^2(\Id\vertcomp s \vertcomp\Id)"] & &
    \Sigma^2\, \q^{-i+2} \augJM_n^{<i-1>} \vertcomp e_{[i]} \vertcomp \Imod
    \ar[d,"\Sigma^2(\Id\vertcomp s \vertcomp \Id)"]\\
\Sigma^2\, \q^{-i+2} \augJM_n^{<i-1>}\vertcomp \Imod \vertcomp e_{[i+1]}
\ar[rr,"\Sigma^2(\Id\vertcomp\Id\vertcomp s)"] & &
    \Sigma^2\, \q^{-i+1} \augJM_n^{<i-1>} \vertcomp \Imod \vertcomp \Imod.
\end{tikzcd}\]

The bottom left term is contractible by Lemma \ref{lem:Ji+1ei_1 is *}, so the left-side arrow in the upper square is an equivalence.  In the middle row we see the term $\Sigma\, \q^{-i+2} \augJM_n^{<i-1>}\vertcomp e_{[i]} = \CKfilt^i/\CKfilt^{i-1}$ which is inductively assumed to be equivalent to $\Sigma^i\CK(i)$.  These equivalences transform the second row into
\[\Sigma^{i+1}\,\q \CK(i) \vertcomp e_{[i+1]}
\xrightarrow{\Sigma^{i+1}(\Id\vertcomp s)}
\Sigma^{i+1}\CK(i)\vertcomp\Imod,\]
and we see by inspection that $\q\CK(i)\vertcomp e_{[i+1]} = \CK(i+1)$, with attaching map determined by the saddle $s$ as desired for Item \eqref{it:CKfilt att maps}.  This completes the $i\neq 0,1 \bmod{n-1}$ case.

In the case when $i\equiv 1 \mod n-1$ (so that $[i-1]=[i]$), we begin with the same cofibration sequence
\[\q^{-(i+1)} \augJM_n^{<i+1>} \rightarrow \Sigma\, \q^{-i+1} \augJM_n^{<i>} \vertcomp e_{[i+1]} \xrightarrow{\Sigma (\Id\vertcomp s)} \Sigma\, \q^{-i} \augJM_n^{<i>}\vertcomp\Imod. \]
In this case we have $\augJM_n^{<i>}\simeq \augJM_n^{<i-2>}\vertcomp \T_{2,[i]}^2$, where $\T_{2,[i]}^2$ consists of a full twist on two strands (crossings in braid position $[i]$) together with $n-2$ vertical strands.  We apply the cofibration sequence of Corollary \ref{cor:T22 simplified as multicone} to both the second term (i.e. the associated graded term) and the third term above, taking advantage of the various equivalences in Corollary \ref{cor:T22 simplified as multicone}.
\[\begin{tikzcd}
\CKfilt^{i+1}/\CKfilt^i = \Sigma\, \q^{-i+1} \augJM_n^{<i>} \vertcomp e_{[i+1]} \ar[rr,"\Sigma(\Id\vertcomp s)"] \ar[d] & &
\Sigma\, \q^{-i}\augJM_n^{<i>}\vertcomp \Imod \ar[d] \\
\Sigma^3\, \q^{-i+6}\augJM_n^{<i-2>}\vertcomp e_{[i-1]} \vertcomp e_{[i+1]} \ar[rr,"\Sigma^3(\Id\vertcomp\Id\vertcomp s)"] \ar[d] & &
\Sigma^3\, \q^{-i+5}\augJM_n^{<i-2>}\vertcomp e_{[i-1]} \vertcomp \Imod \ar[d] \\
\Sigma\, \q^{-i+2} \augJM_n^{<i-1>}\vertcomp e_{[i+1]} \ar[rr] & &
\Sigma\, \q^{-i+1} \augJM_n^{<i-1>}\vertcomp\Imod
\end{tikzcd}\]
Just as in the previous case, the bottom left term is contractible by Lemma \ref{lem:Ji+1ei_1 is *} so that the top left arrow is an equivalence, while the middle row inductively simplifies to
\[\Sigma^{i+1}\q^3\CK(i-1)\vertcomp e_{[i+1]} \xrightarrow{\Sigma^{i+1} (\Id\vertcomp s)} \Sigma^{i+1}\q^2\CK(i-1)\vertcomp\Imod.\]
We see by inspection once again that $\q^3\CK(i-1)\vertcomp e_{[i+1]}=\CK(i+1)$ while $\q^2\CK(i-1)=\CK(i)$ with attaching map determined by the saddle $s$ as desired.

Finally we turn to the case when $i\equiv 0 \mod n-1$ (so that $[i]=[i+1]$).  This time we have $\augJM_n^{<i+1>}\simeq \augJM_n^{<i-1>} \vertcomp \T_{2,[i]}^2$, and we have a simplified cofibration sequence, coming from Corollary \ref{cor:T22 simplified as multicone}:

\[\begin{tikzcd}
\q^{-(i+1)}\augJM_n^{<i+1>} \ar[r] &
\Sigma\,\q^{-i+1}\augJM_n^{<i>}e_{[i+1]} \simeq \Sigma\Cone( \q^{-i+4} \augJM_n^{<i-1>}\vertcomp e_{[i]} \rightarrow *)
\ar[d,"\Sigma({\Id\vertcomp w,*})"]
\\
 & \Sigma\, \q^{-i} \augJM_n^{<i>} \simeq\Sigma\Cone( \q^{-i+2}\augJM_n^{<i-1>}\vertcomp e_{[i]} \xrightarrow{\Id\vertcomp s} \q^{-i+1}\augJM_n^{<i-1>}\vertcomp\Imod),
\end{tikzcd}\]
where $e_{[i]}\xrightarrow{w}e_{[i]}$ is the difference of dotted identity maps.  Then just as before we apply cofibration sequences for the second term (i.e., the associated graded term) and the third term simultaneously to see
\[\begin{tikzcd}
\CKfilt^{i+1}/\CKfilt^i=\Sigma\,\q^{-i+1}\augJM_n^{<i>}e_{[i+1]} \ar[rr,"\Sigma({\Id\vertcomp w,*})"] \ar[d] & &
\Sigma\, \q^{-i} \augJM_n^{<i>} \ar[d] \\
\Sigma^2 \q^{-i+4} \augJM_n^{<i-1>}\vertcomp e_{[i]} \ar[rr,"\Sigma^2(\Id\vertcomp w)"] \ar[d] & &
\Sigma^2 \q^{-i+2} \augJM_n^{<i-1>}\vertcomp e_{[i]} \ar[d,"\Sigma^2(\Id\vertcomp s)"] \\
* \ar[rr] & &
\Sigma^2\, \q^{-i+1}\augJM_n^{<i-1>}\vertcomp\Imod
\end{tikzcd}.\]
The top left arrow is again an equivalence, and this time the middle row inductively simplifies to
\[\Sigma^{i+1} \q^3\CK(i-1)\vertcomp e_{[i]} \xrightarrow{\Sigma^{i+1}(\Id\vertcomp w)} \Sigma^{i+1} \q^1\CK(i-1)\vertcomp e_{[i]},\]
and we see by inspection that $\q^3\CK(i-1)\vertcomp e_{[i]} = \CK(i+1)$ while $\q^1\CK(i-1)\vertcomp e_{[i]} = \CK(i)$, with attaching map determined by $w$ as desired.

\end{proof}

Figure \ref{fig:CK sequence} illustrates in what sense $\augJM_n^\infty$ can be regarded as a smaller model for $\cP_n$, assuming one takes the spectral projector $\cP_{n-1}$ as given.  The sequence appears to exhibit a type of symmetry and periodicity, but it should be reminded that the attaching maps $\CK(i)\xrightarrow{F_i}\CK(i-1)$ in Theorem \ref{thm:Pn simplified in detail}, which depend only on $i\mod{2n-2}$, are only the induced maps on associated graded levels.  The full attaching information $\CK(i)\rightarrow\CKfilt^{i-1}(\proj_n)$ is not completely tracked in the figure.

On the level of homology however, Hogancamp \cite{Hog_polyaction} used the apparent symmetry and periodicity of the figure to construct a map $U_n$ on $P_n$.  We will see that this construction does not always lift to spectra.

We also note that the constructions in the section induce a filtration $F_{CK}$ of $P_n$, given by $F_{CK}^i(P_n)=\mathcal{C}_h(\CKfilt^{i}(\proj_n))$.

\section{Obstructing Hogancamp's Polynomial Action on Spectral Projectors}
\label{sec:Obstructing U_n^k}

In \cite{Hog_polyaction}, Hogancamp constructs maps on projectors which, in our grading conventions, would be written as
\[ U_n : \q^{2n}\Sigma^{2n-2} P_n \rightarrow P_n.\]
We here present statements equivalent to the existence of a lift of such a map (and/or its compositions $U_n^k$) to the spectral category.  We will see that such lifts do not always exist.  Before we begin however, let us note that in the $n=2$ case, it is easy to see that spectral lifts of $U_2$ do exist without any further analysis, proving Theorem \ref{intro thm: n=3 computations} \eqref{intro thm item: yes U2}.

\begin{theorem}\label{thm:U2 lifts}
There exists a map $\mathcal{U}_2:\q^4\Sigma^2 \proj_2 \to \proj_2$ which lifts the map $U_2$ described in \cite{Hog_polyaction}.
\end{theorem}
\begin{proof}
By Theorem \ref{intro thm: spectral Pn is tori} we have
\[\q^{-2}\End(\proj_2)\simeq \X(T(2,\infty)).\]
The links $T(2,k)$ are alternating, and thus \cite{LS_stablehtpytype} the spectrum $\X(T(2,\infty))$ is a wedge sum of Moore spectra, for which the Hurewicz map $\pi^{st}(\X(T(2,\infty))) \rightarrow H_*(\X(T(2,\infty)))$ is a surjection.  Thus there is a lift of $U_2\in H_2(\X(T(2,\infty)))$ to an element of $\pi^{st}(\X(T(2,\infty)))$ as desired.
\end{proof}

Let us now consider the more general setting of $n\geq 3$.  We begin be considering the saddle cobordism $\Itang_n \xrightarrow{s_n} e_{n-1}$.  If we concatenate this with a shifted copy of $\augJM_n^0$, we can build the following diagram of maps in $\Sp\TL_n$:

\begin{equation}\label{eq:maps related to obstructing U_n^k before closure}
\vcenter{\hbox{\begin{tikzpicture}[xscale=1.7mm,yscale=1.4mm]
\node (JM0) at (0,1){
    $\q^{2nk}\Sigma^{(2n-2)k} \augJM_n^0$
    };
\node (JM0e) at (0,0){
    $\q^{2nk-1}\Sigma^{(2n-2)k}\,\,
    \LiftDiagBL$
    };
\node (JMk0res) at (1,0){
    $\q^{-2(n-1)k+2}\Sigma\,\,
    \LiftDiagBOT$
    };
\node (JMk1res) at (2,0){
    $\q^{-2(n-1)k+1} \Sigma\,\,
    \LiftDiagBR$
    };
\node (JMk) at (1.5,1){
    $\q^{-2(n-1)k} \augJM_n^k$
    };
\draw[->] (JM0) -- (JM0e) node[pos=.5,left]{$s_n^k$};
\draw[->] (JM0e) -- (JMk0res) node[pos=.5,above]{$\simeq$};
\draw[->] (JMk0res) -- (JMk1res) node[pos=.5,above]{$\Sigma \sigma_n^k$};
\draw[->] (JMk) -- (JMk0res);
\draw[->,dashed] (JMk1res) -- (JMk);
\draw[ultra thick,dashed,->,red] (JM0) -- (JMk) node[pos=.5,above]{$\exists$ lift $\widehat{s_n^k}$?};
\end{tikzpicture}}}
\end{equation}
Here we use the notation $s_n^k$ for the relevant saddle map on spectra in a slight abuse of notation.  The equivalence is found by using Lemma \ref{lem:augmented JM and turnbacks} together with a single Reidemeister 1 move.  The triangle on the right is the cofibration sequence \eqref{eq:crossing cofib seq} for the lowest crossing in the diagram for $\augJM_n^k$, with corresponding saddle map denoted $\sigma_n^k$; the third map is dashed to indicate that an extra suspension would be needed.  Given the diagram of non-dashed arrows only, one may then ask whether the map $s_n^k$ (after composing with the equivalence) lifts to give a map $\widehat{s_n^k}$, indicated by the thick (red) dashed arrow.

The diagram can be simplified with the adjunction of Theorem \ref{thm:adjunction}, which one can check sends $s_n^k\in\HOM_{\Sp\TL_n}(\augJM_n^0, \q^{-1}\augJM_n^0 e_{n-1})$ to $\Id\in\HOM_{\Sp\TL_{n-1}} (\cP_{n-1},\cP_{n-1})$, giving us the similar diagram:

\begin{equation}\label{eq:maps related to obstructing U_n^k one closure}
\vcenter{\hbox{\begin{tikzpicture}[xscale=1.7mm,yscale=1.4mm]
\node (Pn-1) at (0,1){
    $\q^{2nk}\Sigma^{(2n-2)k} \cP_{n-1}$
    };
\node (Pn-1again) at (0,0){
    $\q^{2nk}\Sigma^{(2n-2)k} \,\,
    \TLiftDiagBL$
    };
\node (JMk0res) at (1,0){
    $\q^{-2(n-1)k+3}\Sigma \,\,
    \TLiftDiagBOT$
    };
\node (JMk1res) at (2.15,0){
    $\q^{-2(n-1)k+2} \Sigma \,\,
    \TLiftDiagBR$
    };
\node (JMk) at (1.5,1){
    $\q^{-2(n-1)k+1} T(\augJM_n^k)$
    };
\draw[->] (Pn-1) -- (Pn-1again) node[pos=.5,left]{$\Id$};
\draw[->] (Pn-1again) -- (JMk0res) node[pos=.5,above]{$\simeq$};
\draw[->] (JMk0res) -- (JMk1res) node[pos=.5,above]{$\Sigma T(\sigma_n^k)$};
\draw[->] (JMk) -- (JMk0res);
\draw[->,dashed] (JMk1res) -- (JMk);
\draw[ultra thick,dashed,->,red] (JM0) -- (JMk) node[pos=.5,above]{$\exists$ lift $\widehat{\phi_n^k}$?};
\end{tikzpicture}}}
\end{equation}

\begin{lemma}\label{lem:lifting diagrams equivalent}
The lift $\widehat{s_n^k}$ of Equation \eqref{eq:maps related to obstructing U_n^k before closure} exists if and only if the lift $\widehat{\phi_n^k}$ of Equation \eqref{eq:maps related to obstructing U_n^k one closure} exists.
\end{lemma}
\begin{proof}
This follows from the functoriality of the adjunction Theorem \ref{thm:adjunction}.
\end{proof}

The proof of the following theorem will make use of the chains functor $\Sp\TL_n\xrightarrow{\chainsfunc}\TL_n$ considered in Section \ref{sec:spTL cat}.  

\begin{theorem}\label{thm:Unk equiv to lifts}
The following statements are equivalent.
\begin{enumerate}
    \item \label{it:partial trace saddle} The saddle map on partial traces $T(\sigma_n^k)$, as shown in Equation \eqref{eq:maps related to obstructing U_n^k one closure}, is null-homotopic.
    \item \label{it:full closure saddle} The corresponding saddle on full closures $T^n(\sigma_n^k)$ is null-homotopic.
    \item \label{it:lifts exist} The lifts $\widehat{s_n^k}$, $\widehat{\phi_n^k}$ shown in Equations \eqref{eq:maps related to obstructing U_n^k before closure}, \eqref{eq:maps related to obstructing U_n^k one closure} exist.
    \item \label{it:Unk exists} 
There exists a map of spectral bimodules $\projmap_n^k\colon \q^{2nk}\Sigma^{(2n-2)k} \proj_n \rightarrow \proj_n$ lifting the map $U_n^k$ described in \cite{Hog_polyaction}.
\end{enumerate}
\end{theorem}
\begin{proof}
First we show \eqref{it:partial trace saddle} $\Leftrightarrow$ \eqref{it:full closure saddle}.  Write $X,Y$ respectively for the lower entries in the cofibration sequence in \eqref{eq:maps related to obstructing U_n^k one closure}.  By construction, $X$ is a homotopy colimit built from $X_i$ consisting of a single copy of the identity module (with some shift, which we suppress from the notation), along with turnbacks.  The spectrum $Y$, similarly, annihilates turnbacks, by definition.  Then we have:
\begin{equation*}
\HOM(X,Y)=\HOM(\hocolim_i X_i,Y) \simeq \holim_i \HOM (X_i,Y)\simeq \HOM(\Imod_{n-1},Y).
\end{equation*}
By the adjunction Theorem \ref{thm:adjunction}, $\HOM(\Imod,Y)\simeq \HOM(\emptyset,q^{n-1}T^{n-1}(Y))$, so that $T(\sigma^k_n)$ is nullhomotopic if and only if $T^{n-1}(T(\sigma^k_n)$ is nullhomotopic.

To show that \eqref{it:partial trace saddle} $\Leftrightarrow$ \eqref{it:lifts exist}, we appeal to the homotopy lifting property of the cofibration sequence for the right triangle, which implies that \eqref{it:partial trace saddle} is equivalent to the lift $\widehat{\phi_n^k}$ existing.  Lemma \ref{lem:lifting diagrams equivalent} then gives $\widehat{s_n^k}$ as well.

To show that $\eqref{it:lifts exist}\Rightarrow\eqref{it:Unk exists}$, we consider the concatenation
\[\Id_{\cP_n}\vertcomp\widehat{s_n^k}: \q^{2nk}\Sigma^{(2n-2)k}\cP_n\vertcomp\augJM_n^0 \rightarrow \q^{-2(n-1)k}\cP_n\vertcomp\augJM_n^k.\]
Note that $\cP_n\vertcomp \augJM_n^0\simeq \cP_n$, while $\cP_n\vertcomp\augJM_n^k\simeq \cP_n\vertcomp\JM_n^k\simeq \q^{2(n-1)k}\cP_n$
by Propositions \ref{prop:projector absorbs smaller projectors} and \ref{prop:projectors absorb crossings}.  In this way our concatenation $\Id_{\cP_n}\vertcomp\widehat{s_n^k}$ induces a map 
\[\mathcal{V}\colon \q^{2nk}\Sigma^{(2n-2)k}\cP_n\rightarrow \cP_n.\]  

It is a consequence of the construction of $U_n^k$ in \cite{Hog_polyaction} that it lives in 
\[\HOM(q^{2nk}\Sigma^{(2n-2)k}P_{n},F_{CK}^{(2n-2)k}P_n),\]
i.e. that $U_n^k$ factors through the first $(2n-2)k$ steps of the CK filtration. We can  postcompose by $\pi\colon F_{CK}^{(2n-2)k}P_n\to (F_{CK}^{(2n-2)k}P_n)/F_{CK}^{(2n-2)k-1}P_n$ to land in the associated graded:

\[
\HOM(q^{2nk}\Sigma^{(2n-2)k}P_n,CK(2k(n-1)))
\]
We can further pull back along $\iota\colon P_{n-1}\horizcomp I_1\to P_{n}$
to obtain a map in 
\[
\HOM(q^{2nk}\Sigma^{(2n-2)k}(P_{n-1}\horizcomp I_1),CK(2k(n-1)))=
\HOM(q^{2nk}\Sigma^{(2n-2)k}(P_{n-1}\horizcomp I_1),P_{n-1}e_{n-1})
\]
The image of $U_n^k$ in this latter group of homomorphism is the saddle $P_{n-1}\horizcomp I_1\to P_{n-1}e_{n-1}$ along the rightmost two strands.  By abuse of notation, we refer to this process:
\[
\pi_*\iota^*\colon \HOM(P_n,F^{(2n-2)k}_{CK}P_n)\to \HOM(P_{n-1}\horizcomp I_1,P_{n-1}e_{n-1})
\]
as \emph{restriction}.

In fact, we claim that $U_n^k$ is the unique homology class whose restriction is the saddle, in $H_*(\HOM(P_{n},P_{n}))$.  This follows from \cite[Theorem 1.12]{Hog_polyaction}, with degrees given in Theorem 3.11 and Definition 3.27.  This may be seen by observing that $U_n$ strictly minimzes the (absolute value of the) ratio of homological grading to quantum grading among all (algebra) generators of $\HOM(P_n,P_n)$, so any monomial in the $\{U_i,\xi_i\}_i$ will have minimal homological to quantum ratio if and only if is a power of $U_n$.  

Thus, to show that $\mathcal{C}_h(\mathcal{V})\simeq U_n^k$, it suffices to show that the restriction of $\mathcal{V}$, as above, is the map on spectra associated to a saddle.  However, this follows directly from the definition of $\mathcal{V}$ in terms of $\widehat{s_n^k}$.  This completes the proof of $\eqref{it:lifts exist}\Rightarrow \eqref{it:Unk exists}$.

Finally, we prove that \eqref{it:Unk exists} $\Rightarrow$ \eqref{it:lifts exist}.  We begin with a map $\projmap_n^k$ (in spite of the notation, it is not asserted that $\projmap_n^k$ is the power of any other map) which we fit into the following diagram
\begin{equation}\label{eq:unk-truncation}
\begin{tikzpicture}[xscale=1.3mm,yscale=.7mm]
\node (shI) at (0,0) {$\q^{2nk}\Sigma^{(2n-2)k} \Imod_n$};
\node (shP) at (1,0) {$\q^{2nk}\Sigma^{(2n-2)k} \cP_n$};
\node (P) at (2,0) {$\cP_n \simeq \augJM_n^\infty$};
\node (J) at (1.5,-1) {$\q^{-2(n-1)k}\augJM_n^k$};
\node (K) at (2.5,-1) {$\q^{2nk}\Sigma^{(2n-2)k}\augJMcone_n^{\infty-k}$}; 
\draw[right hook->] (shI)--(shP) node[pos=.5,above]{$\iota'$};
\draw[->] (shP)--(P) node[pos=.5,above]{$\projmap_n^k$};
\draw[right hook->] (J)--(P);
\draw[->>] (P)--(K);
\end{tikzpicture}
\end{equation}
where the two diagonal maps on the right come from the cofibration sequence for the cone of Corollary \ref{cor:augJM cones in projector}.  The full composition from top-left to bottom-right gives a map in $\Hom(\Imod_n,\augJMcone_n^{\infty-k})$ which is, by the adjunction of Theorem \ref{thm:adjunction}, equivalent to $\Hom(\bS,\q^n T^n (\augJMcone_n^{\infty-k}) ) = \pi_0^{st} \left( T^n(\augJMcone_n^{\infty-k})\atq[\Big]{-n} \right)$.  According to Corollary \ref{cor:augJM cones in projector}, the spectrum $T^n(\augJMcone_n^{\infty-k})$ is built as a hocolimit of closures of terms of the form \eqref{eq:augJM cone terms of lower thru degree}, all of which are trivial at $q$-grading $-n$ since Theorem \ref{thm:inf twist seq stabilizes} shows that the minimal non-trivial $q$-grading of $\cP_{n-1}$ is precisely $-(n-1)$.  Thus our full composition must be trivial and we have a lift $\q^{2nk}\Sigma^{(2n-2)k}\Imod_n \xrightarrow{A_n^k} \q^{-2(n-1)k}\augJM_n^k$.  We then concatenate with $\augJM_n^0$ and use Lemma \ref{lem:augJM equivalences}, to arrive at a map $\q^{2nk}\Sigma^{(2n-2)k} \augJM_n^0 \xrightarrow{\widetilde{A_n^k}} \q^{-2(n-1)k}\augJM_n^k$.  We would like to conclude that this map $\widetilde{A_n^k}$ is precisely the desired lift $\widehat{s_n^k}$.

To see this, we consider the following composition taken from Equation \eqref{eq:maps related to obstructing U_n^k before closure} with $\widetilde{A_n^k}$ taking the place of the proposed lift $\widehat{s_n^k}$:
\[\q^{2nk}\Sigma^{(2n-2)k} \augJM_n^0 \xrightarrow{\widetilde{A_n^k}} \q^{-2(n-1)k}\augJM_n^k \twoheadrightarrow \q^{-2(n-1)k+2}\Sigma (\augJM_n^k)' \simeq \q^{2nk-1}\Sigma^{(2n-2)k} \augJM_n^0 e_{n-1},\]
where $(\augJM_n^k)'$ indicates (the spectrum associated to) $\augJM_n^k$ with the bottom-most crossing replaced by its 0-resolution; see the middle term on the bottom row of Equation \eqref{eq:maps related to obstructing U_n^k before closure}.  This composition is a map in
\begin{equation}\label{eq:ank-map}\Hom_{\Sp\TL_n}(\augJM_n^0, \q^{-1}\augJM_n^0 e_{n-1}) \simeq \Hom(\bS, \q^{n-1} T^{n-1}(\cP_{n-1}))\simeq \pi_0^{st}(T^{n-1}(\cP_{n-1})\atq[\big]{-(n-1)})\end{equation}
via Theorem \ref{thm:adjunction}, and Theorem \ref{thm:inf twist seq stabilizes} shows that $T^{n-1}(\cP_{n-1})\atq[\big]{-(n-1)}\simeq\bS$.

Thus our composition is identified with an element in $\pi_0^{st}(\bS)\cong\Z$.  Using the compatibility \eqref{eq:chains-adjunction} applied to \eqref{eq:ank-map}, together with the fact that the Hurewicz map $\pi_0^{st}(\bS)\to H_0(\bS)$ is an isomorphism, we see that $\widetilde{A_n^k}$ is a suitable lift $\widehat{s_n^k}$ if and only $\mathcal{C}_h(\widetilde{A_n^k})$ lifts $\mathcal{C}_h(s^k_n)$.     

By hypothesis, $\mathcal{U}_n^k$ lifts $U_n^k$.  As in the proof of $\eqref{it:lifts exist}\Rightarrow\eqref{it:Unk exists}$ we have that $U_n^k$ is uniquely specified, up to homotopy,  by the conditions that it is a map $P_{n}\to P_n$ so that:
\begin{enumerate}
    \item The map $U_n^k$ factors through $F_{CK}^{(2n-2)k}P_n\to P_n$
    \item \label{itm:uk-def} The composition $\pi_*\iota^*(U_n^k)$ is the saddle map.
\end{enumerate}
From the definition of $\mathcal{C}_h(\widetilde{A_n^k})$, we have that $P_n\vertcomp\mathcal{C}_h(\widetilde{A^k_n})$ (that is to say, the stacking with $P_n$) factors through $F_{CK}^{(2n-2)k}P_n\to P_n$

Write $p\colon q^{-2(n-1)k}\augJM_{n}^k\to F^{(2n-2)k}_{CK}(P_n)$ for the projection to the associated graded.  Note that this is exactly the left diagonal map in \eqref{eq:maps related to obstructing U_n^k before closure}.  Then, by construction, \[p\circ  \mathcal{C}_h(\widetilde{A^k_n})=p\circ ( U^k_n\circ \mathcal{C}_h(\iota)),\] where $\iota$ is the natural map $\augJM^0_n\to \mathcal{P}_n$.  By item \eqref{itm:uk-def}, the latter term is exactly $\mathcal{C}_h(s^k_n)$.   This completes the proof of $\eqref{it:Unk exists}\Rightarrow \eqref{it:lifts exist}$.
\end{proof}

We will now focus on the map \begin{equation*}\Sigma T(\sigma_n^k)\in [q^{-2(n-1)k+3}\Sigma (\augJM^{k-1}_n)',q^{-2(n-1)k+2}\Sigma (\augJM^{k-1}_n)''],\end{equation*}
where $(\augJM^k_n)',(\augJM^k_n)''$ are respectively the lower-left and lower-right entries of the triangle in \eqref{eq:maps related to obstructing U_n^k one closure}.   We can precompose with the natural map $\Imod_{n}\to \augJM^{k-1}_n$ (together with gluing) to obtain a map 
\begin{equation}\label{eq:closed-map}
\beta_{n,k}'=\Sigma T (\sigma_n^k)\circ \iota \colon \JM^1_{n-1}\to q^{-2(n-1)k+2}\Sigma (\augJM^{k-1}_n)'',
\end{equation}
where we use Reidemeister moves to identify the spectrum of $(\augJM^{k-1}_n)'$, where the subdiagram for $\augJM^{k-1}_n$ is replaced with the identity tangle, with $\JM^1_{n-1}$.  We write $\iota$ for the map $\JM^1_{n-1}\to (\augJM^{k-1}_n)'$ obtained from these Reidemeister moves.  Since $(\augJM^{k-1}_n)''$ kills turnbacks, an argument similar to the proof of Theorem \ref{thm:Unk equiv to lifts} shows that $\iota$ induces a bijection:
\begin{align*}
\iota^*\colon [q^{-2(n-1)k+3}\Sigma(\augJM^{k-1}_n)' &, q^{-2(n-1)k+2}\Sigma (\augJM^{k-1}_n)''] \to \\
 &[q^{-2(n-1)k+3}\Sigma\JM^1_{n-1}, q^{-2(n-1)k+2}\Sigma (\augJM^{k-1}_n)''].
\end{align*}

By Theorem \ref{thm:Unk equiv to lifts},  $\beta_{n,k}'$ is nullhomotopic if and only if $U_n^k$ lifts to spectra.  Let $M_\ell$ be a sequence of finite tangles, together with morphisms $\alpha_\ell\colon M_\ell \to \augJM^{k-1}_n$ so that the minimum $q$-shift in $\Cone(\alpha_\ell)$ goes to $\infty$ as $\ell\to \infty$ (for instance, by using left-handed twists as in Section \ref{sec:construction}).  Write $\beta_{n,k}$ for the map of spectra as in \eqref{eq:closed-map}, but replacing the $\augJM$ terms with their approximations, for some sufficiently large $\ell$ which we suppress from the notation.  By construction, the map $\beta_{n,k}$ is the Khovanov spectrum map of the cobordism of (not infinite) links given by the saddle in \eqref{eq:maps related to obstructing U_n^k one closure}, together with the sequence of Reidemeister moves from the usual presentation of $\JM$ to  the corresponding tangle in \eqref{eq:maps related to obstructing U_n^k one closure}.  By \cite{LLS_func} the homotopy type of the map $\beta_{n,k}$ (up to sign) is a well-defined invariant of this cobordism.  This discussion, together with Theorem \ref{thm:Unk equiv to lifts}, implies that the existence of spectral versions of the $U_n^k$ maps is equivalent to a statement on cobordism maps for Khovanov spectra (of links, as opposed to tangles or limits of Khovanov spectra):

\begin{corollary}\label{cor:interesting-cobordism-map}
    The cobordism map of Khovanov spectra $\beta_{n,k}$ is nullhomotopic if and only if $U_n^k$ lifts to a map of spectra $q^{2nk}\Sigma^{(2n-2)k}\mathcal{P}_n\to \mathcal{P}_n$.  
\end{corollary}

\section{The 3-strand Case}\label{sec:3-strands}


Let us now focus on the case of $n=3$ strands, for which we have access to a handful of helpful computations both for homology and for spectra.  We will use the over-bar notation to indicate the (Khovanov spectrum associated to the) standard closure of a tangle in $S^3$, which coincides with the full trace.  So in particular
\[\closure{\T_n^k}:=T^n(\T_n^k) = \X(T(n,k)),\]
where $T(n,k)$ denotes the $(n,k)$ torus link in $S^3$.  Similarly, we have
\[\closure{\proj_n} \simeq T^n(\T_n^\infty) =: \X_{col}(\mathscr{U}_n),\]
where $\X_{col}(\mathscr{U}_n)$ denotes the \emph{colored Khovanov spectrum} for the $n$-colored unknot $\mathscr{U}$, as defined in \cite{Wil_colored,LOS_colored}.

\subsection{The maps $U_3^k$}
\begin{theorem}\label{thm:U3 does not lift}
There does not exist any map of spectra $\projmap_3^1:\q^6\Sigma^4 \cP_3\rightarrow \cP_3$ which lifts the map $U_3^1$ described in \cite{Hog_polyaction}.
\end{theorem}
\begin{proof}
We utilize Item \eqref{it:full closure saddle} of Theorem \ref{thm:Unk equiv to lifts}, which in the $n=3,k=1$ case shows that the existence of $\projmap_3^1$ is equivalent to the statement that 
\[
\vcenter{\hbox{\begin{tikzpicture}[xscale=.25,yscale=.6]
    \draw[thick] (-.5,.5) to[out=90,in=90,looseness=.3] (2.5,.5);
    \drawover{(.2,2)--(.2,0);}
    \drawover{(1.8,2)--(1.8,0);}
    \drawover{(-.5,.5) to[out=-90,in=-90,looseness=.3] (2.5,.5);}
    \whitebox{-.2,1.15}{2.2,1.85}
    \node[scale=.7] at (1,1.5) {$\cP_2$};
    \draw[thick]
        (.2,0) to[out=-90,in=-90,looseness=.2] (4,0)--(4,2) to[out=90,in=90,looseness=.2] (.2,2)
        (1.8,0) to[out=-90,in=-90,looseness=.2] (3,0)--(3,2) to[out=90,in=90,looseness=.2] (1.8,2);
\end{tikzpicture}}}
\quad \stackrel{?}{\simeq} \quad \q^2\Sigma\,\,
\vcenter{\hbox{\begin{tikzpicture}[xscale=.25,yscale=.6]
    \draw[thick] (-.5,.5) to[out=90,in=90,looseness=.3] (2.5,.5);
    \drawover{(.2,2)--(.2,0);}
    \drawover{(1.8,2)--(1.8,.5) to[out=-90,in=-90,looseness=.3] (-.5,.5);}
    \draw[thick] (2.5,.5) to[out=-90,in=90,looseness=1.1] (1.8,0);
    \whitebox{-.2,1.15}{2.2,1.85}
    \node[scale=.7] at (1,1.5) {$\cP_2$};
    \draw[thick]
        (.2,0) to[out=-90,in=-90,looseness=.2] (4,0)--(4,2) to[out=90,in=90,looseness=.2] (.2,2)
        (1.8,0) to[out=-90,in=-90,looseness=.2] (3,0)--(3,2) to[out=90,in=90,looseness=.2] (1.8,2);
\end{tikzpicture}}}
\quad\bigvee\quad \q\,\,
\vcenter{\hbox{\begin{tikzpicture}[xscale=.25,yscale=.6]
    \draw[thick] (-.5,.5) to[out=90,in=90,looseness=.3] (2.5,.5);
    \drawover{(.2,2)--(.2,0);}
    \drawover{(1.8,2)--(1.8,.5) to[out=-90,in=-90,looseness=.6] (2.5,.5);}
    \drawover{(-.5,.5) to[out=-90,in=90,looseness=.3] (1.8,0);}
    \whitebox{-.2,1.15}{2.2,1.85}
    \node[scale=.7] at (1,1.5) {$\cP_2$};
    \draw[thick]
        (.2,0) to[out=-90,in=-90,looseness=.2] (4,0)--(4,2) to[out=90,in=90,looseness=.2] (.2,2)
        (1.8,0) to[out=-90,in=-90,looseness=.2] (3,0)--(3,2) to[out=90,in=90,looseness=.2] (1.8,2);
\end{tikzpicture}}}\,\,.
\]
The two spectra on the right-hand side are both equivalent to shifted copies of the closed projector $\closure{\cP_2}$ via Reidemeister moves together with Proposition \ref{prop:projectors absorb crossings} (compare \cite[Proposition 3.51]{Hog_polyaction}).  Via Theorem \ref{thm:inf twist is projector}, these are stable limits of spectra for 2-strand torus links, which in particular are alternating and thus have trivial Steenrod square operations \cite{LS_steenrod}.

On the other hand, the left-hand side gives the spectrum of the $(2,1)$-colored Hopf link, computed in \cite{LOS_colored}, and which has non-trivial Steenrod $Sq^2$ operations.  (Indeed we can use stabilization arguments as in \cite{Wil_colored} to see that, in $q$-degree 7, this stable homotopy type is computed in finite approximation by $\closure{\T_3^3}\atq{9}$, which is precisely the first wedge summand of $\closure{\T_3^\infty}$ containing non-trivial $Sq^2$).

Thus Item \eqref{it:full closure saddle} of Theorem \ref{thm:Unk equiv to lifts} fails to hold and the map $U_3^1$ cannot have a lift.
\end{proof}

\begin{remark}\label{rmk:no sym proj or eigenmap}
Theorem \ref{thm:U3 does not lift} can be interpreted as showing that there does not exist any Temperley-Lieb spectrum $\mathcal{Q}_3$ that satisfies the spectral analog of the definition \cite[Definition 3.9]{Hog_polyaction} of a \emph{symmetric projector}.  This does not imply \emph{a priori} that there does not exist any object $\mathcal{Q}_3$ of $\Sp\TL_3$ so that $\chainsfunc(\mathcal{Q}_3)=Q_3$, where $Q_3$ is as described in \cite[Definitions 3.5, 3.9]{Hog_polyaction}, though this seems natural to conjecture.  In the language of categorical diagonalization \cite{cat-diag}, this also indicates that one of the eigenmaps for the full twist on three strands does not lift to spectra, and thus $\FT_3$ cannot be categorically diagonalized in the same manner as $\chainsfunc(\FT_3)$ can be.
\end{remark}


\begin{theorem}\label{thm:U3^2 lifts}
There exists a map $\projmap_3^2:\q^{12}\Sigma^8 \proj_3\rightarrow\proj_3$ which lifts the map $U_3^2$ described in \cite{Hog_polyaction}.
\end{theorem}
\begin{proof}
Using the adjunctions of Theorem \ref{thm:adjunction} and \cite[Proposition 5.11]{LLS_func} together with the idempotency of Proposition \ref{prop:projectors are idempotents}, we have
\[\Hom(\q^{12}\Sigma^8 \proj_3, \proj_3) \simeq \Hom(\q^9\Sigma^8 \mathbb{S}, \closure{\proj_3}).\]
As mentioned earlier, the spectrum $\closure{\proj_3}$ is precisely the colored Khovanov spectrum for the 3-colored unknot $\mathscr{U}_3$, which was studied in low quantum degrees in \cite[Section 4.4]{LOS_colored}.  Using the computer computations cited there, we see that our hom space is equivalent to $\pi_8^{st}\left( \X_{col}(\mathscr{U}_3)\atq{9} \right) \simeq \pi_8^{st}(X(\eta 2 ,5)\vee \Sigma^8\mathbb{S})$.  The Hurewicz map 
\[
\pi_8^{st}(X(\eta 2 ,5)\vee \Sigma^8\mathbb{S}) \to H_8(X(\eta 2, 5) \vee \Sigma^8\mathbb{S})=\mathbb{Z}
\]
is an isomorphism from the $\mathbb{Z}=\pi_8^{st}(\Sigma^8\bS)$ factor of $\pi_8^{st}(X(\eta 2 ,5)\vee \Sigma^8\mathbb{S})$ to $\mathbb{Z}$.  (The notation $X(\eta 2, 5)$ is from \cite{LOS_colored}; the appearance of this spectrum is only relevant to the calculation in so far as the highest-dimensional cell $\Sigma^8\bS$ of $\X_{col}(\mathscr{U}_3)$ splits off).  In particular, there is a lift of $U_3^2\in H_{8}(\closure{\cP_3}\atq{9})$ to an element of $\pi_8(\X_{col}(\mathscr{U}_3))$, as needed.  Note that the set of lifts is in correspondence with $\pi_8^{st}(X(\eta 2, 5))$; we do not assert that the lift is canonical.  
\end{proof}

Interestingly, if one attempts to find (or obstruct) a lift for $U_4^1$ in similar fashion for $U_3^2$ (respectively $U_3$), the required quantum degree of the 4-colored unknot corresponds (in the finite approximation of the infinite twist) to understanding the spectrum $\closure{\T_4^5}\atq{19}$, which is precisely the example cited in \cite[Section 9]{LOS_flowcats} (see also the computations in \cite{JLS_slNmatched}) for which the authors' calculus of framed flow categories is insufficient to decide whether there is an unattached sphere (which would guarantee the existence of $\projmap_4^1$) or an attaching map giving a more complicated stable homotopy type.  
It is known that $Sq^2$ vanishes in that grading, however the Khovanov spectrum in that quantum grading is either a wedge of three spheres, or the wedge of a single sphere and the cone of the nontrivial element in $\varepsilon \in \pi_2^{st}(S^0)$.  (See \cite{LOS} for more details).  Applying the same strategy as in Theorem \ref{thm:U3 does not lift}, we obtain:

\begin{corollary}
    The map $U_4\colon P_4\to P_4$ lifts to a map $\mathcal{U}_4\colon \mathcal{P}_4\to\mathcal{P}_4$ if and only if $\X^{19}(T(4,5))$ is a wedge of spheres.
\end{corollary}
\begin{remark}
    It is a corollary of the argument from Section \ref{sec:Obstructing U_n^k} that $2U_4$ lifts to a morphism of spectra (similarly, $2U_3$ lifts).  In general, if $r\in\mathbb{Z}$ annihilates $\pi^{st}_{(2n-2)k-1}(\closure{\cP_n}\atq{2nk})$, then $rU_n^k$ lifts.  Note that that stable homotopy group is finite, since the rational homology in that degree vanishes by \cite{Hog_polyaction}).  
\end{remark}

\subsection{Computation of the Khovanov spectrum of the 3-colored unknot}

The existence of the map $\projmap_3^2$ on the spectral projector $\proj_3$ allows us to verify a conjecture of Lobb-Orson-Sch\"{u}tz on the Khovanov spectrum of the 3-colored unknot.  Using \cite[Theorem 1.2]{Hog_polyaction} together with the bounded representative for $Q_3$ given in that paper as Example 3.41, it is not hard to show that the maps $U_3^k$ induce isomorphisms on 3-colored homology beyond the earliest $q$-gradings.  We can use this to complete the following computation.

\begin{theorem}[{{\cite[Conjecture 4.1]{LOS_colored}}}]
\label{thm:3colored unknot}
The colored Khovanov spectrum $\X:=\X_{col}(\mathscr{U}_3)$ (or equivalently, either the closed projector $\closure{\proj_3}\simeq\closure{\T_3^\infty}$ or the endomorphism spectrum $\End(\proj_3,\proj_3)$) has summands $\X\atq{j}$ as given in the following table:
\begin{center}
\setlength{\extrarowheight}{2pt}{
\begin{tabular}{||C|C||}
\hline
j & \X\atq{j} \\
\hline
\hline
-3 & \bS \\
\hline
-1 & \bS \\
\hline
1 & \Sigma^2\bS \\
\hline
3 & \Sigma^{-1} \bR P^5/\bR P^2\\
\hline
6m+5 \quad (m\geq 0) & \Sigma^{3+4m}\bS \vee \Sigma^{4+4m}\bS \\
\hline
6m+1 \quad (m\geq 0) & \Sigma^{1+4m}\bS \vee \Sigma^{2+4m}\bS \\
\hline
12m-3 \quad (m\geq 0) & X(\eta 2,8m-3) \vee \Sigma^{8m}\bS \\
\hline
12m+3 \quad (m\geq 0) & \Sigma^{8m+1}\bS \vee X(2\eta,8m+2) \\
\hline
\end{tabular}
}
\end{center}
where we have utilized the notation of \cite{LOS_colored}.
\end{theorem}
\begin{proof}
The map $\projmap_3^2$ is a stable homotopy equivalence by Whitehead's theorem, providing a periodicity $\Sigma^8\closure{\proj_3}\atq{j} \simeq \closure{\proj_3}\atq{j+12}$ for $j\geq 7$.  For all $j\leq 7$, the computation of the remaining cases had already been worked out by \cite{LOS_colored}, but we include a partial summary of it here.  For all $j\not\equiv 3 \mod 6$, the stable homotopy type is easily determined to be a Moore space for the corresponding homology.  Thus the only $q$-degrees which need to be computed by hand are 3, 9, and 15.  Degree 3 was already known in \cite{Wil_TorusLinks} via the computation in \cite{LS_steenrod}, while degrees 9 and 15 are handled via computer in \cite{LOS_colored} which led those authors to their conjecture.

This completes the proof.
\end{proof}

\begin{remark}
It should be emphasized that this result required no new computations.  Instead, it relied on the adjunction (ignoring shifts) $\End(\proj_3,\proj_3)\simeq \pi^{st}(\closure{\proj_3})$ allowing us to use the existing computations for $\closure{\proj_3}=\X_{col}(\mathscr{U}_3)$ to find an endomorphism that was known to give isomorphisms on homology.  Through Whitehead's theorem, this endomorphism then provided the periodicity which could extend the known computations.
\end{remark}

\section{Topological Hochschild homology does not give an invariant for links in $S^1\times S^2$}\label{sec:THH}

In Section \ref{sec:spTL cat} we viewed $(n,n)$-tangles as $T\in\Tang(\varnothing;2n)$ with corresponding Temperley-Lieb spectrum $\X(T)\in\SpmMod(\varnothing; \sarc_{2n})$ viewed as a module over the spectral arc algebra $\sarc_{2n}$ (and similarly for their homology $\Kh(T)$).  This notation is in accordance with the planar algebraic framework utilized in \cite{LLS_func}.  In this section however we will be concerned with the spectral $(\sarc_{n},\sarc_{n})$-bimodule associated to such $T$ as defined in \cite{LLS-Burnside}, where the notation $\X(T)$ is also used.  Here to avoid confusion we will denote such bimodules by $\X(T)_{n,n}$ (and similarly their homology will be denoted $\Kh(T)_{n,n}$).

In \cite{Roz_S1S2}, Lev Rozansky showed that the Hochschild homology of the Khovanov bimodules $Kh(T)_{2n,2n}$ for $(2n,2n)$-tangles $T$ could be used to define Khovanov homology groups for links in $S^1\times S^2$ formed as the $S^1$-closure of such $T$.  Rozansky also described how to compute this Hochschild homology via a different sort of limiting statement for infinite twists, and in \cite{Wil_S1S2} the second author expanded upon this limiting version to define Khovanov homology groups for (2-divisible) links in connect sums of $S^1\times S^2$ before using the approach to define a Lee deformation and Rasmussen $s$-invariant for such links in \cite{MMSW_S1S2sinvt}.

In \cite{LLS-Burnside}, Lawson-Lipshitz-Sarkar conjectured that the topological Hochschild homology of $\X(T)_{2n,2n}$ would similarly define a Khovanov spectrum for such links, and if this were the case one might hope to use this spectrum to improve the genus bounds coming from the $s$-invariant as in \cite{LS_sinvt}.  In this section however, we show that the computations of \cite{LOS_colored} used in the proof of Theorem \ref{thm:U3 does not lift} imply that topological Hochschild homology does \emph{not} give such an invariant.  

\begin{theorem}\label{thm:THH is not S1S2 invt}
The following topological Hochschild homologies are not equivalent:
\begin{equation}\label{eq:THH is not S1S2 invt}
\THH\left(\sarc_2; \X\left(
\vcenter{\hbox{\begin{tikzpicture}[xscale=.25,yscale=.8]
    \draw[thick] (-.5,.5) to[out=90,in=90,looseness=.3] (2.5,.5);
    \drawover{(.2,1)--(.2,0);}
    \drawover{(1.8,1)--(1.8,0);}
    \drawover{(-.5,.5) to[out=-90,in=-90,looseness=.3] (2.5,.5);}
\end{tikzpicture}}} \right)_{2,2}\right)
\not\simeq
\THH\left(\sarc_2; \X\left(
\vcenter{\hbox{\begin{tikzpicture}[xscale=.25,yscale=.8]
    \draw[thick] (.4,1)--(.4,0);
    \draw[thick] (1.8,1)--(1.8,0);
    \draw[thick] (-.5,.5) to[out=90,in=90,looseness=.5] (-2.5,.5) to[out=-90,in=-90,looseness=.5] (-.5,.5);
\end{tikzpicture}}} \, \, \right)_{2,2}\right),
\end{equation}
despite the fact that the corresponding links in $S^1\times S^2$ are isotopic.  
\end{theorem} 

The proof relies on the following definition, as well as a lemma which we will prove in Section \ref{sec:Proof of Lemma p20 computes THH}.

\begin{definition}\label{def:P20 not semi simp}
Let $\proj_{2,0}$ denote the following cone
\[\proj_{2,0}:=\Cone\left(\Imod_2\hookrightarrow \proj_2\right),\]
where the map is the obvious inclusion of $\Imod_2=\augJM_2^0$ into $\proj_2\simeq\augJM_2^\infty$.
\end{definition}

\begin{lemma}\label{lem:P20 computes THH}
For any $(2,2)$-tangle $T$, the topological Hochschild homology of $\sarc_2$ with coefficients in $\X(T)$ can be computed as
\begin{equation}\label{eq:P20 computes THH}
\THH\left(\sarc_2; \X\left(
\vcenter{\hbox{\begin{tikzpicture}[xscale=.25,yscale=.8]
    \draw[thick] (.2,1)--(.2,0);
    \draw[thick] (1.8,1)--(1.8,0);
    \whitebox{-.2,.15}{2.2,.85}
    \node[scale=.8] at (1,.5) {$T$};
\end{tikzpicture}}} \right)_{2,2} \right)
\simeq
\X\left(
\vcenter{\hbox{\begin{tikzpicture}[xscale=.25,yscale=.6]
    \draw[thick] (.2,2)--(.2,0);
    \draw[thick] (1.8,2)--(1.8,0);
    \whitebox{-.2,.15}{2.2,.85}
    \node[scale=.7] at (1,.5) {$T$};
    \whitebox{-.2,1.15}{2.2,1.85}
    \node[scale=.7] at (1,1.5) {$\proj_{2,0}$};
    \draw[thick]
        (.2,0) to[out=-90,in=-90,looseness=.2] (4,0)--(4,2) to[out=90,in=90,looseness=.2] (.2,2)
        (1.8,0) to[out=-90,in=-90,looseness=.2] (3,0)--(3,2) to[out=90,in=90,looseness=.2] (1.8,2);
\end{tikzpicture}}} \right).
\end{equation}
\end{lemma}

Compare Definition \ref{def:P20 not semi simp} with the definition of the so-called \emph{lowest weight projector} in \cite{Roz_S1S2}, and Lemma \ref{lem:P20 computes THH} with the corresponding statement \cite[Equation 6.4]{Roz_S1S2}  about computing Hochschild homology for Khovanov bimodules (Rozansky's notation $\mathbf{P}_n$ there corresponds to our $\proj_{2,0}$ here).  

\begin{proof}[Proof of Theorem \ref{thm:THH is not S1S2 invt} assuming Lemma \ref{lem:P20 computes THH}]
In \cite[Section 4.3]{LOS_colored} it is shown that
\begin{equation}\label{eq:P2 cannot unlink meridian}
\X^q \left(
\vcenter{\hbox{\begin{tikzpicture}[xscale=.25,yscale=.6]
    \draw[thick] (-.5,.5) to[out=90,in=90,looseness=.3] (2.5,.5);
    \drawover{(.2,2)--(.2,0);}
    \drawover{(1.8,2)--(1.8,0);}
    \drawover{(-.5,.5) to[out=-90,in=-90,looseness=.3] (2.5,.5);}
    \whitebox{-.2,1.15}{2.2,1.85}
    \node[scale=.7] at (1,1.5) {$\proj_2$};
    \draw[thick]
        (.2,0) to[out=-90,in=-90,looseness=.2] (4,0)--(4,2) to[out=90,in=90,looseness=.2] (.2,2)
        (1.8,0) to[out=-90,in=-90,looseness=.2] (3,0)--(3,2) to[out=90,in=90,looseness=.2] (1.8,2);
\end{tikzpicture}}}\,
\right)\not\simeq \X^q \left(
\vcenter{\hbox{\begin{tikzpicture}[xscale=.25,yscale=.6]
    \draw[thick] (.2,2)--(.2,0)
            (1.8,2)--(1.8,0);
    \whitebox{-.2,1.15}{2.2,1.85}
    \node[scale=.7] at (1,1.5) {$\proj_2$};
    \draw[thick]
        (.2,0) to[out=-90,in=-90,looseness=.2] (4,0)--(4,2) to[out=90,in=90,looseness=.2] (.2,2)
        (1.8,0) to[out=-90,in=-90,looseness=.2] (3,0)--(3,2) to[out=90,in=90,looseness=.2] (1.8,2);
    \draw[thick]
        (-.5,.5) to[out=90,in=90,looseness=.6] (-2.5,.5) to[out=-90,in=-90,looseness=.6] (-.5,.5);
\end{tikzpicture}}}\,
\right)
\end{equation}
for all $q\geq 7$.  Meanwhile, the exact triangle
\[
\Imod_2 \to \cP_2\to \cP_{2,0}
\]
gives, after gluing, two further exact triangles:
\begin{gather*}
\X\left(\vcenter{\hbox{\begin{tikzpicture}[xscale=.25,yscale=.6]
			\draw[thick] (-.5,.5) to[out=90,in=90,looseness=.3] (2.5,.5);
			\drawover{(.2,2)--(.2,0);}
			\drawover{(1.8,2)--(1.8,0);}
			\drawover{(-.5,.5) to[out=-90,in=-90,looseness=.3] (2.5,.5);}
			\whitebox{-.2,1.15}{2.2,1.85}
			\node[scale=.7] at (1,1.5) {$\Imod_2$};
			\draw[thick]
			(.2,0) to[out=-90,in=-90,looseness=.2] (4,0)--(4,2) to[out=90,in=90,looseness=.2] (.2,2)
			(1.8,0) to[out=-90,in=-90,looseness=.2] (3,0)--(3,2) to[out=90,in=90,looseness=.2] (1.8,2);
\end{tikzpicture}}}\right)
\to
\X\left(\vcenter{\hbox{\begin{tikzpicture}[xscale=.25,yscale=.6]
			\draw[thick] (-.5,.5) to[out=90,in=90,looseness=.3] (2.5,.5);
			\drawover{(.2,2)--(.2,0);}
			\drawover{(1.8,2)--(1.8,0);}
			\drawover{(-.5,.5) to[out=-90,in=-90,looseness=.3] (2.5,.5);}
			\whitebox{-.2,1.15}{2.2,1.85}
			\node[scale=.7] at (1,1.5) {$\proj_2$};
			\draw[thick]
			(.2,0) to[out=-90,in=-90,looseness=.2] (4,0)--(4,2) to[out=90,in=90,looseness=.2] (.2,2)
			(1.8,0) to[out=-90,in=-90,looseness=.2] (3,0)--(3,2) to[out=90,in=90,looseness=.2] (1.8,2);
\end{tikzpicture}}}\right)
\to
\X\left(\vcenter{\hbox{\begin{tikzpicture}[xscale=.25,yscale=.6]
			\draw[thick] (-.5,.5) to[out=90,in=90,looseness=.3] (2.5,.5);
			\drawover{(.2,2)--(.2,0);}
			\drawover{(1.8,2)--(1.8,0);}
			\drawover{(-.5,.5) to[out=-90,in=-90,looseness=.3] (2.5,.5);}
			\whitebox{-.2,1.15}{2.2,1.85}
			\node[scale=.7] at (1,1.5) {$\proj_{2,0}$};
			\draw[thick]
			(.2,0) to[out=-90,in=-90,looseness=.2] (4,0)--(4,2) to[out=90,in=90,looseness=.2] (.2,2)
			(1.8,0) to[out=-90,in=-90,looseness=.2] (3,0)--(3,2) to[out=90,in=90,looseness=.2] (1.8,2);
\end{tikzpicture}}}\right)
\\
\X\left(\vcenter{\hbox{\begin{tikzpicture}[xscale=.25,yscale=.6]
			\draw[thick] (.2,2)--(.2,0)
			(1.8,2)--(1.8,0);
			\whitebox{-.2,1.15}{2.2,1.85}
			\node[scale=.7] at (1,1.5) {$\Imod_2$};
			\draw[thick]
			(.2,0) to[out=-90,in=-90,looseness=.2] (4,0)--(4,2) to[out=90,in=90,looseness=.2] (.2,2)
			(1.8,0) to[out=-90,in=-90,looseness=.2] (3,0)--(3,2) to[out=90,in=90,looseness=.2] (1.8,2);
			\draw[thick]
			(-.5,.5) to[out=90,in=90,looseness=.6] (-2.5,.5) to[out=-90,in=-90,looseness=.6] (-.5,.5);
\end{tikzpicture}}}\right)
\to
\X\left(\vcenter{\hbox{\begin{tikzpicture}[xscale=.25,yscale=.6]
			\draw[thick] (.2,2)--(.2,0)
			(1.8,2)--(1.8,0);
			\whitebox{-.2,1.15}{2.2,1.85}
			\node[scale=.7] at (1,1.5) {$\proj_2$};
			\draw[thick]
			(.2,0) to[out=-90,in=-90,looseness=.2] (4,0)--(4,2) to[out=90,in=90,looseness=.2] (.2,2)
			(1.8,0) to[out=-90,in=-90,looseness=.2] (3,0)--(3,2) to[out=90,in=90,looseness=.2] (1.8,2);
			\draw[thick]
			(-.5,.5) to[out=90,in=90,looseness=.6] (-2.5,.5) to[out=-90,in=-90,looseness=.6] (-.5,.5);
\end{tikzpicture}}}\right)
\to
\X\left(\vcenter{\hbox{\begin{tikzpicture}[xscale=.25,yscale=.6]
			\draw[thick] (.2,2)--(.2,0)
			(1.8,2)--(1.8,0);
			\whitebox{-.2,1.15}{2.2,1.85}
			\node[scale=.7] at (1,1.5) {$\proj_{2,0}$};
			\draw[thick]
			(.2,0) to[out=-90,in=-90,looseness=.2] (4,0)--(4,2) to[out=90,in=90,looseness=.2] (.2,2)
			(1.8,0) to[out=-90,in=-90,looseness=.2] (3,0)--(3,2) to[out=90,in=90,looseness=.2] (1.8,2);
			\draw[thick]
			(-.5,.5) to[out=90,in=90,looseness=.6] (-2.5,.5) to[out=-90,in=-90,looseness=.6] (-.5,.5);
\end{tikzpicture}}}\right)
\end{gather*}
The terms with $\Imod_2$ are Khovanov spectra of (finite) links, and so are contractible in all but finitely many $q$-degrees.  Thus, in all but finitely many $q$-degrees we have the terms with $\proj_2$ equivalent to the terms with $\proj_{2,0}$. Combining this with Equation \eqref{eq:P2 cannot unlink meridian} implies
\begin{equation}\label{eq:P20 cannot unlink meridian}
\X^q \left(
\vcenter{\hbox{\begin{tikzpicture}[xscale=.25,yscale=.6]
    \draw[thick] (-.5,.5) to[out=90,in=90,looseness=.3] (2.5,.5);
    \drawover{(.2,2)--(.2,0);}
    \drawover{(1.8,2)--(1.8,0);}
    \drawover{(-.5,.5) to[out=-90,in=-90,looseness=.3] (2.5,.5);}
    \whitebox{-.2,1.15}{2.2,1.85}
    \node[scale=.7] at (1,1.5) {$\proj_{2,0}$};
    \draw[thick]
        (.2,0) to[out=-90,in=-90,looseness=.2] (4,0)--(4,2) to[out=90,in=90,looseness=.2] (.2,2)
        (1.8,0) to[out=-90,in=-90,looseness=.2] (3,0)--(3,2) to[out=90,in=90,looseness=.2] (1.8,2);
\end{tikzpicture}}}\,
\right) \not\simeq \X^q \left(
\vcenter{\hbox{\begin{tikzpicture}[xscale=.25,yscale=.6]
    \draw[thick] (.2,2)--(.2,0)
            (1.8,2)--(1.8,0);
    \whitebox{-.2,1.15}{2.2,1.85}
    \node[scale=.7] at (1,1.5) {$\proj_{2,0}$};
    \draw[thick]
        (.2,0) to[out=-90,in=-90,looseness=.2] (4,0)--(4,2) to[out=90,in=90,looseness=.2] (.2,2)
        (1.8,0) to[out=-90,in=-90,looseness=.2] (3,0)--(3,2) to[out=90,in=90,looseness=.2] (1.8,2);
    \draw[thick]
        (-.5,.5) to[out=90,in=90,looseness=.6] (-2.5,.5) to[out=-90,in=-90,looseness=.6] (-.5,.5);
\end{tikzpicture}}}\,
\right)
\end{equation}
for infinitely many $q$ degrees as well.  Lemma \ref{lem:P20 computes THH} then completes the proof.
\end{proof}

As noted in \cite{LOS_colored}, the proof above does not contradict the results in \cite{Roz_S1S2,Wil_S1S2} precisely because the inequivalence of Equation \eqref{eq:P2 cannot unlink meridian} becomes an equivalence in homology.  That is to say, one can show that
\[Kh\left(
\vcenter{\hbox{\begin{tikzpicture}[xscale=.25,yscale=.6]
    \draw[thick] (-.5,.5) to[out=90,in=90,looseness=.3] (2.5,.5);
    \drawover{(.2,2)--(.2,0);}
    \drawover{(1.8,2)--(1.8,0);}
    \drawover{(-.5,.5) to[out=-90,in=-90,looseness=.3] (2.5,.5);}
    \whitebox{-.2,1.15}{2.2,1.85}
    \node[scale=.7] at (1,1.5) {$P_2$};
    \draw[thick]
        (.2,0) to[out=-90,in=-90,looseness=.2] (4,0)--(4,2) to[out=90,in=90,looseness=.2] (.2,2)
        (1.8,0) to[out=-90,in=-90,looseness=.2] (3,0)--(3,2) to[out=90,in=90,looseness=.2] (1.8,2);
\end{tikzpicture}}}\,
\right) \cong Kh \left(
\vcenter{\hbox{\begin{tikzpicture}[xscale=.25,yscale=.6]
    \draw[thick] (.2,2)--(.2,0)
            (1.8,2)--(1.8,0);
    \whitebox{-.2,1.15}{2.2,1.85}
    \node[scale=.7] at (1,1.5) {$P_2$};
    \draw[thick]
        (.2,0) to[out=-90,in=-90,looseness=.2] (4,0)--(4,2) to[out=90,in=90,looseness=.2] (.2,2)
        (1.8,0) to[out=-90,in=-90,looseness=.2] (3,0)--(3,2) to[out=90,in=90,looseness=.2] (1.8,2);
    \draw[thick]
        (-.5,.5) to[out=90,in=90,looseness=.6] (-2.5,.5) to[out=-90,in=-90,looseness=.6] (-.5,.5);
\end{tikzpicture}}}\,
\right)\]
in all but finitely many $q$ degrees, which are precisely the degrees in which $Kh(\Imod_2)$ is  non-trivial.

\subsection{Proof of Lemma \ref{lem:P20 computes THH}}
\label{sec:Proof of Lemma p20 computes THH}
We begin by presenting the simplified model for $\proj_2$ prescribed by Theorem \ref{thm:P2 simp}, as follows.  We will return to the convention of section \ref{sec:Obstructing U_n^k}, usually suppressing the Khovanov spectrum functor $\X$ from the notation.    
\begin{equation}\label{eq:p2-restated}
\proj_2 \simeq \left( \cdots
\xrightarrow{\ILhresDotT \, - \, \ILhresDotB}
\ILhres[scale=.5]
\xrightarrow{\ILhresDotT \, + \, \ILhresDotB}
\ILhres[scale=.5]
\xrightarrow{\ILhresDotT \, - \, \ILhresDotB}
\ILhres[scale=.5]
\xrightarrow{\ILvertsaddle}
\ILvres[scale=.5]
\right).
\end{equation}
Recall that the right-hand side of \eqref{eq:p2-restated} is interpreted as (the realization of) a semisimplicial object, where there is a unique nontrivial face map between adjacent objects.  We write $\ssimpproj_2$ for the semisimplicial object appearing on the right-hand side.  There is a map $\iota\colon \Imod_2\rightarrow \proj_2$ coming from including $\Imod_2$ as the last term in this sequence.  In fact, $\iota$ is the realization of a map $\iota^{\mathrm{ssimp}}$ of semisimplicial spectra, from $\Imod_2 \to \ssimpproj_2$ where we abuse notation and write $\Imod_2$ for the semisimplicial spectrum whose only nontrivial entry is at $[0]$, at which it is $\Imod_2$.   The cone $\mathrm{Cone}(\iota)$ is then the realization:
\[
\mathrm{Cone}(\iota) \simeq | \mathrm{Cone}\, (\iota^{\mathrm{ssimp}})|.
\] 
Using that the cone of $\iota^{\mathrm{ssimp}}$ is contractible at $[0]\in \Delta_{inj}$,  we can present the cone from a semisimplicial object:
\[
\mathrm{Cone}\, (\iota^{\mathrm{ssimp}})\simeq \left( \cdots
\xrightarrow{\ILhresDotT \, - \, \ILhresDotB}
\ILhres[scale=.5]
\xrightarrow{\ILhresDotT \, + \, \ILhresDotB}
\ILhres[scale=.5]
\xrightarrow{\ILhresDotT \, - \, \ILhresDotB}
\ILhres[scale=.5]
\xrightarrow{}
*
\right).
\] 
As in \eqref{eq:p2-restated}, the right-hand semisimplicial spectrum has a unique nontrivial face map between adjacent objects, at $d_0$.  By inspection, this semisimplicial spectrum is (cf. \cite[III.5]{Goerss-Jardine}) the suspension $\Sigma \proj_{2,0}^{min}$ of:
\begin{equation}\label{eq:P20}
\proj_{2,0}^{min} :=\left( \cdots
\xrightarrow{\ILhresDotT \, - \, \ILhresDotB}
\ILhres[scale=.5]
\xrightarrow{\ILhresDotT \, + \, \ILhresDotB}
\ILhres[scale=.5]
\xrightarrow{\ILhresDotT \, - \, \ILhresDotB}
\ILhres[scale=.5]
\right).
\end{equation}
We write $\proj_{2,0}$ for the realization of $\proj_{2,0}^{min}$. In the expression \eqref{eq:P20}, the arrows indicate $d_0$ of the semisimplicial object (all other $d_i$ are trivial).  The superscript indicates that $\proj_{2,0}^{min}$ is a smaller presentation of $\proj_{2,0}$ than the semisimplicial object $\proj_{2,0}^{ssimp}$, which we introduce below.

\begin{remark} The spectrum $\proj_{2,0}$ is a spectrification of a projector (in the category of chain complexes over $H_4$) $P_{2,0}$; and the presentations of $\proj_{2}$ and $\proj_{2,0}$ as semisimplicial spectra lift the standard presentations of $P_2$ and $P_{2,0}$ as chain complexes.  Indeed, more generally there are many `projectors' in $\TL_n$ for larger $n$; see \cite{Cooper-Hogancamp} for the details.  We conjecture that all of these projectors have well-defined (under suitable hypotheses) lifts to spectral modules. 
\end{remark}

More relevant for our current problem of considering THH, in the chain setting, Rozansky proves \cite[Theorem 6.5]{Roz_S1S2} stating that $P_{2,0}\simeq \chainsfunc(\proj_{2,0})$ can be used to compute the Hochschild homology of $(2,2)$-tangle bimodules by showing that $P_{2,0}$ provides a projective resolution of the identity $(H_2,H_2)$-bimodule (and more generally for other projectors $P_{2n,0}$).

In the setting of spectra, in place of the notion of projective resolutions, one uses cofibrant/fibrant replacements.  Instead of directly choosing a model category structure on our categories of module spectra to perform cofibrant/fibrant replacement, we use the bar construction.

That is, we are in the setting of \cite[Definition 7.2]{LLS_tangles}, so that $\THH(\sarc_2,M)$, for $M$ a $(\sarc_2,\sarc_2)$-bimodule, is given by the homotopy colimit of the diagram
\begin{equation}\label{eq:thh}
\rightfourarrow M\wedge \sarc_2\wedge \sarc_2 \rightthreearrow M \wedge \sarc_2 \rightrightarrows M.
\end{equation}
Here, the diagram is simpler than that of \cite[Definition 7.2]{LLS_tangles}, because there is a unique idempotent for $H_2$; see Remark \ref{rmk:comparing our THH with LLS THH} below.
The arrows are given by the multiplication map of $\sarc_2\wedge \sarc_2\to \sarc_2$, together with the right-action of the leftmost $\sarc_2$ on $M$, and the left action of the rightmost $\sarc_2$ on $M$.   The homotopy colimit in \eqref{eq:thh} is exactly the smash product of the bar resolution of the identity $(\sarc_2,\sarc_2)$-bimodule, $\sarc_2$, with $M$.  Alternatively, \eqref{eq:thh} is the tensor product $|B(\sarc_2)\otimes_{\sarc_2\otimes \sarc_2^{op}} M|$, where $B(\sarc_2)$ is the bar resolution.  

For $T$ a $(2,2)$-tangle, we continue to use $\X(T)_{2,2}$ to denote the Khovanov $(\sarc_2,\sarc_2)$-bimodule associated to $T$; this is not the same as the Temperley-Lieb spectrum $\X(T)\in\SpmMod(\varnothing;\sarc_4)$ viewing $T$ as a $(\varnothing ;4)$-tangle.  We note that $T$ also determines a $(4;\varnothing)$-tangle $M_T$ as shown in Figure \ref{fig:M as H4;0 module}, with corresponding Khovanov module $\X(M_T)\in\SpmMod(\sarc_4;\varnothing)$.

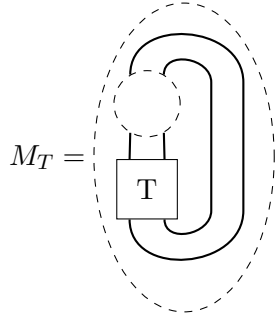
\begin{figure}
    \[
    M_T=
    \vcenter{\hbox{\begin{tikzpicture}[xscale=.15mm,yscale=.25mm]
		\emptydisk{E}{0}{.8}
		\modbox{T}{0}{-.8}{\text{T}}
        \node[align=center,fit={(E disk)(T disk)(3.5,0)}, draw,circle,xscale=.7,yscale=1.2,dashed] (ET) {};
        \draw[thick] (T box.\disksw) to[out=-90,in=-90,looseness=.8] (3,-1.2)--(3,1.2) to[out=90,in=90,looseness=.8] (E disk.\disknw);
        \draw[thick] (T box.\diskse) to[out=-90,in=-90,looseness=.8] (2,-1.2)--(2,1.2) to[out=90,in=90,looseness=.8] (E disk.\diskne);
        \draw[thick] (E disk.\disksw)--(T box.\disknw);
        \draw[thick] (E disk.\diskse)--(T box.\diskne);
    \end{tikzpicture}}}
    \]
    \caption{The $(4;\emptyset)$-tangle $M_T$ associated to the $(2,2)$-tangle $T$.}
    \label{fig:M as H4;0 module}
\end{figure}

In what follows, we give an explicit recipe for the diagram \eqref{eq:thh} as the realization:
\[
|\ssimpproj_{2,0}\otimes_{\sarc_4}\X(M_T) |
\]
In the appendix, we also prove, for independent interest, Lemma \ref{lem:induction-spectrum}, which gives:
\[
\ssimpproj_{2,0}\simeq B(\sarc_2)\otimes_{\sarc_2\otimes \sarc_2^{op}} \sarc_4,
\] 
from which it follows formally that $THH(\sarc_2,M)=|B(\sarc_2)\otimes_{\sarc_2\otimes \sarc_2^{op}}^{\mathbb{L}} \X(T)_{2,2}|\simeq |\ssimpproj_{2,0}\otimes_{\sarc_4}^{\mathbb{L}}\X(T)|$ viewing $\X(T)$ as simply a spectral $\sarc_4$-module.

We first give a new semisimplicial description of $\proj_{2,0}$ as the realization of a semi-simplicial $(\varnothing;\sarc_4)$-module which we denote $\ssimpproj_{2,0}$, to move closer to a bar resolution.  

Next we will construct a second semi-simplicial $(\varnothing;\sarc_4)$-module $\ssimpBar_2$ such that the realization $\ssimpBar_2\otimes_{\sarc_4} \X(M_T)$  will give the bar construction appearing in \eqref{eq:thh} for $M=\X(T)_{2,2}$.  

Finally we will construct a map of semi-simplicial $(\varnothing;\sarc_4)$-modules
\[\ssimpproj_{2,0} \xrightarrow{F} \ssimpBar_2,\]
and show that the map $|F|$ on the realizations induces an isomorphism on homology, so that Whitehead's theorem implies that $|F|$ is a stable equivalence.  That gives 
\begin{equation}\label{eq:thh-steps}
THH(\sarc_2,\X(T)_{2,2})=\left| \ssimpBar_{2}\otimes_{\sarc_4} \X(M_T) \right|\simeq \left| \ssimpBar_{2}\right|\otimes_{\sarc_4}^{\mathbb{L}}\X(M_T)\simeq  \left|\ssimpproj_{2,0}\right| \otimes_{\sarc_4}^{\mathbb{L}} \X(M_T).
\end{equation}
The gluing theorem for tangle invariants identifies the last term with the right-hand side of Lemma \ref{lem:P20 computes THH}, which will complete the proof of the lemma.  

To begin on this path, we define $\ssimpproj_{2,0}$ to be the semisimplicial $(\varnothing;\sarc_4)$-module whose terms are all copies of the tangle module for the split tangle $\ILhres$, and whose face maps on the $n$th level are all zero (basepoint maps) besides $d_0=\ILhresDotT$ and $d_n=\ILhresDotB$.  It is direct to check that this object $\ssimpproj_{2,0}$ satisfies the semisimplicial identitites.

\begin{remark}\label{rmk:strict}
We are being slightly imprecise, in that we need to identify $\X(\ILhresDotT)\circ \X(\ILhresDotB)$ with $\X(\ILhresDotB)\circ \X(\ILhresDotT)$ for the semisimplicial identities to hold.  However, in \cite[Section 2.11]{LLS_tangles} the Khovanov-Burnside functor is constructed as a functor from a certain \emph{embedded cobordism category} $\mathrm{Cob}_e$ of circles in $(0,1)^2$, to the Burnside category, so that these composites agree up to coherent homotopy, but are not identical.  The morphisms $\X(\ILhresDotT),\X(\ILhresDotB)$ are both associated to particular embedded cobordisms (namely, the dot map is the same as a birth followed by a merge.  Note, these are not viewed as embedded cobordisms in a 4-manifold, but as embedded cobordisms in $\mathbb{R}^3$), and in fact we can extend the embedded cobordisms pictured in \eqref{eq:p2-ssimp-definition} to a homotopy-coherent semi-simplicial object in $\mathrm{Cob}_e$.  The semi-simplicial object $\ssimpproj_{2,0}$ is then obtained by applying the machinery of \cite{LLS_tangles} to that diagram in $\mathrm{Cob}_e$.
\end{remark}
We represent $\ssimpproj_{2,0}$ graphically as
\begin{equation}\label{eq:p2-ssimp-definition}
\ssimpproj_{2,0}:=\left(
\ILtikzpic[xscale=3]{
\node(dots) at (.5,0) {$\cdots$};
\node(res3) at (1,0) {$\ILhres[scale=1]$};
\node(res2) at (2,0) {$\ILhres[scale=1]$};
\node(res1) at (3,0) {$\ILhres[scale=1]$};
\node(enddots) at (3.5,0) {$\cdots$};
\foreach \p/\q in {res3/res2, res2/res1} {
    \draw[->] (\p) to[out=45,in=180-45,looseness=1.25] node[pos=.5,above=.2](\p\q mapT){} (\q);
    \node at (\p\q mapT) {$d_0=\ILhresDotT$};
    \draw[->,dashed] (\p) to node[pos=.5,above]{$\varnothing$} (\q);
    \draw[->] (\p) to[out=-45,in=180+45,looseness=1.25] node[pos=.5,below=.2](\p\q mapB){} (\q);
    }
\node at (res3res2mapB) {$d_n=\ILhresDotB$};
\node at (res2res1mapB) {$d_{n-1}=\ILhresDotB$};
}\right).
\end{equation}
We have that $|\ssimpproj_{2,0}|\simeq |\proj_{2,0}^{min}|$ using $|\proj_{2}^{min}|\simeq |\ssimpproj_2|$ together with (coherent) uniqueness of the defining map $\Imod_2\to \proj_2$ of a projector.  

Next we define $\ssimpBar_2$ to be the semisimplicial $(\varnothing;\sarc_4)$-module represented graphically as
\begin{equation}\label{eq:bsimp-def}
\ssimpBar_2 := \left(
\ILtikzpic[xscale=4]{
\node(dots) at (.5,0) {$\cdots$};
\node(res3) at (1,0) {$\ILtikzpic[scale=1]{
                    \draw[thick] (0,2) to[out=-70,in=250] (1,2);
                    \draw[thick] (0,-1) to[out=70,in=110] (1,-1);
                    \draw[thick] (.5,0) circle (.35);
                    \draw[thick] (.5,1) circle (.35);
                    }$};
\node(res2) at (2,0) {$\ILtikzpic[scale=1]{
                    \draw[thick] (0,1) to[out=-70,in=250] (1,1);
                    \draw[thick] (0,-1) to[out=70,in=110] (1,-1);
                    \draw[thick] (.5,0) circle (.35);
                    }$};
\node(res1) at (3,0) {$\ILhres[scale=1]$};
\draw[->] (res2.north east) to[out=0,in=135] node[pos=.5,above,arrows=-] {$d_0=\ILtikzpic[scale=.35]{
                    \draw[thick] (0,1) to[out=-70,in=250] node[pos=.5,inner sep=0](top){} (1,1);
                    \draw[thick] (0,-1) to[out=70,in=110] (1,-1);
                    \draw[thick] (.5,0) circle (.35);
                    \draw[thick,red] (.5,.35)--(top);
                    }$} (res1);
\draw[->] (res2.south east) to[out=0,in=-135] node[pos=.5,below,arrows=-] {$d_1=\ILtikzpic[scale=.35]{
                    \draw[thick] (0,1) to[out=-70,in=250] (1,1);
                    \draw[thick] (0,-1) to[out=70,in=110] node[pos=.5,inner sep=0](bot){} (1,-1);
                    \draw[thick] (.5,0) circle (.35);
                    \draw[thick,red] (.5,-.35)--(bot);
                    }$} (res1);
\draw[->] (res3.north east) to[out=0,in=120] node[pos=.5,above,arrows=-] {$d_0=\ILtikzpic[scale=.35]{
                    \draw[thick] (0,2) to[out=-70,in=250] node[pos=.5,inner sep=0](top){} (1,2);
                    \draw[thick] (0,-1) to[out=70,in=110] (1,-1);
                    \draw[thick] (.5,0) circle (.35);
                    \draw[thick] (.5,1) circle (.35);
                    \draw[thick,red] (.5,1.35)--(top);
                    }$} (res2);
\draw[->] (res3) to[out=0,in=180] node[pos=.4,above,arrows=-] {$d_1=\ILtikzpic[scale=.35]{
                    \draw[thick] (0,2) to[out=-70,in=250] (1,2);
                    \draw[thick] (0,-1) to[out=70,in=110] (1,-1);
                    \draw[thick] (.5,0) circle (.35);
                    \draw[thick] (.5,1) circle (.35);
                    \draw[thick,red] (.5,.35)--(.5,.65);
                    }$} (res2);
\draw[->] (res3.south east) to[out=0,in=-120] node[pos=.5,below,arrows=-] {$d_2=\ILtikzpic[scale=.35]{
                    \draw[thick] (0,2) to[out=-70,in=250] (1,2);
                    \draw[thick] (0,-1) to[out=70,in=110] node[pos=.5,inner sep=0](bot){} (1,-1);
                    \draw[thick] (.5,0) circle (.35);
                    \draw[thick] (.5,1) circle (.35);
                    \draw[thick,red] (.5,-.35)--(bot);
                    }$} (res2);
}\right).
\end{equation}
In words, the $n$th term in $\ssimpBar_2$ consists of $n$ circles arranged in a line between a pair of turnbacks at the top and bottom.  The face maps then consist of the various saddle maps between adjacent circles/turnbacks.  A similar discussion as in Remark \ref{rmk:strict} is used to construct the diagram.

The gluing theorem gives that $\X(\ILhres)\otimes_{\sarc_4}^{\mathbb{L}}\X(M_T)$ recovers the underlying spectrum of the bimodule $\X(T)_{2,2}$, which is by definition the Khovanov spectrum of the (unique) closure of $T$ as a $(2,2)$-tangle.  More generally, the $n$-th term of Equation \eqref{eq:thh} agrees with $\ssimpBar_2([n])\otimes^{\mathbb{L}}_{\sarc_4} \X(M_T)$.  To see that the arrows of \eqref{eq:thh} agree with those coming from \eqref{eq:bsimp-def}, we note that roughly speaking these arrows come from the same cobordisms.  To make this more precise, we can identify the Burnside functors underlying \eqref{eq:thh} and $\ssimpBar_2\otimes_{\sarc 4} \X(M_T)$ directly; a more complicated instance of a similar argument is carried out in the appendix.  With the underlying Burnside functors identified, it follows that \eqref{eq:thh}, for $M=\X(T)_{2,2}$, agrees with $\left| \ssimpBar_{2}\otimes_{\sarc_4}\X(M_T)\right|$.

\begin{remark}\label{rmk:comparing our THH with LLS THH}
Compare this to the construction used for computing $\THH$ of (a single summand of) the spectral platform algebra in \cite[Section 5]{LLS_CK_Spectra}, where various crossingless matchings of the points on both the top and bottom must be considered.  Here with only two points on either the top or the bottom, there is only one crossingless matching available.
\end{remark}

Finally we define the map $\ssimpproj_{2,0}\xrightarrow{F}\ssimpBar_2$ on the $n$th term to be the map induced by birthing $n$ dotted circles.  By definition, $F$ will commute with $d_0$ and $d_n$, while commutation with the other $d_i$ relies on the fact that having two dots on a connected component induces the basepoint map (via the empty correspondence in the relevant Burnside category).

To show that $|F|$ gives an equivalence on realizations, we consider the effect of $|F|$ on homology.  We will abuse notation, and write $\Kc(\left| \ssimpBar_2\right|)$ (and similarly for $\ssimpproj_{2,0}$ ) for the Khovanov chain complex of the diagram in $\mathrm{Cob}_e$ in equation \eqref{eq:bsimp-def}.  To be more precise, recall that the chain functor applied to $\left| \ssimpBar_2\right|$ is quasi-isomorphic to this Khovanov complex, though not identical to it, so that the complex (rather than its homotopy type) $\Kc(\left| \ssimpBar \right|)$ is not determined by the spectrum $\ssimpBar$.  Upon passing to the relevant chain complexes, we can find a reverse map $\Kc^*(\left|\ssimpBar_2\right|)\xrightarrow{G} \Kc^*(\left|\ssimpproj_{2,0}\right|)$ defined by capping off each of the disjoint circles (these caps are not dotted).  Some computation shows that these capping maps do define a chain map $G$, and then it is clear that $G\circ \Kc(|F|)$ is the identity on chains, and similarly for $\Kc(|F|)\circ G$, so that $\Kc(|F|)$ gives an isomorphism on homology.  The quasi-isomorphisms $\mathcal{C}_h(\left|\ssimpBar_2\right|) \to \Kc (\left|\ssimpproj_2\right| )$, etc., are compatible with the maps $\mathcal{C}_h (|F|)$ and $\Kc(F)$ (again having abused notation), so that we also have that $\mathcal{C}_h(|F|)$ induced an isomorphism on homology.  By Whitehead's Theorem, $|F|$ is a stable equivalence of spectral $\sarc_4$-modules:

\begin{equation}\label{eq:simpB2 equiv to simp P20}
\left|\ssimpBar_2\right|\simeq\left| \ssimpproj_{2,0}\right|.
\end{equation}

Applying equation \eqref{eq:thh-steps}, together with
\[
\left|\ssimpproj_{2,0}\right| \otimes^{\mathbb{L}}_{\sarc_4} M_T \\
\simeq \X\left(
\vcenter{\hbox{\begin{tikzpicture}[xscale=.25,yscale=.6]
    \draw[thick] (.2,2)--(.2,0);
    \draw[thick] (1.8,2)--(1.8,0);
    \whitebox{-.2,.15}{2.2,.85}
    \node[scale=.7] at (1,.5) {$T$};
    \whitebox{-.2,1.15}{2.2,1.85}
    \node[scale=.7] at (1,1.5) {$\proj_{2,0}$};
    \draw[thick]
        (.2,0) to[out=-90,in=-90,looseness=.2] (4,0)--(4,2) to[out=90,in=90,looseness=.2] (.2,2)
        (1.8,0) to[out=-90,in=-90,looseness=.2] (3,0)--(3,2) to[out=90,in=90,looseness=.2] (1.8,2);
\end{tikzpicture}}} \right),
\]
completes the proof of Lemma \ref{lem:P20 computes THH}.

\subsection{Sliding dots past crossings, and an explanation for what went wrong}
\label{sec:what went wrong for THH}
Given the perhaps surprising Theorem \ref{thm:THH is not S1S2 invt}, one might wonder where the proof of $S^1\times S^2$-invariance for the corresponding homology groups fails in the spectral setting.  We will see that one problem arises from the homotopies required to `slide dots past crossings', indicating that constructions such as \cite{batson-seed} may be complicated in the spectral setting.  We will only sketch the arguments here; the details are similar to those of Section \ref{sec:Proof of Lemma p20 computes THH}.  

We illustrate the problem by considering the dotted identity maps
\[f: 
q^2\ILtikzpic[xscale=.25,yscale=.6]{
    \draw[thick] (-.5,.8) to[out=90,in=90,looseness=.3] (2.5,.8);
    \drawover{(.2,0)--(.2,1) to[out=90,in=90] (1.8,1) --(1.8,0);}
    \drawover{(-.5,.8) to[out=-90,in=-90,looseness=.3] (2.5,.8);}
    }
\xrightarrow{
\ILtikzpic[xscale=.25,yscale=.6]{
    \draw[thick] (-.5,.8) to[out=90,in=90,looseness=.3] (2.5,.8);
    \drawover{(.2,0)--(.2,1) to[out=90,in=90] node[pos=.7,fill=black,circle,scale=.4]{} (1.8,1) --(1.8,0);}
    \drawover{(-.5,.8) to[out=-90,in=-90,looseness=.3] (2.5,.8);}
    }
}
\ILtikzpic[xscale=.25,yscale=.6]{
    \draw[thick] (-.5,.8) to[out=90,in=90,looseness=.3] (2.5,.8);
    \drawover{(.2,0)--(.2,1) to[out=90,in=90] (1.8,1) --(1.8,0);}
    \drawover{(-.5,.8) to[out=-90,in=-90,looseness=.3] (2.5,.8);}
    },
\]
\[g: 
q^2\ILtikzpic[xscale=.25,yscale=.6]{
    \draw[thick] (-.5,.8) to[out=90,in=90,looseness=.3] (2.5,.8);
    \drawover{(.2,0)--(.2,1) to[out=90,in=90] (1.8,1) --(1.8,0);}
    \drawover{(-.5,.8) to[out=-90,in=-90,looseness=.3] (2.5,.8);}
    }
\xrightarrow{
\ILtikzpic[xscale=.25,yscale=.6]{
    \draw[thick] (-.5,.8) to[out=90,in=90,looseness=.3] (2.5,.8);
    \drawover{(.2,0)--(.2,1) to[out=90,in=90] (1.8,1) to[out=-90,in=90] node[pos=.8,fill=black,circle,scale=.4]{} (1.8,0);}
    \drawover{(-.5,.8) to[out=-90,in=-90,looseness=.3] (2.5,.8);}
    }
}
\ILtikzpic[xscale=.25,yscale=.6]{
    \draw[thick] (-.5,.8) to[out=90,in=90,looseness=.3] (2.5,.8);
    \drawover{(.2,0)--(.2,1) to[out=90,in=90] (1.8,1) --(1.8,0);}
    \drawover{(-.5,.8) to[out=-90,in=-90,looseness=.3] (2.5,.8);}
    },
\]
and
\[h:
q^2\,\ILtikzpic[xscale=.25,yscale=.4]{
    \draw[thick] (.2,0)--(.2,1) to[out=90,in=90] (1.8,1) --(1.8,0);
    \draw[thick] (-.3,1) to[out=90,in=90,looseness=.8] (-1.7,1) to[out=-90,in=-90,looseness=.8] (-.3,1);
    }
\xrightarrow{
\ILtikzpic[xscale=.25,yscale=.4]{
    \draw[thick] (.2,0)--(.2,1) to[out=90,in=90] (1.8,1) to[out=-90,in=90] node[pos=.7,fill=black,circle,scale=.4]{} (1.8,0);
    \draw[thick] (-.3,1) to[out=90,in=90,looseness=.8] (-1.7,1) to[out=-90,in=-90,looseness=.8] (-.3,1);}
    }
\ILtikzpic[xscale=.25,yscale=.4]{
    \draw[thick] (.2,0)--(.2,1) to[out=90,in=90] (1.8,1) --(1.8,0);
    \draw[thick] (-.3,1) to[out=90,in=90,looseness=.8] (-1.7,1) to[out=-90,in=-90,looseness=.8] (-.3,1);
    }.
\]

These maps are interpreted as maps between the corresponding Khovanov spectra, associated to merging on an `$x$'-labeled circle at the location of the dot.  

By definition $f\circ f$, $g\circ g$, and $h\circ h$ are all zero maps (via the empty correspondence in the relevant Burnside category).  In particular, the fact that $f\circ f$ is canonically nullhomotopic (as opposed to $f\circ f$ being simply null-homotopic) is part of the data used to construct the homotopy coherent diagram whose hocolim defines the left-hand side of Equation \eqref{eq:P20 cannot unlink meridian}.  Similarly, the fact $h\circ h=*$ is part of the data used to construct the right-hand side of Equation \eqref{eq:P20 cannot unlink meridian}.

We will be interested in four semi-simplicial objects, denoted $F,G,H,K$.  Here $T$ will denote the dot map on the top turnback, and $k$ stands for the dot map on the rightmost strand at the bottom of $K$.  The maps shown are $d_0$; all other $d_i$ for $i\neq 0$ are basepoint maps.

\begin{equation}\label{eq:f-definition}
F:=\left(
\ILtikzpic[xscale=3]{
\node(dots) at (.25,0) {\scriptsize $\cdots$};
\node(res3) at (1,0) {$q^5\ILfres$};
\node(res2) at (2,0) {$q^3\ILfres$};
\node(res1) at (3,0) {$q\ILfres$};
\draw[->] (res2)--(res1) node[pos=.5,anchor=south]  {\scriptsize $T-f$};
\draw[->] (res3)--(res2) node[pos=.5,anchor=south] {\scriptsize $T+f$};
\draw[->] (dots)--(res3) node[pos=.5,anchor=south] {\scriptsize $T-f$};
}\right).
\end{equation}
\begin{equation}\label{eq:g-definition}
G:=\left(
\ILtikzpic[xscale=3]{
\node(dots) at (.25,0) {\scriptsize $\cdots$};
\node(res3) at (1,0) {$q^5\ILfres$};
\node(res2) at (2,0) {$q^3\ILfres$};
\node(res1) at (3,0) {$q\ILfres$};
\draw[->] (res2)--(res1) node[pos=.5,anchor=south]  {\scriptsize $T-g$};
\draw[->] (res3)--(res2) node[pos=.5,anchor=south] {\scriptsize $T+g$};
\draw[->] (dots)--(res3) node[pos=.5,anchor=south] {\scriptsize $T-g$};
}\right).
\end{equation}
\begin{equation}\label{eq:h-definition}
H:=\left(
\ILtikzpic[xscale=3]{
\node(dots) at (.25,0) {\scriptsize $\cdots$};
\node(res3) at (1,0) {$q^5\ILdisres$};
\node(res2) at (2,0) {$q^3\ILdisres$};
\node(res1) at (3,0) {$q\ILdisres$};
\draw[->] (res2)--(res1) node[pos=.5,anchor=south]  {\scriptsize $T-h$};
\draw[->] (res3)--(res2) node[pos=.5,anchor=south] {\scriptsize $T+h$};
\draw[->] (dots)--(res3) node[pos=.5,anchor=south] {\scriptsize $T-h$};
}\right).
\end{equation}
\begin{equation}\label{eq:k-definition}
K:=\left(
\ILtikzpic[xscale=3.5]{
\node(dots) at (.25,0) {\scriptsize $\cdots$};
\node(res3) at (1,0) {$q^5\ILkres$};
\node(res2) at (2,0) {$q^3\ILkres$};
\node(res1) at (3,0) {$q\ILkres$};
\draw[->] (res2)--(res1) node[pos=.5,anchor=south]  {\scriptsize $T-k$};
\draw[->] (res3)--(res2) node[pos=.5,anchor=south] {\scriptsize $T+k$};
\draw[->] (dots)--(res3) node[pos=.5,anchor=south] {\scriptsize $T-k$};
}\right).
\end{equation}

To see the semisimplicial structure of these objects, we begin with $F$ which is defined as the tensor product of $\ssimpproj_{2,0}$ with $\left(
\vcenter{\hbox{\begin{tikzpicture}[xscale=.25,yscale=.6]
    \draw[thick] (.2,.5)--(.2,0);
    \draw[thick] (1.8,.5)--(1.8,0);
    \draw[thick] (-.5,.6) to[out=90,in=90,looseness=.3] (2.5,.6);
    \drawover{ (2.5,.6) to[out=-90,in=-90,looseness=.3] (-.5,.6);}
    \drawover{ (.2,.5)--(.2,1);}
    \drawover{ (1.8,.5)--(1.8,1);}
\end{tikzpicture}}} \,  \right)$.  We next construct the semisimplicial object $K$ in the same manner as $\ssimpproj_{2,0}$.  That is to say, for any double composite $d_0\circ d_0$ in $\Delta_{inj}^{op}$, we choose a nullhomotopy of $(T\pm k)\circ (T\mp k)$ by a nullhomotopy of $(1-1)\colon \mathbb{S}\to \mathbb{S}$.  Higher composites are then canonically nullhomotopic due to the $q$-degree shifts (note that the underlying mapping spectra in $K$ are equivalent to the underlying mapping spectra in $\ssimpproj_{2,0}$ since the inner boundary of $K$ has only one possible closure).  With $K$ in hand, we define the semisimplicial object $G$ to be the tensor product of $K$ with $\left(\ILtikzpic[xscale=.25,yscale=.6]{
    \draw[thick] (-.5,.8) to[out=90,in=90,looseness=.3] (2.5,.8);
    \drawover{(.2,.2)--(.2,1) to[out=90,in=90] (1.8,1) --(1.8,.2);}
    \drawover{(-.5,.8) to[out=-90,in=-90,looseness=.3] (2.5,.8);}
}\right)$.  Finally, we define $H$ as the tensor product of $\left(\ILtikzpic[xscale=.25,yscale=.4]{
    \draw[thick] (.2,0)--(.2,1) to[out=90,in=90] (1.8,1) --(1.8,0);
    \draw[thick] (-.3,1) to[out=90,in=90,looseness=.8] (-1.7,1) to[out=-90,in=-90,looseness=.8] (-.3,1);
}\right)$ with $K$.  We note that as semisimplicial objects, $F,G,H,K$ all depend on these choices of nullhomotopies (which are not reflected in the notation).   Finally, and most crucially, assuming we have chosen the same nullhomotopies for $K$ that we did for $\ssimpproj_{2,0}$, we have that $H$ is the tensor product of $\ssimpproj_{2,0}$ with a disjoint circle; in particular $H$ is $q\ssimpproj_{2,0}\vee q^{-1}\ssimpproj_{2,0}$.  

We will consider the relationship between $F,G,$ and $H$.  On the one hand, if we view $H$ as a tensor product with $K$, then we see that $G$ and $H$ are equivalent semisimplicial spectra, since the two spectra $\left(\ILtikzpic[xscale=.25,yscale=.6]{
    \draw[thick] (-.5,.8) to[out=90,in=90,looseness=.3] (2.5,.8);
    \drawover{(.2,.2)--(.2,1) to[out=90,in=90] (1.8,1) --(1.8,.2);}
    \drawover{(-.5,.8) to[out=-90,in=-90,looseness=.3] (2.5,.8);}
}\right)$ and $\left(\ILtikzpic[xscale=.25,yscale=.4]{
    \draw[thick] (.2,0)--(.2,1) to[out=90,in=90] (1.8,1) --(1.8,0);
    \draw[thick] (-.3,1) to[out=90,in=90,looseness=.8] (-1.7,1) to[out=-90,in=-90,looseness=.8] (-.3,1);
}\right)$ are equivalent.  On the other hand, if we view $H$ as a tensor product with $\ssimpproj_{2,0}$, then Equation \eqref{eq:P20 cannot unlink meridian} shows that $F$ and $H$ do not have equivalent realizations, and thus $F$ and $G$ cannot have equivalent realizations.  However, the identity map $F([i])\to G([i])$ homotopy commutes with $f$ by Proposition 7.1 of \cite{LLS_func} (and homotopy commutes with $T$ automatically), so that we have the following \emph{homotopy commutative} diagram  where the vertical arrows are equivalences.  (This is not to be interpreted as a diagram of semisimplicial spectra).

\begin{equation}\label{eq:fg}
\ILtikzpic[xscale=3]{
\node(dots) at (.25,0) {\scriptsize $\cdots$};
\node(res3) at (1,0) {$q^5\ILfres$};
\node(res2) at (2,0) {$q^3\ILfres$};
\node(res1) at (3,0) {$q\ILfres$};
\draw[->] (res2)--(res1) node[pos=.5,anchor=south]  {\scriptsize $T-f$};
\draw[->] (res3)--(res2) node[pos=.5,anchor=south] {\scriptsize $T+f$};
\draw[->] (dots)--(res3) node[pos=.5,anchor=south] {\scriptsize $T-f$};

\node(dotsb) at (.25,-2) {\scriptsize $\cdots$};
\node(res3b) at (1,-2) {$q^5\ILfres$};
\node(res2b) at (2,-2) {$q^3\ILfres$};
\node(res1b) at (3,-2) {$q\ILfres$};
\draw[->] (res2b)--(res1b) node[pos=.5,anchor=south]  {\scriptsize $T-g$};
\draw[->] (res3b)--(res2b) node[pos=.5,anchor=south] {\scriptsize $T+g$};
\draw[->] (dotsb)--(res3b) node[pos=.5,anchor=south] {\scriptsize $T-g$};

\draw[->] (res1)--(res1b);
\draw[->] (res2) -- (res2b);
\draw[->] (res3) -- (res3b);
}
\end{equation}

We also have that there is a morphism $F_{\leq 2}\to G_{\leq 2}$, constructed as follows, where we write $F_{\leq 2}$, etc., for the semisimplicial object obtained by restricting to simplices $[k]$ for $k\leq 2$ of $F$ (etc.).  Indeed, the maps defined in \eqref{eq:fg} give a prism:

\begin{equation}\label{eq:prism}
\ILtikzpic[xscale=3]{
\node(res3) at (1,-1.5) {$q^5\ILfres$};
\node(res2) at (2,0) {$q^3\ILfres$};
\node(res1) at (3,0) {$q\ILfres$};
\draw[->] (res2)--(res1) node[pos=.5,anchor=south]  {\scriptsize $T-f$};
\draw[->] (res3)--(res2) node[pos=.5,anchor=south east] {\scriptsize $T+f$};
\draw[->] (res3)--(res1) node[pos=.7,anchor=north west]  {\scriptsize $0$};

\node(res3b) at (1,-4) {$q^5\ILfres$};
\node(res2b) at (2,-2.5) {$q^3\ILfres$};
\node(res1b) at (3,-2.5) {$q\ILfres$};
\draw[->] (res2b)--(res1b) node[pos=.5,anchor=south]  {\scriptsize $T-g$};
\draw[->] (res3b)--(res2b) node[pos=.5,anchor=south east] {\scriptsize $T+g$};
\draw[->] (res3b)--(res1b) node[pos=.5,anchor=north west] {\scriptsize $0$};

\draw[->] (res1)--(res1b);
\draw[->] (res2) -- (res2b);
\draw[->] (res3) -- (res3b);
}
\end{equation}

We have already defined the length $1$-maps in \eqref{eq:prism}.  We wish to extend these to a homotopy-coherent diagram from the prism (the prism `category' is $\cube\times \{0,1,2\}$, where $\{0,1,2\}$ is the category from the totally ordered poset $\{0,1,2\}$) to graded spectra.  The back and left face extensions are defined by Proposition 7.1 of \cite{LLS_func}, although \cite{LLS_func} does not claim that the homotopy type of that filling is well-defined.  The top face is filled by the definition of $F$.  

By the inner Kan condition, the prism may be filled to a homotopy-coherent diagram, and using that the top face is obtained by a tensor product of a nullhomotopy of $(1-1)\colon \mathbb{S}\to\mathbb{S}$ with $fT$, it is possible to be more specific, and fill so that the bottom face is the same homotopy of $(1-1)\colon \mathbb{S}\to\mathbb{S}$, tensored with $gT$.  That is, for one choice of nullhomotopy in $G$, $F_{\leq 2}$ is equivalent to $G_{\leq 2}$ as semisimplicial spectra  (namely, if $F$ is defined from a choice of nullhomotopy in $\ssimpproj_{2,0}$, the correpsonding nullhomotopy in the definition of $K$ will be the one for which \eqref{eq:prism} can be filled).

If we had fixed the nullhomotopy already in the definition of $G$, the 2-dimensional faces of \eqref{eq:prism} would define a map $A_2\colon \Sigma F([2])\to G([0])$; the nonvanishing of the map $A_2$ is the obstruction to building an equivalence $F_{\leq 2}\to G_{\leq 2}$ extending the choices of the length $1$ arrows and \cite{LLS_func} homotopies.  

There are similar prisms for each pair of adjacent nonzero arrows in the definition of $F$; they may all be filled the same way.  In particular, we obtain a partially-defined diagram $A_{F,G}\colon \Delta_{inj,\leq 3}^{op}\times \cube\to \gSp$ using the homotopies constructed in \eqref{eq:prism} and the same prism, going between $F([3])$ and $G([1])$.  To be more specific, $A_{F,G}|_{\Delta_{inj,\leq 3}^{op}\times \{i\}}$ is $F$ for $i=0$ and $G$ for $i=1$.  There may be several choices of $A_{F,G}$ up to homotopy, corresponding to different choices of the homotopies used from Proposition 7.1 of \cite{LLS_func}, we call any $A_{F,G}$ constructed by such dot-sliding homotopies \emph{good}.  There is an obstruction to extending $A_{F,G}$ to a homotopy-coherent diagram (instead of only being partly defined).  We will call this obstruction $A_3$; it is a homotopy class of map from $\Sigma^2 F([3])\to G([0])$.  

As noted above, if there is a morphism $F_{\leq 3}\to G_{\leq 3}$, then it would extend to a morphism $F\to G$ of semisimplicial spectra (strictly, we need to show that subject to the choices made so far, the other length-3 segments of $F,G$ will be equivalent if and only if $F_{\leq 3}\simeq G_{\leq 3}$, which we leave to the reader).  However, Section \ref{sec:Proof of Lemma p20 computes THH} gives that there is no such morphism $F\to G$.  We have then proved:

\begin{corollary}\label{cor:bad-map-from-dots}
The obstruction map $A_3$ defined above is nonzero, for any good diagrams $A_{F,G}$ obtained by Lawson-Lipshitz-Sarkar dot-sliding homotopies.  
 \end{corollary}

We interpret Corollary \ref{cor:bad-map-from-dots} as saying that the dot-sliding maps in spectra have a more complicated interaction with Reidemeister moves than in the chain complex setting (the chain analog of Corollary \ref{cor:bad-map-from-dots} is false).  We think it is an interesting question if there is a simpler proof of Corollary \ref{cor:bad-map-from-dots}; here the proof depends on calculating $\ssimpproj_{2,0}$, but we ask if Corollary \ref{cor:bad-map-from-dots} can be seen more directly.

\appendix

\section{Induced Khovanov bimodules}

The goal of this appendix is to prove Lemma \ref{lem:induction-spectrum}, which gives a strengthening of Lemma \ref{lem:P20 computes THH} (used in the proof of Theorem \ref{thm:THH is not S1S2 invt}). Notation will be as in section \ref{sec:Proof of Lemma p20 computes THH}.  We note that for any $n,m>0$, there is a somewhat tautological morphism  of ring spectra $C\colon \sarc_{2n}\otimes \sarc_{2m}^{op}\to \sarc_{2(n+m)}$ (induced by a morphism of Burnside functors).  Similarly, we recall that a $(2n,2m)$-tangle $T$ may be viewed either as a $(2n,2m)$-tangle, or as a $(0,2(n+m)$-tangle, giving two different Khovanov spectra $\X(T)_{2n,2m}$ or $\X(T)_{2(n+m)}$, where the former is a bimodule, and the latter is a module over $\sarc_{2(n+m)}$, each suitable for different gluing statements (the subscripts not usually being written).  The machinery of \cite{LLS_tangles} implies that $\X(T)_{2n,2m}$ is naturally equivalent to $C^*(\X(T)_{2(n+m)})$.  

\begin{lemma}\label{lem:induction-spectrum}
    Let $(\ssimpBar_{2})'$ denote the $\sarc_2\otimes \sarc_2^{op}$-spectral module associated to the diagram \eqref{eq:bsimp-def} (viewing each tangle diagram as a $(\sarc_2,\sarc_2)$-bimodule, instead of as a $\sarc_4$-module).  There is a homotopy equivalence, of semi-simplicial modules:
    \[
    (\ssimpBar_2)'\otimes_{\sarc_2\otimes\sarc_2^{op}} \sarc_4\simeq \ssimpBar_2.
    \]
\end{lemma}
The lemma says that $\ssimpBar_2$ is induced up from a bimodule (in general, one may not expect that the module invariant of a tangle is determined from its bimodule invariant).  In particular, the lemma allows us to write $\ssimpBar_2$ without the superscript, with the module structure (over $\sarc_4$ or $(\sarc_2\otimes \sarc_2^{op})$) interpreted from context.  We note that the claim analogous to the lemma \emph{for individual terms} in the description \eqref{eq:bsimp-def} follows from the gluing theorem of \cite{LLS_tangles}, since each term can be obtained as a gluing of diagrams in $TL^2_0$ and $TL_0^2$.  However, the pictured diagram of cobordisms is not obtained as a gluing of diagrams in $TL_0^2$ and $TL^2_0$, and so the usual gluing theorem does not apply.  However, it turns out that the ideas of the usual gluing argument go through essentially without change - this is sketched in the following argument.
\begin{proof}
    The general strategy of the proof is along the lines of arguments in \cite{LLS_tangles}; the details happen at the level of Burnside functors, and so we provide only a sketch in order to avoid diving into the details of the construction of the Khovanov Burnside functor (see \cite[Section 2.11]{LLS_tangles}).

    First, write $e_1$ for the $\sarc_4$-module of the turnback $e_1\in \TL_2$, and write $e_1'$ for the associated $(\sarc_2,\sarc_2)$-bimodule.  The first step of the proof is to construct an equivalence of $\sarc_4$-modules:
    \begin{equation}\label{eq:induction-morph}
    I\colon e_1'\otimes_{\sarc_2\otimes \sarc_2^{op}}^{\bL}\sarc_4\to e_1.
    \end{equation}
    The second step is to show that the morphisms in the definition \eqref{eq:bsimp-def} are compatible with the morphism $I$ in \eqref{eq:induction-morph}.

    To construct $I$, we follow the strategy of \cite[Section 5]{LLS_tangles} quite closely.  First, we define a \emph{gluing multicategory} $\mathcal{U}^0$.  The objects are:
    \begin{enumerate}[label=(U-\arabic*)]
        \item \label{itm:match-1} Pairs $(b_1,b_2)$ of pairs of crossingless matchings on $2$ points.  More explicitly, let $e_1\in TL_2$ intersect the upper half of the boundary circle in $p_1,p_1'$ and the lower half of the boundary circle in $q_1,q_1'$.  Then $b_1=b_2$ will each be the pair of crossingless matchings $((p_1,p_1'),(q_1,q_1'))$.
        \item \label{itm:match-2} Pairs $(c_1,c_2)$ of crossingless matchings of the points $\{p_1,p_1',q_1,q_1'\}$.
        \item Pairs $(T_1,b)$, where $T_1$ is a placeholder, and $b$ is a pair of crossingless matchings as in Item \ref{itm:match-1}. \label{itm:match-3}
        \item Triples $(b,T_2,c)$, where $b$ is a pair of crossingless matchings of two points as in the previous items, $T_2$ is interpreted as a placeholder, and $c$ is a crossingless matching of $4$ points $\{p_1,p_1',q_1,q_1'\}$.  \label{itm:match-4}
        \item Pairs $(T_1T_2,c)$ where $c$ is a crossingless matching of $4$ points $\{p_1,p_1',q_1,q_1'\}$. \label{itm:match-5} 
    \end{enumerate}
    There is a unique multimorphism for any string
    \[
    (T_1,b_1),(b_1,b_2),\ldots,(b_{j-1},b_j),(b_j,T_2,c_1),(c_1,c_2),\ldots,(c_{k-1},c_k)\to (T_1T_2,c_k).
    \]
    (Some of the substrings may not occur, e.g. $k=1$ in the above formula is allowed).  These multimorphisms are called the \emph{basic multimorphisms} of the \emph{thickened gluing multicategory} (enriched in groupoids) $\mathcal{U}$, which has the same objects as $\mathcal{U}^0$.  Following the discussion around \cite[Definition 3.6]{LLS_tangles}, an object of the multimorphism groupoid of $\mathcal{U}$ is a tree of basic multimorphisms of $\mathcal{U}^0$, composable in the natural sense. 

    Lemma 5.2 of \cite{LLS_tangles} continues to hold in our context.  Said very loosely, there is a multifunctor $\underline{G}\colon \mathcal{U}\to \widetilde{\underline{\mathrm{Cob}}}_d$ (the target being a version of the cobordism category of surfaces in $\mathbb{R}^3$, with extra data), that is compatible with the functors to cobordism categories used in the construction of $\X(2)$ and $\X(4)$.  Applying the Khovanov-Burnside functor to $\underline{G}$, we have a functor $G_{\mathrm{Burn}}\colon \mathcal{U}\to \gBurn$.  In turn, postcomposing with Elmendorff-Mandell $K$-Theory and strictifying (see \cite[Section 4.1]{LLS_tangles}), we obtain what we will call the \emph{gluing functor}
    \[
    Gl\colon \mathcal{U}^0\to \underline{\gSp}.
    \]
The underline on $\Sp$ denotes the \emph{multicategory} of (graded) spectra, which has the same objects and $1$-input morphisms as the category of (graded) spectra; $Gl$ is a multifunctor between multicategories.  To be slightly more specific, $Gl$ restricted to the subcategory from objects of types \ref{itm:match-1} and \ref{itm:match-3} is exactly the functor to spectra used to define $e_1'$, as a bimodule over $\sarc_2$.  Similarly, $Gl$ restricted to items \ref{itm:match-2}-\ref{itm:match-4} is the multifunctor defining $\sarc_4$, \emph{as a $(\sarc_2\otimes \sarc_2^{op},\sarc_4)$-bimodule}.

Part of the data of $Gl$ is, for each multimorphism $(T_1,b),(b,T_2,c)\to (T_1T_2,c)$,  a map
\[
Gl((T_1,b),(b,T_2,c)\to (T_1T_2,c))\colon Gl(T_1,b)\wedge Gl(b,T_2,c)\to Gl((T_1T_2,c)).
\]
These in fact reduce to a map of right $\sarc_4$-modules (cf. \cite[Lemma 5.4]{LLS_tangles}):
\begin{equation}\label{eq:i-def}
I\colon e_1'\otimes_{\sarc_2\otimes\sarc_2^{op}}^{\bL}\sarc_4\to e_1.
\end{equation}
An argument taking place at the level of chains then suffices to show that $I$ is an equivalence.  Indeed, the constructions above also produce a map between Khovanov chain complexes; this map on Khovanov chain complexes is quasi-isomorphic to $\mathcal{C}_h(I)$, the chains functor $\mathcal{C}_h$ applied to the functor $I$.  One can directly check that the Khovanov chain complex analog of $I$ in \eqref{eq:i-def} is a homotopy equivalence, giving that $I$ is also an equivalence by Whitehead's theorem.  

Indeed, we note that the entire argument, up to \eqref{eq:i-def}, did not use anything about the tangle $e_1\in TL_2$; we only need to know the form of $e_1\in TL_2$ in order for the map in \eqref{eq:i-def} to be an equivalence.

To show that $\ssimpBar_2\simeq (\ssimpBar_2)'$ (instead of just an agreement at each object in the description of these two spectra), one repeats the same argument, but replacing the gluing multicategory $\mathcal{U}$ with a larger gluing multicategory, encoding the semi-simplicial structure maps in the description \eqref{eq:bsimp-def}; this category is completely analogous to the category $\underline{2}^{N_1\mid N_2}\tilde{\times} \mathcal{U}$ in \cite{LLS_tangles}.  The result will be a morphism $J\colon (\ssimpBar_2)'\otimes_{\sarc_2\otimes\sarc_2^{op}}\sarc_4\to \ssimpBar_2$, as semi-simplicial objects, and by the preceding argument, $J$ is a levelwise homotopy equivalence (Namely, on each level, a copy of $I$).  This completes the proof.

\end{proof}

Recall that we are interested in the description of THH given in \eqref{eq:thh}; here we complete the proof of Lemma \ref{lem:P20 computes THH} from Lemma \ref{lem:induction-spectrum}.

Because $\sarc_2$ is the same as the Khovanov spectrum of a disjoint circle, with multiplication in $\sarc_2$ induced by saddle maps, we can identify, for any $(2,2)$-tangle $T$:
\begin{equation}\label{eq:THH via Bar}
\THH\left(\sarc_2;\X\left(
\vcenter{\hbox{\begin{tikzpicture}[xscale=.25,yscale=.8]
    \draw[thick] (.2,1)--(.2,0);
    \draw[thick] (1.8,1)--(1.8,0);
    \whitebox{-.2,.15}{2.2,.85}
    \node[scale=.8] at (1,.5) {$T$};
\end{tikzpicture}}} \right) \right)
\simeq
\left|\ssimpBar_2\otimes_{\sarc_2\otimes\sarc_2^{op}}\X\left(
\vcenter{\hbox{\begin{tikzpicture}[xscale=.25,yscale=.8]
    \draw[thick] (.2,1)--(.2,0);
    \draw[thick] (1.8,1)--(1.8,0);
    \whitebox{-.2,.15}{2.2,.85}
    \node[scale=.8] at (1,.5) {$T$};
\end{tikzpicture}}} \right)  \right|
\end{equation}
The right-hand side denotes the realization of the semi-simplicial spectrum specified by the tensor product on the right.  
By associativity of the tensor product, together with Lemma \ref{lem:induction-spectrum}, we have
\[
\left|\ssimpBar_2\otimes_{\sarc_4}^{\bL} \X\left(
\vcenter{\hbox{\begin{tikzpicture}[xscale=.25,yscale=.8]
    \draw[thick] (.2,1)--(.2,0);
    \draw[thick] (1.8,1)--(1.8,0);
    \whitebox{-.2,.15}{2.2,.85}
    \node[scale=.8] at (1,.5) {$T$};
\end{tikzpicture}}} \right)  \right|\simeq \left|(\ssimpBar_2\otimes_{\sarc_2\otimes\sarc_2^{op}}\sarc_4)\otimes_{\sarc_4}^{\bL} \X\left(
\vcenter{\hbox{\begin{tikzpicture}[xscale=.25,yscale=.8]
    \draw[thick] (.2,1)--(.2,0);
    \draw[thick] (1.8,1)--(1.8,0);
    \whitebox{-.2,.15}{2.2,.85}
    \node[scale=.8] at (1,.5) {$T$};
\end{tikzpicture}}} \right)  \right|\simeq \left|\ssimpBar_2\otimes_{\sarc_2\otimes\sarc_2^{op}}  \X\left(
\vcenter{\hbox{\begin{tikzpicture}[xscale=.25,yscale=.8]
    \draw[thick] (.2,1)--(.2,0);
    \draw[thick] (1.8,1)--(1.8,0);
    \whitebox{-.2,.15}{2.2,.85}
    \node[scale=.8] at (1,.5) {$T$};
\end{tikzpicture}}} \right)  \right|.
\]
It is interesting to note that the above equivalences are completely tautological in (for example) the setting of \cite{Hog_polyaction}, where computations directly with tangle diagrams are allowed.

Moreover, using that homotopy colimits commute with the monoidal structure:
\[
\left|\ssimpBar_2\otimes_{\sarc_4}^{\bL} \X\left(
\vcenter{\hbox{\begin{tikzpicture}[xscale=.25,yscale=.8]
    \draw[thick] (.2,1)--(.2,0);
    \draw[thick] (1.8,1)--(1.8,0);
    \whitebox{-.2,.15}{2.2,.85}
    \node[scale=.8] at (1,.5) {$T$};
\end{tikzpicture}}} \right)  \right|
\simeq
\left|\ssimpBar_2\right|\otimes_{\sarc_4}^{\bL} \X\left(
\vcenter{\hbox{\begin{tikzpicture}[xscale=.25,yscale=.8]
    \draw[thick] (.2,1)--(.2,0);
    \draw[thick] (1.8,1)--(1.8,0);
    \whitebox{-.2,.15}{2.2,.85}
    \node[scale=.8] at (1,.5) {$T$};
\end{tikzpicture}}} \right)  .
\]

This completes the alternative proof of Lemma \ref{lem:P20 computes THH}.
\bibliographystyle{alpha}
\bibliography{references}

\end{document}